\PassOptionsToPackage{czech,english}{babel}
\documentclass[a4paper, 12pt]{article} 
\usepackage{textcomp}
\usepackage{amsmath}
\usepackage{authblk}
\usepackage{mathtools}
\usepackage{amssymb}
\usepackage{amsthm}
\usepackage[ruled,vlined]{algorithm2e}
\usepackage{setspace}
\usepackage{mleftright}
\usepackage{multirow}
\usepackage[figuresright]{rotating}
\usepackage[utf8]{inputenc}
\usepackage[english]{babel}
\usepackage{filecontents}
\usepackage{graphicx}
\usepackage{titletoc}
\usepackage[toc,page]{appendix}
\graphicspath{ {figures/} }
\usepackage{array}
\usepackage{float}
\usepackage{scalerel,stackengine}
\usepackage{accents}
\usepackage{color, graphics, graphicx}
\usepackage{xcolor,colortbl}
\usepackage{undertilde}
\usepackage[T1]{fontenc}
\usepackage{lmodern}
\usepackage[babel=true]{microtype}
\usepackage[czech]{babel}
\usepackage{hyperref}
\usepackage{paralist}

\newtheorem{theorem}{Theorem}[section]
\newtheorem{proposition}{Proposition}[section]
\newtheorem{corollary}{Corollary}[section]
\newtheorem{lemma}[theorem]{Lemma}

\numberwithin{equation}{section}
\numberwithin{theorem}{section}

\newcommand{\mbf}[1]{\mbox{\boldmath $#1$}}
\newcommand{\bth}{{\mbf \theta}}

\newcommand{\al}{{\mbf \alpha}}

\newcommand{\bu}{{\mbf u}}

\newcommand{\btau}{{\mbf \tau}}

\DeclareMathAlphabet\mathbfcal{OMS}{cmsy}{b}{n}
\def\s{\sum_{t=1}^n}

\def\bth{\mbox{\boldmath$\theta$}}

\def\bLam{\mbox{\boldmath$\Lambda$}}

\def\bDelta{\mbox{\boldmath$\Delta$}}
\def\bSigma{\mbox{\boldmath$\Sigma$}}
\def\bXi{\mbox{\boldmath$\Xi$}}
\def\bmu{\mbox{\boldmath$\mu$}}
\def\bxi{\mbox{\boldmath$\xi$}}
\def\bpsi{\mbox{\boldmath$\psi$}}
\def\bGamma{\mbox{\boldmath$\Gamma$}}
\def\bOmega{\mbox{\boldmath$\Omega$}}
\def\bUpsilon{\mbox{\boldmath$\Upsilon$}}
\def\bvp{\mbox{\boldmath$\varphi$}}
\def\bepsilon{\mbox{\boldmath$\epsilon$}}

\def\R{{\mathbb R}}
\def\Z{{\mathbb Z}}
\def\C{{\mathbb C}}
\def\D{{\mbf D}}
\def\N{{\mathbb N}}

\def\K{{\mbf K}}

\def\J{{\mbf J}}

\def\M{{\mbf M}}
\def\bz{{\mbf z}}

\def\0{{\mbf 0}}
\def\W{{\mbf W}}
\def\P{{\mbf P}}
\def\Q{{\mbf Q}}
\def\S{{\mbf S}}

\def\T{{\mbf T}}

\def\A{{\mbf A}}
\def\B{{\mbf B}}
\def\CC{{\mbf C}}
\def\DD{{\mbf D}}
\def\F{{\mbf F}}
\def\H{{\mbf H}}
\def\hbf{{\mbf h}}
\def\G{{\mbf G}}
\def\I{{\mbf I}}
\def\X{{\mbf X}}
\def\x{{\mbf x}}
\def\y{{\mbf y}}
\def\RR{{\mbf R}}
\def\ZZ{{\mbf Z}}

\def\0{{\mbf 0}}

\def\z{{\mbf z}}
\def\e{{\mbf e}}

\def\n{^{(n)}}

\def\pms{\mspace{-1mu}{\scriptscriptstyle\pm}}
\def\Fstar{F^\star_{d;\mathfrak{f}}}

\newcommand{\cqfd}{\ensuremath{\hfill\Box}}

\usepackage{stackengine}
\stackMath
\newcommand\tenq[2][1]{%
\def\useanchorwidth{T}%
\ifnum#1>1%
\stackunder[0pt]{\tenq[\numexpr#1-1\relax]{#2}}{\scriptscriptstyle\thicksim}%
\else%
\stackunder[1pt]{#2}{\scriptscriptstyle\thicksim}%
\fi%
}

\title{ \Large \bf \sc Center-Outward R-Estimation for  \\ Semiparametric VARMA Models}

\author{M. Hallin, D. La Vecchia,  and H. Liu    \\
{\small{ ECARES, Université libre de Bruxelles CP 114/4 \\Avenue F.D. Roosevelt 50 - 
B-1050 Bruxelles, Belgium \\
 Email: mhallin@ulb.ac.be \\ \vspace{0.1cm}
 
 Research Center for Statistics, University of Geneva \\ Boulevard du Pont d'Arve 40 - CH-1211 Geneva, Switzerland \\
 Email: davide.lavecchia@unige.ch \\ \vspace{0.1cm}

 Department of Mathematics and Statistics, Lancaster University \\ LA1 4YF Lancaster, UK \\
 Email: h.liu11@lancaster.ac.uk 
 
   } }}
\date{}

\begin{document} 
\maketitle


\begin{abstract}
{\small{We propose a new class of R-estimators for semiparametric VARMA models in which the innovation density plays the role of the nuisance parameter. Our estimators are based on the novel concepts of multivariate center-outward ranks and signs.  We show that these concepts,  combined with Le Cam's asymptotic  theory of statistical experiments, yield  a class of semiparametric estimation procedures, which are   efficient (at a given reference density), 
root-$n$ consistent, and asymptotically normal under a broad class of (possibly non elliptical) actual innovation densities.  No kernel density estimation is required to implement our procedures. A Monte Carlo   comparative study of our R-estimators and other routinely-applied competitors  demonstrates the benefits of the novel methodology, in large and small sample.   Proofs, computational aspects, and further numerical results are available in the supplementary material.}}
\end{abstract}




\section{Introduction}
\subsection{Quasi-maximum likelihood and R-estimation}
Gaussian {\it quasi}-likelihood methods are pervasive in several areas of statistics. Among them is  time series analysis,  univariate and multivariate, linear and non-linear. 
In particular, quasi-maximum 
likelihood estimation (QMLE))  and  correlogram-based testing are the daily practice golden standard for ARMA and VARMA models. They only require the specification of 
 the first two conditional moments, which depend on an unknown Euclidean parameter, while a Gaussian (misspecified) innovation density  is assumed.  Their properties  
 are generally considered as fully satisfactory: QMLEs, in particular, are   root-$n$  consistent, parametrically efficient under Gaussian innovations, and asymptotically normal under finite fourth-order moment assumptions. 

Despite  their popularity, QMLE methods are not without some undesirable consequences, though, which are often overlooked: 
{\it (i)} while achieving efficiency under Gaussian innovations, their asymptotic performance can be quite poor under non-Gaussian ones;  
{\it (ii)} due to technical reasons (the {\it Fisher consistency} requirement), the choice of a quasi-likelihood is always the most pessimistic one: quasi-likelihoods automatically are based on the  {least favorable} innovation density  (here, a Gaussian one); {\it (iii)} root-$n$ consistency is far   from  being  uniform   across innovation densities; {\it (iv)} actual fourth-order moments may be infinite.  

In principle, the ultimate theoretical remedy to those problems is the semiparametric
estimation method described in the monograph by Bickel et al. (1993), which yields uniformly, locally and asymptotically, semiparametrically efficient estimators. For VARMA  models, the semiparametric approach does not specify the innovation density (an infinite-dimensional nuisance) and the estimators based on Bickel et al. (1993) methodology are uniformly, locally and asymptotically {\it parametrically} efficient (VARMA models are {\it adaptive}, thus semiparametric and parametric efficiency  coincide). However, semiparametric  estimation procedures are not easily implemented, since they rely on kernel-based estimation of the actual innovation density (hence the choice of a kernel, the selection of a bandwidth) and the use of sample splitting techniques. All these niceties require relatively large samples and are hard to put into practice even for univariate time series. 

A more flexible and computationally less heavy alternative in the presence of unspecified noise or  innovation densities is   R-estimation, which reaches efficiency at  some chosen reference density  (not necessarily Gaussian or least favorable) or class of densities. R-estimation has been proposed first in the context of location (Hodges and Lehmann~1956) and regression models with independent observations  
(Jure\v{c}kov\' a~1971, Koul~1971, van Eeden~1972, Jaeckel~1972).   Later on, it was extended to autoregressive time series   (Koul and Saleh 1993, Koul and Ossiander~1994, Terpstra et al.~2001, Hettmansperger and McKean~2008, Mukherjee and Bai~2002, Andrews~2008, 2012) and  non-linear  time series (Mukherjee~2007, Andreou and Werker~2015,  Hallin and La Vecchia~2017, 2019). 

Multivariate extensions of these approaches, however, run into the major difficulty of defining an adequate concept of ranks in the multivariate context. This is most regrettable, as the drawbacks of quasi-likelihood methods for  observations in dimension $d=1$ only get worse as the dimension~$d$  increases (see Section~\ref{subsec.glance} for a numerical example in dimension~$d=2$) while  the use of the semiparametric method of Bickel et al.\ becomes  problematic:  the higher the dimension, the more delicate multivariate kernel density estimation and the larger the required sample size. A natural question is thus: ``Can R-estimation palliate the drawbacks of the QMLE and the Bickel et al.\ technique in dimension~$d\geq 2$ the way it does in dimension~$d=1$?''    This question immediately comes up against another one: ``What are ranks and signs, hence, what is  R-estimation, in dimension~$d\geq 2$?'' Indeed, starting with dimension two, the real space~$\mathbb{R}^d$ is no longer canonically ordered.

The main contribution of this paper is to provide a positive answer to these questions.  To this end, we propose a multivariate version of  R-estimation,  establish its asymptotic properties (root-$n$ consistency and asymptotic normality), and  demonstrate its feasibility and excellent finite-sample performance in the context of  semiparametric VARMA models. Our approach  builds on Chernozhukov et al.~(2017), Hallin~(2017), and Hallin et al.~(2020a), who introduce novel concepts of {\it center-outward ranks and signs} based on measure transportation ideas.  These center-outward ranks and signs (see Section~3.2 for details) enjoy all the properties 
 that make traditional univariate ranks a successful tool of inference. In particular, they are distribution-free (see Hallin et al.~(2020) for details), thus 
 preserve the validity of rank-based procedures irrespective of the possible misspecification of the innovation density. 
 Moreover, they are invariant with respect to shift and global scale factors and equivariant under orthogonal transformations; see Hallin et al.\ (2020b). 
Extensive numerical exercises reveal the  finite-sample superiority of our R-estimators over the conventional QMLE in the presence of asymmetric innovation densities (skew-normal, skew-$t$, Gaussian mixtures) and in the presence of outliers. All these advantages, however,  do not come at the cost of a loss of efficiency under symmetry.
 
Other notions of multivariate ranks and signs have been proposed in the statistical literature. Among them,  the componentwise ranks (Puri and Sen~1971), the spatial ranks (Oja~2010), 
the depth-based ranks (Liu~1992; Liu and Singh~1993), and the {Mahalanobis ranks and signs} (Hallin and Paindaveine~(2002a)). 
Those ranks and signs all have their own merits but also some drawbacks, which make them unsuitable for our needs (essentially, they are not  distribution-free, or not maximally so); we refer to 
the introduction of Hallin et al.~(2020) for details. {The Mahalanobis ranks and signs} have been successfully considered for testing purposes  in the time series context  (Hallin and Paindaveine~2002b,  2004). However, no results on estimation are available, and their distribution-freeness property  is limited to elliptical densities---a very strong symmetry assumption which we are dropping here.  

Leaving aside Wasserstein-distance-based methods, our contribution  constitutes  the first  inferential application of measure transportation ideas to semiparametric inference for multivariate time series. Measure transportation, which goes back to Gaspard Monge (1746-1818) and his 1781 {\it M\' emoire sur la Th\' erorie des D\' eblais et des Remblais}, in the past few years has become one of the most active and fertile subjects in pure and applied contemporary  mathematics. Despite some crucial forerunning contributions (Cuesta-Albertos and Matr\' an~(1997); Rachev and R\" uschendorf~(1998)), statistics was somewhat slower to join. However, some recent papers on multiple-output quantile regression (Carlier et al.~2016), distribution-free tests of vector independence and multivariate goodness-of-fit  (Boeckel et al.~(2018); Deb and Sen~(2019); Shi et al.~(2019); Ghosal and Sen~(2019)) demonstrate the growing interest of the statistical community in measure transportation results. We refer to Panaretos and Zemel~(2019) for a review. 

\subsection{A motivating example}\label{subsec.glance}

As a justification of the practical interest of our R-estimation, let us consider 
the very simple but highly representative   motivating example of a bivariate VAR(1) model 
\begin{equation}\label{VARex}\left(\I_d - \A L \right) \X_t =  \bepsilon_t, \quad t \in \Z,
\end{equation}
with  parameter   vec$(\A)=:(a_{11}, a_{21}, a_{12}, a_{22})^\prime$ taking the  
value $(0.2, -0.6, 0.3, 1.1)^\prime$. We generated~300 replications of a realization of length $n=1000$ of the stationary solution of \eqref{VARex} with two innovation densities---a spherical Gaussian one and a Gaussian mixture   (see (\ref{Eq. Mixture}) for details)---which both satisfy the   conditions for QMLE validity. The resulting boxplots of the QMLE and the Gaussian score (van der Waerden)  R-estimator (see Section \ref{Sec: examples} for a definition) are shown in Figure~\ref{box3Mix}, along with the mean squared error (MSE) ratios of the QMLE over the R-estimator. 

Even a vary rapid inspection of the plots reveals that, under the mixture distribution, the R-estimator yields sizeably smaller MSE values than the QMLE. For instance, as far as the estimation of $a_{11}$ is concerned, the MSE ratio is 2.657: the R-estimator is  strikingly less dispersed  than the QMLE. 
On the other hand, under  Gaussian innovations (hence, with the QMLE coinciding with the  MLE and achieving parametric efficiency),   the QMLE and the R-estimator perform similarly, with MSE ratios extremely  close to one for all the  parameters. 
While our R-estimator quite significantly outperforms  the QMLE under the mixture distribution, thus, this benefit comes  at no cost  
under  Gaussian innovations. Further numerical results are provided in Section~\ref{Sec: MC} and Appendix~\ref{Supp.sim};  they all lead to the same conclusion.

\begin{figure}[htbp]
\caption{Boxplots of the QMLE and the R-estimator (van der Waerden) of the para\-meters~$a_{11}, a_{21}, a_{12}$, and $ a_{22}$ of the bivariate VAR(1) \eqref{VARex}  under the Gaussian mixture  \eqref{Eq. Mixture} (upper panel) and spherical Gaussian (lower panel) innovation densities, respectively (300 replications of length $n=1000$). In each panel, the MSE ratio of the QMLE with respect to the R-estimator is reported. The horizontal line represents the actual parameter value.}
\hspace{0.5cm}
\includegraphics[width=1\textwidth, height=0.7\textwidth]{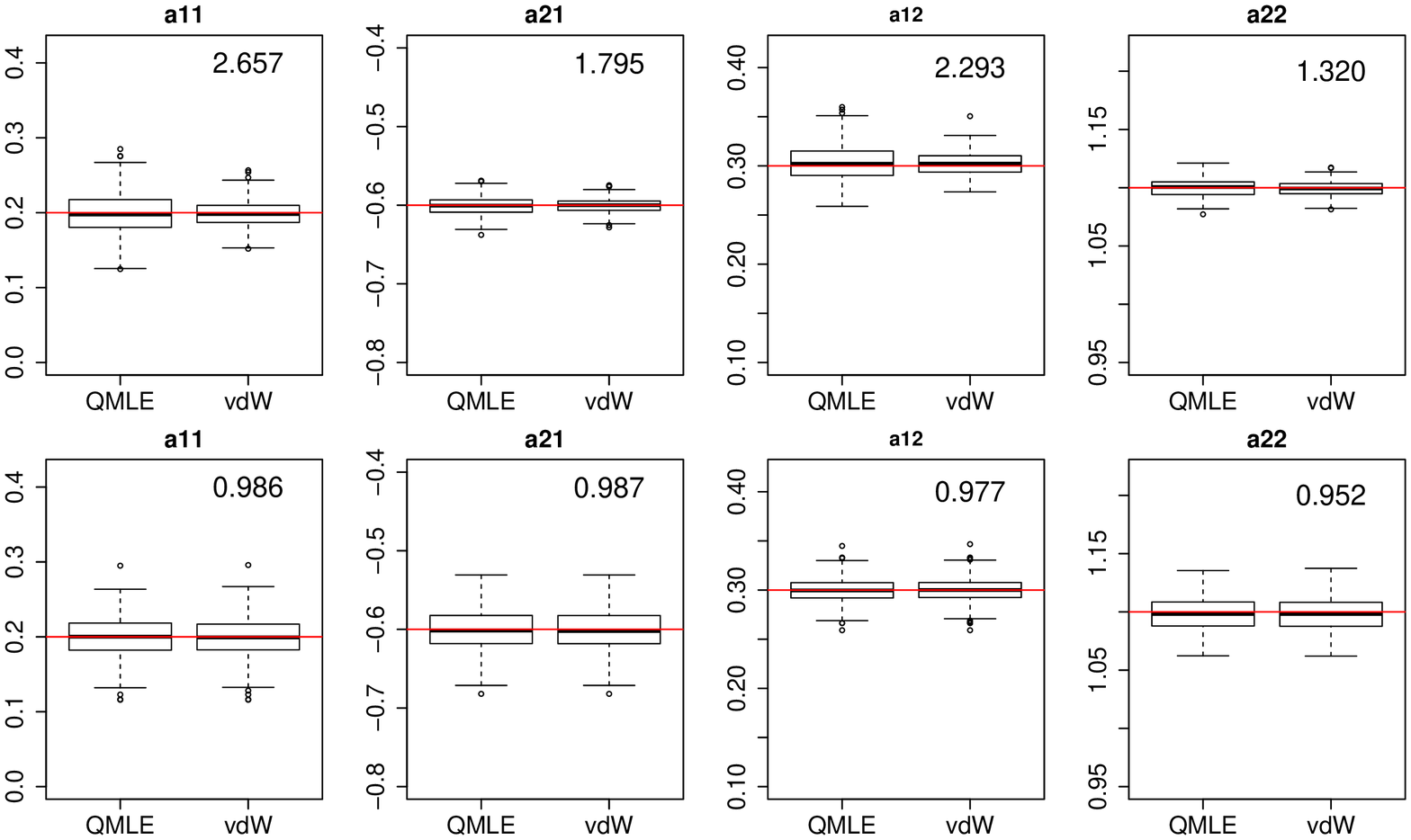}
\label{box3Mix}
\end{figure}\vspace{-3mm}

\subsection{Outline of the paper}

The rest of the paper is organized as follows. Section~\ref{seclan} briefly recalls a local asymptotic normality result for the VARMA model with  
nonelliptical innovation density: an analytical form of the central sequence  as a function of the residuals is provided, which  indeed plays a key role in the construction of our estimators. 
In Section~\ref{secranks}, 
 we introduce the measure transportation-based notions of center-outward ranks and signs; for the sake of analogy, we also recall the definition of Mahalanobis ranks and signs, and shortly discuss their respective invariance properties.  In Section \ref{SecSR}, we explain the key idea of our construction of R-estimators, which consists in replacing the residuals 
appearing in central sequence with some adequate function of their center-outward ranks and  signs, yielding a rank-based version of the latter: our R-estimators are obtained by incorporating that rank-based  central sequence into a classical Le Cam one-step procedure. Root-$n$ consistency and asymptotic normality are established in   Proposition~\ref{asy}  under  absolutely continuous innovation densities  admitting finite second moment.  
Some standard score functions are discussed in Section~\ref{Sec: examples}. Section~\ref{Sec: MC} presents simulation results under various densities of the various estimators; comparing their performance confirms the findings of the motivating example of Section~\ref{subsec.glance}.  In Section~\ref{empirsec}, we show how  our R-estimation method applies to a real dataset  borrowed from econometrics, where a VARMA($3, 1$) model is identified. Finally, Section~\ref{Sec:concl} concludes and provides some perspectives for future research.

All proofs are concentrated in  Appendices~\ref{AppLAN} and \ref{Tech.proofs}. Sections~\ref{seclan} and~\ref{secranks} are technical and can be skipped at  first reading: the applied statistician can focus directly on the description of 
one-step R-estimation in (\ref{onestep.def}) (implementation details are provided in Appendix~\ref{secalg}) and  the numerical results of Sections~\ref{Sec: MC} and~\ref{empirsec} (Appendix~\ref{Supp.sim}).

\section{Local asymptotic normality}\label{seclan}

Local asymptotic normality (LAN) is an essential ingredient in the construction of our estimators and the derivation of their asymptotic properties. 
In this section, referring to results by Garel and Hallin~(1995) and Hallin and Paindaveine (2004),  we state, along with the required assumptions, the LAN property for stationary VARMA models, with an explicit expression for the central sequence to be used later on.  
The corresponding technical material is available in Appendices~\ref{AppLAN} and \ref{Tech.proofs}. \vspace{-2mm}

\subsection{Notation and assumptions}\label{secdef}

We throughout consider the $d$-dimensional VARMA($p, q$) model
\begin{equation}
\Big(\I_d - \sum_{i = 1}^p \A_i L^i \Big) \X_t = \Big(\I_d + \sum_{j = 1}^q \B_j L^j\Big) \bepsilon_t, \quad t \in \Z, \vspace{-2mm}\label{VARMA_mod}
\end{equation}
where $\A_1, \ldots , \A_p, \B_1, \ldots , \B_q$ are $d \times d$ matrices, $L$ is the lag operator, and $\{\bepsilon_t; t \in \Z\}$ is an i.i.d.\  mean-zero innovation process  with density $f$. The observed series is $\{\X^{(n)}_1, \ldots , \X^{(n)}_n\}$ (superscript\! $^{(n)}\!$  omitted whenever possible) and the $(p+q)d^2$-dimensional parameter of interest~is
$$\bth := \big((\text{vec}{\A_1})^\prime, \ldots , (\text{vec}{\A_p})^\prime, (\text{vec}{\B_1})^\prime, \ldots , (\text{vec}{\B_q})^\prime\big)^\prime,$$  
 where $^\prime$ indicates 
  transposition. Letting~$\A(L) := \I_d - \sum_{i = 1}^p \A_i L^i$, and~$\B(L) := \I_d+ \sum_{j = 1}^q \B_j L^j$,   the following   conditions
 are assumed to hold. 
 \smallskip
 
\noindent\textbf{Assumption (A1)}.  
{\it (i)} All solutions of the determinantal equations
$$\text{det$\left(\I_d - \sum_{i = 1}^p \A_i z^i \right) = 0\quad$ and \quad det$\left( \I_d + \sum_{i = 1}^q \B_i z^i \right) = 0$, \quad$z \in \C$}$$
 lie outside the unit ball in $\C$;
 {\it (ii)}  $|\A_p| \neq 0 \neq |\B_q|$; 
 {\it (iii)}   $\I_d$ is  the greatest
common left divisor of~$\I_d - \sum_{i = 1}^p \A_i z^i$ and $\I_d + \sum_{i = 1}^q \B_i z^i$. \medskip 

{Assumption (A1) is standard in the time series literature; the restrictions it imposes on the model parameter  ensure the asymptotic stationarity of any  
solution to (\ref{VARMA_mod}).} 

To proceed further,  we assume that the innovation density~$f$ is non-vanishing over~$\R^d$.  
More precisely we  assume that, for all $c \in \R^+$, there exist constants $b_{c; f}$ and $a_{c; f}$ in $\mathbb{R}$ such that~$0<b_{c; f}  \leq a_{c; f}<\infty$ and $ b_{c; f} \leq f(\x) \leq a_{c; f}$
 for~$\Vert \x\Vert  \leq c$: 
denote by ${\cal F}_d$ this family of densities. 

\noindent\textbf{Assumption (A2).} The innovation density $f\in{\cal F}_d$ is such that 
{\it (i)}   $\int \x f(\x) d\mu = \0$ and the covariance  $ \bXi:=\int \x \x^\prime f(\x) d\mu$ is positive definite; 
{\it (ii)} there exists a square-integrable random vector $\DD f^{1/2}$ such that, for all sequence~$\hbf\in\mathbb{R}^d$ such that  $\0 \neq \hbf \rightarrow \0$,
$$(\hbf^\prime \hbf)^{-1} \int [f^{1/2}(\x + \hbf) - f^{1/2}(\x) - \hbf^\prime \DD  f^{1/2} (\x)]^2 d\mu \rightarrow 0,$$
i.e., $f^{1/2}$ is mean-square differentiable, with mean square gradient $\DD f^{1/2}$; 
{\it (iii)} letting
\begin{equation}\label{bvpf}
\bvp_{f}(\x) := ({\varphi}_1(\x), \ldots , {\varphi}_d(\x))^\prime := -2 (\DD f^{1/2})/f^{1/2},
\end{equation}
 $\int {\varphi}_i^4 (\x) f(\x) d\mu < \infty,\ i = 1, \ldots , d$;
{\it (iv)} the score function $\bvp_f$ is piecewise Lipschitz, i.e., there exists a finite measurable partition of $\R^d$ into $J$ non-overlapping subsets $I_j$, $j = 1, \ldots , J$ and a constant $K<\infty$ such that~$\Vert \bvp_f(\x) - \bvp_f(\y)\Vert  \leq K \Vert \x - \y\Vert $ for all $\x, \y$ in $I_j$, $j = 1, \ldots , J$.\medskip

Assumption (A2)(i) requires the existence of the second moment of the innovations (a necessary condition for finite VARMA  Fisher information). 
(A2)(ii) is a multivariate version
of the classical one-dimensional quadratic mean differentiability assumption on $f^{1/2}$.  
Together,  (A2)(i) and (A2)(iii) imply the existence and finiteness of the   Fisher information matrix for location 
$\mathbfcal{I}(f) = \int \bvp_f(\x) \bvp_f^\prime(\x) f(\x) d\mu$ appearing in Proposition 2.1 below. 
See Garel and Hallin (1995) for further discussion. \color{black}

Let $\ZZ_1^{(n)}(\bth),\ldots,\ZZ_n^{(n)}(\bth)$ denote the residuals computed  from the initial values $\bepsilon_{-q+1}, \ldots , \bepsilon_0$ and~$\X_{-p + 1}, \ldots , \X_{0}$, the parameter value $\bth$, and the observations $\X^{(n)}:=\big(\X_{1}, \ldots , \X_{n}\big)$; those residuals can be computed recursively, or from \eqref{Zt.recursive}.  
  Clearly, $\X^{(n)}$ is the finite realization of a solution of~\eqref{VARMA_mod} with parameter value $\bth$  iff $\ZZ^{(n)}_1(\bth), \ldots , \ZZ^{(n)}_n(\bth)$ and~$\bepsilon_1, \ldots , \bepsilon _n$ coincide. Denoting by ${\rm P}^{(n)}_{\bth ;f}$ the distribution of $\X^{(n)}$ under parameter value $\bth$ and innovation density~$f$, the residuals $\ZZ_1^{(n)}(\bth),\ldots,\ZZ_n^{(n)}(\bth)$ under ${\rm P}^{(n)}_{\bth ;f}$ are i.i.d.\ with density $f$. \vspace{-1mm}

\subsection{LAN}\label{secLAN}

Writing~$L^{(n)}_{\bth + n^{-1/2}\btau^{(n)}/\bth; f}:=\log {\rm dP}\n_{\bth + n^{-1/2}\btau^{(n)};f}/{\rm dP}\n_{\bth;f}$ for the log-likelihood ratio of ${\rm P}^{(n)}_{\bth + n^{-1/2}\btau^{(n)};f}$
  with respect to~${\rm P}^{(n)}_{\bth;f}$, where $\btau^{(n)} $ is a bounded sequence of $\mathbb{R}^{(p+q)d^2}$,   
let \vspace{-1mm}
\begin{equation}\label{Delta}
\bDelta^{(n)}_{f} (\bth) := \M_{\bth}^\prime  \P_{\bth}^\prime \Q_{\bth}^{(n)\prime} \bGamma_{f}^{(n)}(\bth),
\vspace{-1mm}\end{equation}
where $\M_{\bth}$, $ \P_{\bth}$, and $\Q^{(n)}_{\bth}$  (see \eqref{defM} and \eqref{defPQ} in Appendix~A for an explicit form)  do not depend on $f$ nor $\btau^{(n)} $ and 
\begin{equation}\label{Sf1}
\bGamma_{f}^{(n)}(\bth) := \big((n-1)^{1/2} (\text{vec}{\bGamma_{1, f}^{(n)}(\bth)})^\prime, \ldots , (n-i)^{1/2} (\text{vec}{\bGamma_{i, f}^{(n)}(\bth)})^\prime, \ldots , (\text{vec}{\bGamma_{n-1, f}^{(n)}(\bth)})^\prime\big)^\prime
\end{equation}
with the so-called  $f$-{\it cross-covariance matrices}\vspace{-1mm}
\begin{equation}\label{Gamma1}
\bGamma_{i, f}^{(n)}(\bth) := (n-i)^{-1} \sum_{t=i+1}^n \bvp_{f}(\ZZ_t^{(n)}(\bth)) \ZZ_{t-i}^{(n)\prime} (\bth).
\end{equation}
We then have the following LAN result (see Appendix B for a proof).
\begin{proposition}\label{Prop.LAN1}
Let Assumptions (A1) and (A2) hold. Then, for any bounded sequence~$\btau^{(n)}$ in~$\R^{(p+q)d^2}$, under ${\rm P}^{(n)}_{\bth;f}$, as $n \rightarrow \infty$,\vspace{-1mm} 
\begin{equation}\label{lik1}
L^{(n)}_{\bth + n^{-1/2}\btau^{(n)}/\bth; f} = \btau^{{(n)}\prime} \bDelta^{(n)}_{f}(\bth) - \frac{1}{2} \btau^{(n)\prime} \bLam_{f} (\bth) \btau^{(n)} + o_{\rm P}(1)\vspace{-1mm}
\end{equation}
with 
$$\bLam_{f} (\bth) := \M_{\bth}^\prime  \P_{\bth}^\prime \underset{n\rightarrow \infty}{\lim} \left\lbrace \Q_{\bth}^{(n)\prime} \left[\I_{n-1} \otimes \bXi \otimes \mathbfcal{I}(f)\right] \Q_{\bth}^{(n)} \right\rbrace  \P_{\bth} \M_{\bth},$$ 
and
$\bDelta^{(n)}_{f} (\bth) $ is asymptotically normal, with mean $\0$ and variance $ \bLam_{f} (\bth))$.
\end{proposition}

The class   ${\cal F}_d$ contains, among others,  the elliptical densities.  
Recall that a $d$-dimensional random vector $\ZZ$ has centered elliptical distribution with  {\it scatter matrix} $\bSigma$ and {\it radial density} ${\mathfrak f}$ if its density has the form 
 $ 
f({\bf z}) = \kappa^{-1}_{d,{\mathfrak f}}\big(\text{det}{\boldsymbol \Sigma}\big)^{-1/2}{\mathfrak f}\big( ({\bf z}^\prime{\boldsymbol \Sigma}^{-1}{\bf z})^{1/2}
\big)
$  
for some symmetric positive definite~$\boldsymbol \Sigma$ and some    function
~$\mathfrak f$:  $\mathbb{R}^+\to\mathbb{R}^+$ 
 such that  
 $\int_0^\infty r^{d-1}{\mathfrak f}(r)\,{\mathrm d}r<~\!\infty$;  \linebreak  $\kappa_{d,{\mathfrak f}}:= \big({2\pi^{d/2}}/{\Gamma (d/2)}\big)\int_0^\infty r^{d-1}{\mathfrak f}(r)\,{\mathrm d}r$ is a norming constant. When $\ZZ$ is elliptical with  shape matrix~$\boldsymbol \Sigma$ and radial density~$\mathfrak f$, 
 $\Vert\boldsymbol \Sigma^{-1/2}{\ZZ} \Vert$ (where $\boldsymbol \Sigma^{1/2}$ stands for the symmetric root of $\boldsymbol \Sigma$) has density 
 ${f}^\star_{ d;{\mathfrak f}}(r)=(\mu_{d-1;{\mathfrak{f}}})^{-1}r^{d-1}\mathfrak{f}(r)I[r>0]$ 
  where~$\mu_{d-1;{\mathfrak{f}}}\!:=\!\int_0^\infty r^{d-1}\mathfrak{f}(r){\rm d}\, r$, 
 and   distribution function $F^\star_{d;{\mathfrak f}}$.   Assumption~(A2){(ii)}  on $f$ then is equivalent to   the  mean square differentiability, with quadratic mean derivative ${\rm D}{\mathfrak f}^{1/2}$,  of~$r\mapsto{\mathfrak f}^{1/2}(r)$, $r\in\R_0^+$; 
 letting $\varphi_{\mathfrak f}\!:=\!-2{\rm D}{\mathfrak f}^{1/2}/ {\mathfrak f}^{1/2}(r)
 $, we  get
 $
{\cal I}_{d;{\mathfrak f}}:=
 \int_0^1\!\Big(\varphi_{\mathfrak f} \circ \big( F^\star_{d;{\mathfrak f}}\big)^{-1}\!(u) \Big) ^2{\rm d}u <~\!\infty.$

Elliptic random vectors admit the following representation in terms of spherical uniform variables. Denoting by  ${\mathbb S}_d$ and $\mathcal{S}_{d-1}$   the open unit ball
and the unit sphere in~$\mathbb{R}^d$, respectively, define the spherical uniform distribution ${\rm U}_d$ over ${\mathbb S}_d$ as the product of the uniform measure over~$\mathcal{S}_{d-1}$ 
with a uniform measure over the unit interval of distances to the origin. 
 A  $d$-dimensional random vector $\ZZ$ has centered elliptical distribution   
 iff~$F^\star_{d;{\mathfrak f}}(\Vert \bSigma^{-1/2}\ZZ\Vert) \bSigma^{-1/2} \ZZ / \Vert \bSigma^{-1/2}\ZZ\Vert\sim {\rm U}_d$.  Putting  
 $\S_{\bSigma,t}^{(n)} := \bSigma^{-1/2} \ZZ^{(n)}_t/\Vert \bSigma^{-1/2}\ZZ^{(n)}_t\Vert$, where   $\ZZ^{(n)}_t:=\ZZ^{(n)}_t(\bth)$,  it follows from Hallin and Paindaveine~(2004) that
the central sequence \eqref{Delta} for elliptical $f$ considerably simplifies and  takes the form \eqref{Delta} with\vspace{-1mm}
\begin{equation}\label{Gamma2}
\bGamma_{i, f}^{(n)}(\bth) := (n-i)^{-1} \bSigma^{-1/2}\sum_{t=i+1}^n
 {\varphi}_{1}(\Vert \bSigma^{-1/2}\ZZ^{(n)}_t\Vert) 
 {\varphi}_{2}(\Vert \bSigma^{-1/2}\ZZ^{(n)}_{t-i}\Vert)\S_{\bSigma,t}^{(n)}\S^{(n)\prime}_{\bSigma,t-i}\bSigma^{1/2}
\vspace{-1mm}\end{equation}
where  ${\varphi}_{1}(r):={\varphi}_{\mathfrak f}(r)$ and ${\varphi}_{2}(r):=r$, $r\in\mathbb{R}^+$.  

\section{Center-outward ranks and signs}\label{secranks}

Parametrically optimal (in the H\' ajek-Le Cam asymptotic sense) rank-based inference procedures in LAN families is possible if the   LAN central sequence can be expressed in terms of signs and ranks. 
 In Section \ref{SecSR}, we explain how to achieve this goal using the notions of multivariate ranks and signs proposed by Chernozhukov et al.~(2017) (under the name of Monge-Kantorovich ranks and signs) and developed in Hallin~(2017) and Hallin et al.~(2020a)  
 under the name of {\it center-outward ranks and signs}. 
 This new concepts  hinge on measure transportation theory;  their empirical versions are  based on  an optimal  
 coupling of the sample residuals ${\bf Z}\n_t$ with a regular grid over the unit ball. 

\subsection{Mapping the residuals to the unit ball}
Let $\mathcal{P}_d$ denote the family of all distributions~$\rm P$ with densities in ${\mathcal F}_d$---for this family  the center-outward distribution functions defined below  are continuous; see Hallin et al.~(2020a). The {\it center-outward distribution function}~$\F_{{\pms}}$ is defined as the a.e.~unique gradient of convex function mapping $\R^d$ to $\mathbb{S}_{d}$  and  pushing~$\rm P$ forward to    ${\rm U}_d$. For ${\rm P}\in{\mathcal P}_d$, such mapping is  a homeomorphism between ${\mathbb S}_d\setminus\{{\bf 0}\}$ and $\mathbb{R}^d\setminus \F_{{\pms}}^{-1}(\{{\bf 0}\})$ (Figalli~2018) and   the corresponding {\it center-outward quantile function} is defined  (letting, with a small abuse of notation,~$\Q_{\pms} ({\bf 0}) := \F_{\pms}^{-1}(\{{\bf 0}\})$) as~$\Q_{\pms} := \F_{\pms}^{-1}$. For any given distribution $\rm P$, $\Q_{{\pms}}$ induces a collection of continuous, connected, and nested quantile  contours and regions;   the {\it center-outward median}  $\Q_{\pms}(\0)$ is a uniquely defined  compact set of Lebesgue measure zero. We refer to Hallin et al.~(2020a) for details. 

Turning to the sample, for any  $\bth \in \boldsymbol{\Theta}$, the residuals 
$
\ZZ^{(n)}(\bth)\! :=\! (\ZZ_1^{(n)}(\bth),\ldots , \ZZ_n^{(n)}(\bth))
$ 
under ${\rm P}^{(n)}_{\bth;f}$
  are i.i.d.\ with density $f\in{\cal F}_d$ and center-outward 
distribution function $\F_{\pms}$.  
 For the empirical counterpart~$\F^{(n)}_{\pms}$ of $\F_{\pms}$,   let $n$ factorize into $n=n_R n_S +n_0,$ for $n_R, n_S, n_0 \in \mathbb{N}$ and~$0\leq n_0 < \min \{ n_R, n_S \}$, where $n_R \rightarrow \infty$ and $n_S \rightarrow \infty$ as $n \rightarrow \infty$, and consider a sequence of grids, where each grid
consists of the intersection between an $n_S$-tuple $(\boldsymbol{u}_1,\ldots \boldsymbol{u}_{n_S})$ of unit vectors, and the $n_R$-hyperspheres centered at the origin, with radii $1/(n_R+1),\ldots ,n_R/(n_R+1)$, along with~$n_0$ copies of the origin. The resulting grid  is such that the discrete distribution with probability mass~$1/n$ at each gridpoint and probability mass $n_0/n$ at the origin converges weakly to the uniform ${\rm U}_d$ over the ball $\mathbb{S}_d$. Then, we  define $\F_{{\pms}}^{(n)}(\ZZ^{(n)}_t) $, for $t = 1,\ldots ,n$ as the solution (optimal mapping) of a coupling problem between the residuals 
and the grid. 

Specifically, the empirical center-outward distribution function is the (random) mapping $$
\F^{(n)}_{{\pms}}: \ZZ^{(n)}:= (\ZZ^{(n)}_1,\ldots ,\ZZ^{(n)}_n) \mapsto (\F^{(n)}_{{\pms}}(\ZZ^{(n)}_1),\ldots ,\F^{(n)}_{{\pms}}(\ZZ^{(n)}_n))
$$
satisfying 
\begin{equation}\label{Fpm0}
\s \Vert \ZZ^{(n)}_t - \F^{(n)}_{\pms}(\ZZ^{(n)}_t)\Vert ^2 = \underset{T \in \mathcal{T}}{\min} \s \Vert \ZZ^{(n)}_t - T(\ZZ^{(n)}_t)\Vert ^2,
\end{equation}
where $\ZZ^{(n)}_t=\ZZ^{(n)}_t(\bth)$, 
the set $\{\F^{(n)}_{{\pms}}(\ZZ^{(n)}_t)\vert t = 1,\ldots ,n\}$ coincides with the~$n$ points of the grid,
and $\mathcal{T}$ denotes the set of all possible bijective mappings between $\ZZ^{(n)}_1, \ldots , \ZZ^{(n)}_n$ and the~$n$ grid points.  The sample counterpart of $\Q_{\pms}$ then is defined as  $\Q^{(n)}_{\pms} := (\F_{{\pms}}^{(n)})^{-1}$  (again, 
with the small abuse of notation~$\Q^{(n)}_{\pms} ({\bf 0}) := (\F_{{\pms}}^{(n)})^{-1}(\{{\bf 0}\})$).  See Appendix D.1 for a graphical illustration of these concepts.

Based on this empirical center-outward distribution function,  
  the {\it center-outward ranks}  and {\it signs} are 
\begin{equation}
{R^{(n)}_{{\pms}, t}}:=  {R^{(n)}_{{\pms}, t}}(\bth):=({n_R + 1}) \Vert \F^{(n)}_{\pms}(\ZZ^{(n)}_t) \Vert,  \label{Ranks}
\end{equation}
and (for $\F^{(n)}_{\pms}(\ZZ^{(n)}_t) = \boldsymbol{0}$, let $\S^{(n)}_{{\pms}, t}:=\boldsymbol{0}$)\vspace{-2mm} 
\begin{equation}
\S^{(n)}_{{\pms}, t}:=  {\S^{(n)}_{{\pms}, t}}(\bth):=  \frac{\F^{(n)}_{\pms}(\ZZ^{(n)}_t)}{
\Vert \F^{(n)}_{\pms}(\ZZ^{(n)}_t) \Vert} I [\F^{(n)}_{\pms}(\ZZ^{(n)}_t) \neq \boldsymbol{0}], \label{Signs}\vspace{-2mm}
\end{equation}
respectively.   It follows that $\F^{(n)}_{\pms}(\ZZ_t^{(n)})$ factorizes into
\begin{equation}\label{factFpm}
\F^{(n)}_{\pms}(\ZZ_t^{(n)})=   \frac{R^{(n)}_{{\pms}, t}}{n_R + 1}  \S^{(n)}_{{\pms}, t},\quad\text{ hence }\quad \ZZ_t^{(n)}= \Q^{(n)}_{\pms}\Big(
 \frac{R^{(n)}_{{\pms}, t}}{n_R + 1}  \S^{(n)}_{{\pms}, t}\Big).
\end{equation}


Conditional on the grid (in case the latter is random), those ranks and signs are jointly distribution-free: more precisely,  under~${\rm P}^{(n)}_{\bth;f}$,  the $n$-tuple $\F^{(n)}_{\pms}(\ZZ_1^{(n)}),\ldots , \F^{(n)}_{\pms}(\ZZ_n^{(n)})$ is uniformly distributed over the $n!$ permutations\footnote{Actually, for $n_0>1$, the $n!/n_0!$ permutations with repetitions.} of the~$n$  gridpoints, irrespective of~$f\in{\cal F}_d$. 
 
We refer to Sections~3.1 and~6 in Hallin et al.~(2020a) for further details,  comments on the main properties of  $\F^{(n)}_{\pms}$ and $\F_{\pms}$, and for remarks on the sufficiency and maximal ancillarity of the sub-$\sigma$-fields generated by the order statistic\footnote{An {\it order statistic} $\ZZ_{(\cdot)}^{(n)}$ of the un-ordered $n$-tuple $\ZZ^{(n)}$ is an arbitrarily ordered version of the same; see Appendix D and Hallin et al.~(2020a).}   and by the center-outward ranks and signs, in the fixed-$\bth$ experiment $\{{\rm P}\n_{\bth ;f}\vert f\in{\mathcal F}_d\}$. 
\subsection{Mahalanobis ranks and signs}\label{Mahalsec}
Definitions  (\ref{Ranks}) and (\ref{Signs}) call for a comparison with the earlier concepts of {\it elliptical} or {\it Mahalanobis ranks and signs} introduced in Hallin and Paindaveine~(2002a and b, 2004), which we now describe.  
Associated with the centered elliptical distribution with scatter $\bSigma$ and radial density $\mathfrak f$,   consider  
the mapping 
$\bz\mapsto {\F}_{\text{\tiny\!{\rm ell}}}(\bz):=\Fstar(\Vert\bSigma^{-1/2}\bz\Vert) \bSigma^{-1/2} \bz / \Vert\bSigma^{-1/2}\bz\Vert$ 
from~$\mathbb{R}^d$ to~${\mathbb S}_d$. In measure transportation  parlance, ${\F}_{\text{\tiny\!{\rm ell}}}$, just as  ${\F}_{\pms}$, 
{\it pushes}  the elliptical distribution of~$\ZZ$ {\it forward to} the uniform ${\rm U}_d$ over the unit ball~${\mathbb S}_d$. 
This allows us to connect the Mahalanobis ranks and signs to the center-outward ones. 

Denoting by $\widehat{\bSigma}^{(n)}$ a consistent estimator of $\bSigma$ measurable with respect to the order statistic\footnote{That is, a symmetric function of the $\ZZ_t$'s.} of the $\ZZ\n_t$'s and by $F^{\star (n)}$ the empirical distribution function 
of the moduli~$\big\Vert \big(\widehat{\bSigma}^{(n)}\big)^{-1/2}\ZZ\n_t\big\Vert$, an empirical counterpart of ${\F}_{\text{\tiny\!{\rm ell}}}(\ZZ\n_t)$ is \vspace{-2mm}
\begin{equation}\label{hatmapell}
 {\F}^{(n)}_{\text{\tiny\!{\rm ell}}}(\ZZ\n _t)
 :=F^{\star (n)}\big(\big\Vert \big(\widehat{\bSigma}^{(n)}\big)^{-1/2}\ZZ\n_t\big\Vert\big) 
  \frac{\big(\widehat{\bSigma}^{(n)}\big)^{-1/2} \ZZ\n_t }{\big\Vert \big(\widehat{\bSigma}^{(n)}\big)^{-1/2}\ZZ\n_t\big\Vert}
\vspace{-2mm}\end{equation}
with the 	 {\it  Mahalanobis ranks}	 (compare to \eqref{Ranks})												\begin{equation}
R^{(n)}_{\text{ell}, t}:= R^{(n)}_{\text{ell}, t}(\bth):= (n+1)\Vert  {\F}\n_{\text{\tiny\!{\rm ell}}}(\ZZ\n _t)\Vert=(n+1)F^{\star (n)}\big(\big\Vert \big(\widehat{\bSigma}^{(n)}\big)^{-1/2}\ZZ\n_t\big\Vert\big) 
\label{MalRanks}
\end{equation}
and {\it  Mahalanobis signs}  (compare to \eqref{Signs})
\begin{equation}
\S^{(n)}_{\text{ell}, t}:=\S^{(n)}_{\text{ell}, t}(\bth)
:= \frac{\F^{(n)}_{\text{\tiny\!{\rm ell}}}(\ZZ^{(n)}_t)  }{
\Vert \F^{(n)}_{\text{\tiny\!{\rm ell}}}(\ZZ^{(n)}_t) \Vert }I [\F		^{(n)}_{\text{\tiny\!{\rm ell}}}(\ZZ^{(n)}_t) \neq \boldsymbol{0}]=\frac{\big(\widehat{\bSigma}^{(n)}\big)^{-1/2} \ZZ\n_t}{\big\Vert \big(\widehat{\bSigma}^{(n)}\big)^{-1/2}\ZZ\n_t\big\Vert}
I \big[\big(\widehat{\bSigma}^{(n)}\big)^{-1/2} \ZZ\n_t \neq \boldsymbol{0}\big] \label{MalSigns}
\end{equation}
(for $\F^{(n)}_{\text{\tiny\!{\rm ell}}}(\ZZ^{(n)}_t)  = \boldsymbol{0}=\big(\widehat{\bSigma}^{(n)}\big)^{-1/2} \ZZ\n_t $, let $\S^{(n)}_{\text{ell}, t}:=\boldsymbol{0}$). Similar to \eqref {factFpm}), we have 
\begin{equation}\nonumber
\F^{(n)}_{\text{\tiny\!{\rm ell}}}(\ZZ_t^{(n)})=   \frac{R^{(n)}_{\text{ell},t}}{n+1}  \S^{(n)}_{\text{ell}, t},\quad\!\!\text{hence}\!\!\quad \widehat{\bSigma}^{(n)-1/2} \ZZ_t^{(n)}\!=  (F^{\star (n)})^{-1}\!\left(\frac{R^{(n)}_{\text{ell},t}}{n+1} \right)\S^{(n)}_{\text{ell}, t}= {\bSigma}^{-1/2} \ZZ_t^{(n)} + o_{\mathrm P}(1). \vspace{-1mm}
\end{equation}
\subsection{Elliptical ${\F}_{\text{\tiny\!{\rm ell}}}$, center-outward ${\F}_{\pms}$, and affine invariance}\label{affinvsec}

Both ${\F}_{\text{\tiny\!{\rm ell}}}$ and  ${\F}_{\pms}$ are 
{\it pushing}  the elliptical distribution of~$\ZZ$ {\it forward to} ${\rm U}_d$. However, unless~$\boldsymbol \Sigma$ is proportional to 
identity ($\bSigma = c\I_d$ for some $c>0$), ${\F}_{\text{\tiny\!{\rm ell}}}$ and  ${\F}_{\pms}$ are distinct, so that~${\F}_{\text{\tiny\!\!
{\rm ell}}}$ cannot be the gradient of a convex function. Moreover,
both ${\F}_{\text{\tiny\!{\rm ell}}}$ and  ${\F}_{\pms}$ {\it sphericize} the distribution of~$\ZZ$.   Some key differences are worth to be mentioned, though.

First, while   sphericization and probability integral transformation, in ${\F}_{\pms}$, are inseparably combined,  ${\F}_{\text{\tiny\!{\rm ell}}}$ proceeds in two separate steps: a  {\it Mahalanobis sphericization} step (the parametric affine transformation~$\bz \mapsto \bz_{\bSigma,\boldsymbol\mu}:={\bSigma}^{-1/2} (\bz-\boldsymbol\mu$)) 
 first, followed by the spherical probability integral trans\-formation~$\bz_{{\bSigma,\boldsymbol\mu}}\mapsto F^\star_{d;{\mathfrak f}}(\Vert\bz_{{\bSigma,\boldsymbol\mu}}\Vert)\bz_{{\bSigma,\boldsymbol\mu}}/\Vert\bz_{{\bSigma,\boldsymbol\mu}}\Vert$. 

Second, Mahalanobis sphericization requires  centering, hence the definition of a location parameter $\boldsymbol\mu$. Distinct choices of location (mean, spatial median, etc.) all yield the same result under ellipticity,  but not under non-elliptical distributions. This is in sharp contrast with ${\F}_{\pms}$, which is location-invariant (see Hallin et al.~(2020b)). Similarly, all definitions and sensible estimators of the scatter yield the same results under
elliptical symmetry but not under non-elliptical
distributions.

Third, even under additional  assumptions ensuring the identification of $\boldsymbol\Sigma$, the Mahalanobis sphericization, hence also ${\F}_{\text{\tiny\!{\rm ell}}}$,   only sphericizes elliptical distributions, 
whilst ${\F}_{\pms}$ sphericizes them all. 
Its  preliminary Mahalanobis sphericization step actually makes ${\F}_{\text{\tiny\!{\rm ell}}}$ affine-invariant. Assuming that sensible choices of $\boldsymbol\mu$ and $\boldsymbol\Sigma$ are available,  performing the same Mahalanobis transformation prior to determining ${\F}_{\pms}$ similarly would make center-outward distribution functions affine-invariant and the corresponding center-outward quantile functions affine-equivariant (in fact, for elliptical distributions, the resulting~${\F}_{\pms}$  then coincides with~${\F}_{\text{\tiny\!{\rm ell}}}$).  Whether this is desirable is a matter of choice. While affine-invariance, in view of the central role of the affine group in elliptical families,  is quite natural under elliptical symmetry, its relevance  is much less  obvious away from ellipticity. A more detailed discussion of this fact, along with additional arguments related to the lack of affine invariance of  non-elliptical local experiments, can be found in Hallin et al.~(2020b).

%


Besides affine invariance issues, center-outward distribution functions, ranks, and signs inherit, from the invariance properties of  Euclidean distances, elementary but remarkable invariance and equivariance properties: as shown in~Hallin et al. (2020b)  they enjoy invariance/equivariance with respect to shift,  global scale factors, and orthogonal transformations.

\subsection{A center-outward sign- and rank-based central sequence} \label{SecSR}

Efficient estimation in  LAN experiments is based on  central sequences and the so-called {\it Le Cam one-step method}. Our R-estimation is based on the same principles. Specifically, in the central sequence associated with some reference density $f$, we replace the residuals $\ZZ(\bth)$  with some adequate function of their ranks and their signs. Then, from the resulting rank-based statistic, we implement  a suitable adaptation of the one-step method. If, under innovation density $f$, the substitution yields a genuinely rank-based, hence distribution-free, version of the central sequence, 
the resulting R-estimator achieves parametric efficiency under $f$ while remaining valid under other innovation densities; see Hallin and Werker~(2003)  for a discussion.


In dimension $d=1$, 
Allal et al.~(2001), Hallin and La Vecchia (2017, 2019 and references therein) explain how  to construct R-estimators for linear and nonlinear semiparametric time series models. 
In dimension $d>1$, under elliptical innovations density,  Hallin et al.~(2006) exploit similar ideas for the estimation of  shape matrices,  based on the  Mahalanobis ranks and signs.  Hallin and Paindaveine~(2004), in a hypothesis testing context, show that replacing~${\bf Z}\n_t$ in \eqref{Gamma2} with  
\begin{equation}\label{ellr}  
\widehat{\bSigma}^{(n)1/2} 
F^{\star -1}_{d;{\mathfrak f}}
(R^{(n)}_{\text{ell},t}/(n+1)) \S^{(n)}_{\text{ell},t}= \F_{\text{\tiny\!{\rm ell}}}^{-1}((R^{(n)}_{\text{ell},t}/(n+1))\S^{(n)}_{\text{ell},t}) = \F_{\text{\tiny\!{\rm ell}}}^{-1}(\F^{(n)}_{\text{\tiny\!{\rm ell}}}(\ZZ_t^{(n)}(\bth))) 
\end{equation}
(where $R^{(n)}_{\text{ell},t}=R^{(n)}_{\text{ell},t}(\bth)$,  $\S^{(n)}_{\text{ell},t}=\S^{(n)}_{\text{ell},t}(\bth)$, and  $\widehat{\bSigma}^{(n)}$ is a suitable estimator  of the scatter matrix) yields a rank-based version of the central sequence associated with the elliptic density~$f$---namely, a random vector $\tenq\bDelta\n_f(\bth)$   measurable with respect to the Mahalanobis ranks and signs (hence, distribution-free under ellipticity)  such that, under $f$,~$\tenq\bDelta\n_f(\bth) - \bDelta\n_f(\bth) $ is~$ o_{\rm P}(1)$ as~$n\to\infty$.

However, this construction is valid only for the family of elliptical innovation densities (in dimension one, the family of symmetric innovation densities), under which Mahalanobis ranks and signs are distribution-free. This is a severe limitation, which is unlikely to be satisfied in most applications. 
If the attractive properties of R-estimators  in univariate semiparametric time series models are to be extended to dimension two and higher, center-outward ranks and signs, the distribution-freeness of which holds under  any density $f\in{\cal F}_d$, are to be considered instead of the Mahalanobis ones.

Building on this remark,  we  propose to substitute $\ZZ\n_t(\bth)$  in \eqref{Gamma1} with
\begin{equation}
\F_{\pms}^{-1}(({R^{(n)}_{{\pms}, t}}/({n_R + 1)} ) \S^{(n)}_{{\pms}, t})=
  \F_{\pms}^{-1}(\F_{\pms}^{(n)}(\ZZ_t^{(n)}(\bth)))=\Q_{\pms}\circ \F_{\pms}^{(n)}(\ZZ_t^{(n)}(\bth)),\label{Subst} 
\end{equation}
where $R^{(n)}_{\pm,t}=R^{(n)}_{\pm,t}(\bth)$,  $\S^{(n)}_{\pm,t}=\S^{(n)}_{\pm,t}(\bth)$,  and $\F_{\pms}$ and $\Q_{\pms}$ are associated with some chosen reference innovation density $f\in{\cal F}_d$. 
This yields  rank-based, hence distribution-free, $f$-cross-covariance matrices of the form ($i = 1, \ldots , n - 1$)
\begin{equation}\label{tildeGamgen}
\tenq{\bGamma}_{i, f}^{(n)}(\bth) := (n-i)^{-1} \sum_{t=i+1}^n {\boldsymbol\varphi}_f\left(\F_{\pms}^{-1}\left(\frac{R^{(n)}_{{\pms}, t}}{n_R + 1}  \S^{(n)}_{{\pms}, t}\right)\right) \F_{\pms}^{-1\prime}\left(\frac{R^{(n)}_{{\pms}, t-i}}{n_R + 1}  \S^{(n)}_{{\pms}, t-i}\right) .
\end{equation}

While this looks quite straightforward, practical implementation requires an analytical expression for $\F_{\pms}$, which typically is unavailable for general innovation densities. 
And, were such closed forms available, the problem of choosing an adequate multivariate reference density~$f$ remains.

Now, note that in the univariate case  all standard reference densities are symmetric---think of Gaussian,  logistic, double-exponential densities,   leading to van der Waerden, Wilcoxon, or sign test scores. Therefore, in the sequel, we concentrate on rank-based cross-covariance matrices of the form  ($ i = 1, \ldots , n - 1$)
\begin{equation}\label{tildeGam}
\tenq{\bGamma}_{i, {J}_1, {J}_2}^{(n)}(\bth) := (n-i)^{-1} \sum_{t=i+1}^n {J}_1\left(\frac{R^{(n)}_{{\pms}, t}}{n_R + 1}\right) {J}_2\left(\frac{R^{(n)}_{{\pms}, t-i}}{n_R + 1}\right)  \S^{(n)}_{{\pms}, t} \S^{(n)\prime}_{{\pms}, t-i}
\end{equation}
to which  $\tenq{\bGamma}_{i, f}^{(n)}(\bth)$ in \eqref{tildeGamgen} reduces, with ${J}_1(u)={\varphi}_{\mathfrak f}(F^{\star -1}_{d;{\mathfrak f}}(u))$ and~${J}_2(u)=F^{\star -1}_{d;{\mathfrak f}}(u)$,  
 in the case of a spherical reference $f$ with radial density $\mathfrak f$, yielding  a rank-based version $\tenq\bDelta\n_f$ of the spherical central sequence $\bDelta\n_f$. 
More generally, we propose to use statistics of the form~(\ref{tildeGam}) with scores~${J}_1: [0,1) \rightarrow  \mathbb{R}$ and~${J}_2: [0,1) \rightarrow  \mathbb{R}$ which are not necessarily related to any spherical density. Then, the notation $\tenq\bDelta\n_{{J}_1, {J}_2}$  will be used in an obvious fashion,  indicating that~$\tenq\bDelta\n_{{J}_1, {J}_2}$ needs not be a central sequence. 

The  next section provides details on the choice of ${J}_1$ and ${J}_2$ and establishes the asymptotic properties (root-$n$ consistency and asymptotic normality) of the related R-estimators.

\section{R-estimation} \label{Sec_Rest}

\subsection{One-step R-estimators: definition and asymptotics}\label{subsec.def}

We now proceed with a precise definition of our R-estimators and establish their asymptotic properties. Throughout,~${J}_1$ and $ {J}_2$  are assumed to satisfy the following  assumption.\smallskip

\noindent\textbf{Assumption (A3).} The score  functions ${J}_1$ and $ {J}_2$ in \eqref{tildeGam}  {\it (i)} are square-integrable, that is,~$\sigma_{{J}_l}^2:=
\int_0^1 {J}_l^2(r) {\rm d}r < \infty,\  l=1,2,$
and  {\it (ii)}  are continuous differences of two monotonic increasing functions.\medskip

Assumption~(A3) is quite mild and it is satisfied, e.g., by all square-integrable functions with bounded variation. 
Define
$\J_{{J}_2,f} := \int_{\mathbb{S}_d} {J}_2(\Vert \bu \Vert) ({\bu}/{\Vert \bu \Vert}) \F^{-1 \prime}_{\pms}(\bu)  {\rm dU}_d(\bu),$
and
\begin{align}
\K_{{J}_1, {J}_2, f} &:=   \int_{\mathbb{S}_d} 
 {J}_1(\Vert \bu \Vert)\left[\I_d \otimes  \frac{\bu}{\Vert \bu \Vert}
 \right]\J_{{J}_2,f}\left[\I_d \otimes   \bvp^\prime_{f}\Big(\F_{\pms}^{-1}(\bu)
 \Big)
 \right] {\rm dU}_d(\bu).
\end{align}
These two matrices under Assumptions (A2) and 
(A3) exist and are finite in view of the Cauchy--Schwarz inequality
 since ${\bu}/{\Vert \bu \Vert}$ is bounded.

R-estimation requires the asymptotic linearity of the rank-based objective function involved. Sufficient conditions for  such linearity are available in the literature (see e.g.  Jure\v{c}kov\' a~(1971) and van Eeden~(1972), Hallin and Puri~(1994), Hallin and Paindaveine~(2005) or Hallin et al.~(2015)).  
In the same spirit, 
 we introduce the following assumption   on the rank-based statistics~$\tenq{\bGamma}_{i, {J}_1, {J}_2}^{(n)}(\bth)$; the form of the linear term in the right-hand side of \eqref{assA4} follows from the form of the asymptotic shift in Lemma~\ref{bar.til.Gam}. 

 Decomposing the matrix $\Q^{(n)}_{\bth}$ defined in \eqref{defPQ}  into $d^2 \times d^2 (p+q)$ blocks (note that these blocks do not depend on $n$),  write  
$\Q^{(n)}_{\bth} =
\big(
\Q_{1, \bth}^\prime 
\ldots 
\Q_{n-1, \bth}^\prime
\big)^\prime$  
and consider the following assumption. 
\color{black}

\noindent\textbf{Assumption (A4)}  For any positive integer $i$ and $d^2(p+q)$-dimensional vector $\btau$, under actual density $f$, as~$n\to\infty$
\begin{equation}\label{assA4}
(n-i)^{1/2}\left[\text{vec} (\tenq{\bGamma}^{(n)}_{i, {J}_1, {J}_2}(\bth + n^{-1/2}\btau)) - \text{vec} (\tenq{\bGamma}^{(n)}_{i, {J}_1, {J}_2}(\bth)) \right] = - \K_{{J}_1, {J}_2, {f}} \Q_{i, \bth} \P_{\bth} \M_{\bth} \btau + o_{\rm P}(1), 
\end{equation}
where  
${\bf M}_{\bth}$ and ${\bf P}_{\bth}$, which do not depend on $f$, ${J}_1$ nor ${J}_2$,  are given in 
(\ref{defM}) and (\ref{defPQ}) in Appendix~\ref{AppLAN}. \medskip

Next, for $m \leq n-1$, consider   \vspace{-1mm} 
\begin{equation}\label{tilde.S.m}
\tenq{\bGamma}_{{J}_1, {J}_2}^{(m,n)}(\bth) := ((n-1)^{1/2}(\text{vec} \tenq{\bGamma}_{1, {J}_1, {J}_2}^{(n)}(\bth))^\prime, \ldots , (n-m)^{1/2}(\text{vec} \tenq{\bGamma}_{m, {J}_1, {J}_2}^{(n)}(\bth))^\prime)^\prime,
\end{equation}
and the truncated version  
\begin{equation}\label{tildeDelta.m}
\tenq{\bDelta}^{(n)}_{m, {J}_1, {J}_2}(\bth) := \T^{(m+1)}_{\bth} \tenq{\bGamma}_{{J}_1, {J}_2}^{(m,n)}(\bth)\quad\text{where}\quad \T^{(m+1)}_{\bth} := \M_{\bth}^\prime  \P_{\bth}^\prime \Q_{\bth}^{(m+1)\prime}
\end{equation}
of   $\tenq{\bDelta}^{(n)}_{{J}_1, {J}_2}(\bth)$. 
This truncation is just a theoretical device required in the statement of asymptotic results and, as  explained in Appendix C,  there is no need to implement it in practice.  The asymptotic linearity \eqref{assA4} of~$\tenq{\bGamma}_{i, {J}_1, {J}_2}^{(n)}(\bth)$ entails, for $\tenq{\bDelta}^{(n)}_{{J}_1, {J}_2}(\bth)$, the following result.

\begin{proposition}\label{asy0}
Let Assumptions (A1), (A2), (A3), and (A4) hold. Then,  for  any $(m,n)$ such that $m  \leq n-1$ and  $m\to\infty$ (hence also~$n\to\infty$), \vspace{-1mm} 
\begin{equation}\label{asy.linear2}
\tenq{\bDelta}^{(n)}_{{J}_1, {J}_2}(\bth + n^{-1/2}\btau) - \tenq{\bDelta}^{(n)}_{m, {J}_1, {J}_2}(\bth) =  -  \bUpsilon^{(m+1)}_{{J}_1, {J}_2, f}(\bth)  \btau + o_{\rm P}(1),
\end{equation}
where $\bUpsilon^{(m+1)}_{{J}_1, {J}_2, f}(\bth) :=   \T^{(m+1)}_{\bth} (\I_{m} \otimes \K_{{J}_1, {J}_2, f}) \T^{(m+1)\prime}_{\bth}.$ 
\end{proposition}

With the above asymptotic linearity result, we are now ready to define our R-estimators. First, let us introduce some notations. Under Assumption (A1), let $\bUpsilon_{{J}_1, {J}_2, f}(\bth)\! :=\! \underset{n \rightarrow \infty}{\lim}  \bUpsilon^{(n)}_{{J}_1, {J}_2, f}(\bth)$
and define  the {\it cross-information matrix}
\begin{equation}\label{Xinfo}
\mathbf{I}_{{J}_1, {J}_2, f}(\bth) := \underset{n \rightarrow \infty}{\lim} {\rm E}_{\bth, f} \left[ \tenq{\bDelta}^{(n)}_{{J}_1, {J}_2}(\bth)  \bDelta^{(n)}_{f}(\bth)^\prime  \right].
\end{equation}
Let
 \begin{equation}\label{barGam}
\bar{\bGamma}_{i, {J}_1, {J}_2}^{(n)}(\bth) := (n-i)^{-1} \sum_{t=i+1}^n {J}_1(\Vert  \F_{{\pms}, t} \Vert ) {J}_2(\Vert  \F_{{\pms}, t-i} \Vert )  \S_{{\pms}, t} \S^\prime_{{\pms}, t-i} 
\end{equation}
with 
$
\S_{{\pms}, t} := \F_{{\pms}, t}/\Vert  \F_{{\pms}, t} \Vert 
$ 
representing the ``sign'' of  $\F_{{\pms}, t}:=\F_{\pms}(\ZZ^{(n)}_t(\bth))$. Denote by $\bar{\bDelta}^{(n)}_{{J}_1, {J}_2}(\bth)$ the central sequence resulting from substituting  $\bar{\bGamma}_{i, {J}_1, {J}_2}^{(n)}(\bth)$ for $\tenq{\bGamma}_{i, {J}_1, {J}_2}^{(n)}(\bth)$ in $\tenq{\bDelta}^{(n)}_{{J}_1, {J}_2}(\bth)$. Following the proofs in Lemma \ref{lem.bar.til.Delta.n} and Lemma \ref{bar.til.Gam}, it is not difficult to see that the difference between~$\bar{\bDelta}^{(n)}_{{J}_1, {J}_2}$ and~$\tenq{\bDelta}^{(n)}_{{J}_1, {J}_2}$ converges to zero in quadratic mean as $n \rightarrow \infty$. 
Therefore,~$\bUpsilon_{{J}_1, {J}_2, f}(\bth)$ coincides with the cross-information matrix \eqref{Xinfo} when Assumptions (A1), (A2) and (A3) hold; see the proof of Lemmas B.1 and B.4 in Appendix B.

Let  $\hat{\bUpsilon}_{{J}_1, {J}_2}^{(n)}$ denote a consistent (under innovation density $f$) estimator of $\bUpsilon_{{J}_1, {J}_2, f}(\bth)$;   such an estimator 
 is provided in~(\ref{asy.linear2}), see Appendix \ref{secalg} for details. Also, denote by $\hat{\bth}^{(n)}$ a preliminary root-n consistent and asymptotically discrete\footnote{Asymptotic discreteness is only  a theoretical requirement since, in practice, $\hat{\bth}^{(n)}$ anyway only has a bounded number of digits; see Le Cam and Yang (2000, Chapter 6) and van der Vaart (1998, Section 5.7) for   details. } estimator of $\bth$. 
Our one-step R-estimator then is defined as 
\begin{equation}\label{onestep.def}
\tenq{\hat{\bth}}\n := \hat{\bth}^{(n)} + n^{-1/2} \left( \hat{\bUpsilon}_{{J}_1, {J}_2}^{(n)} \right)^{-1} \tenq{\bDelta}^{(n)}_{{J}_1, {J}_2} ( \hat{\bth}^{(n)} ).
\end{equation}
The following proposition  establishes its root-${n}$ consistency and asymptotic normality.

\begin{proposition}\label{asy}
Let Assumptions (A1), (A2), (A3), and (A4) hold. Let \vspace{-1mm} 
$$\bOmega^{(n)} :=  d^{-2}\sigma^2_{{J}_1}\sigma^2_{{J}_2}  \left( \bUpsilon^{(n)}_{{J}_1, {J}_2, f}(\bth) \right)^{-1}  \T^{(n)}_{\bth} \T^{(n)\prime}_{\bth} \left( \bUpsilon^{(n) \, \prime}_{{J}_1, {J}_2, f}(\bth)  \right)^{-1}.$$
Then, denoting by $\big(\bOmega^{(n)}\big)^{-1/2}$ the symmetric square root of $\bOmega^{(n)} $, 
\begin{equation}
n^{1/2}\big(\bOmega^{(n)}\big)^{-1/2} (\tenq{\hat{\bth}}\n - \bth) \rightarrow {\mathcal N}(\0, \I_{d^2(p+q)}),\vspace{-1mm} 
\end{equation}
 under innovation density $f$, as  both $n_R$ and $n_S$ tend to infinity.
 \end{proposition}

See Appendix B for the proof.  Appendix \ref{secalg} discusses the   computational aspects of the procedure and describes the algorithm we are using.  
Codes are available from the authors'   {\tt GitHub} page \url{https://github.com/HangLiu10/RestVARMA}.

\subsection{Some standard score functions} \label{Sec: examples}

The rank-based cross-covariance matrices 
$\tenq{\bGamma}^{(n)}_{{J}_1, {J}_2}$, hence also the resulting R-estimator,  depend on the choice of score functions ${J}_1$ and ${J}_2$. 
We provide three examples of sensible choices extending scores that are widely applied in the 
univariate     (see e.g. Hallin and La Vecchia~(2019)) and  the elliptical multivariate setting  (see Hallin and Pandaveine~(2004)). 


\textbf{Example 1} ({\it Sign test} scores). Setting ${J}_1(u) =1= {J}_2(u) $  yields the center-outward sign-based  cross-covariance matrices 
\begin{align}
\tenq{\bGamma}_{i, \text{sign}}^{(n)}(\bth) = (n-i)^{-1} \sum_{t=i+1}^n   \S^{(n)}_{{\pms}, t}(\bth) \S^{(n)\prime}_{{\pms}, t-i}(\bth), \quad i = 1, \ldots , n-1. \label{Eq. Signs}
\end{align}
The resulting   $\tenq{\bDelta}^{(n)}_{\text{sign}}(\bth)$  entirely relies  on the center-outward  signs $\S^{(n)}_{{\pms}, t}(\bth)$, which should make them particularly robust  and explains the terminology  {\it sign test}  scores.  

\textbf{Example 2} ({\it Spearman} scores).   Another simple choice is ${J}_1(u) = {J}_2(u) = u$.   The corresponding rank-based cross-covariance matrices   are
\begin{align}
\tenq{\bGamma}_{i, \text{Sp}}^{(n)}(\bth) = (n-i)^{-1} \sum_{t=i+1}^n   \F^{(n)}_{{\pms}, t} \F^{(n)\prime}_{{\pms}, t-i}, \quad i = 1, \ldots , n-1,
\end{align}
with $ \F^{(n)}_{{\pms}, t}:= \F^{(n)}_{{\pms}}({\bf Z}\n_t(\bth))$, reducing, for $d=1$, to Spearman autocorrelations, whence the terminology  {\it Spearman}  scores. 

\textbf{Example 3} ({\it van der Waerden} or {\it normal} scores). Finally,  ${J}_1(u) = {J}_2(u) = \big((F^{\chi^2}_d)^{-1} (u)\big)^{1/2}\!$, where $F^{\chi^2}_d$ denotes the chi-square distribution function with $d$ degrees of freedom, yields   the {\it van der Waerden} (vdW) {\it rank scores}, with 
 cross-covariance matrices  
\begin{align}
\tenq{\bGamma}_{i, \text{vdW}}^{(n)}(\bth) &= (n-i)^{-1}\!\! \sum_{t=i+1}^n \left[\big(F^{\chi^2}_d\big)^{-1}\! \left(\frac{R^{(n)}_{{\pms}, t}(\bth)}{n_R + 1}\right)\right]^{1/2}\! \left[\big(F^{\chi^2}_d\big)^{-1}\! \left(\frac{R^{(n)}_{{\pms}, t-i}(\bth)}{n_R + 1}\right)\right]^{1/2} 
\!\!  \S^{(n)}_{{\pms}, t} (\bth)\S^{(n)\prime}_{{\pms}, t-i}(\bth),\nonumber \\
& \qquad \qquad \qquad\quad\qquad \qquad \qquad\quad\qquad \qquad\qquad \qquad\qquad
\quad i = 1, \ldots , n-1. \label{vdWGam}
\end{align}

 Adequate  choices of ${J}_1$ and ${J}_2$, namely,  
 \begin{equation}\label{vp1vp2.optimal}
{J}_1  = \varphi_{\mathfrak f}
 \circ \left(F^\star_{d;\mathfrak f}\right)^{-1}   \quad \text{and} \quad  {J}_2 = \left(F^\star_{d;\mathfrak f}\right)^{-1} ,
\end{equation}
  yield asymptotic efficiency of  $\tenq{\hat{\bth}}\n$ under spherical distributions with radial density $\mathfrak f$. Indeed,   it is shown in Chernozhukov et al. (2017)   that, for  spherical distributions, $\F_{{\pms}}$ actually coincides with~$\F_{\text{\tiny\!{\rm ell}}}$. Hence,  $\bar{\bDelta}^{(n)}_{{J}_1, {J}_2}$, under spherical density $f$,    coincides with~the central sequence~$\bDelta^{(n)}_{f}$. 
Therefore, due to the convergence  in quadratic mean of $\tenq{\bDelta}^{(n)}_{{J}_1, {J}_2}$ to~$ \bar{\bDelta}^{(n)}_{{J}_1, {J}_2}$,  
 $\tenq{\bDelta}^{(n)}_{{J}_1, {J}_2}$ and 
$\bDelta^{(n)}_{f}$ are asymptotically equivalent and $\bUpsilon_{{J}_1, {J}_2, f}(\bth)$ coincides with the Fisher information matrix and~$\tenq{\hat{\bth}}\n$  achieves (parametric) asymptotic efficiency. 

Condition \eqref{vp1vp2.optimal} is satisfied by the van der Waerden scores for Gaussian $\mathfrak f$: the corresponding R-estimator, thus, is parametrically efficient under spherical Gaussian innovations. If the residuals are sphericized prior to the computation of  center-outward ranks and signs, then parametric efficiency is reached under any Gaussian innovation density; we have explained in Section~\ref{affinvsec} why this may be desirable or not. Neither the Spearman nor the sign test scores  satisfy \eqref{vp1vp2.optimal} for any $\mathfrak f$.  Efficiency, however,  is just one possible criterion for the selection of $J_1$ and $J_2$ and many alternative options are available, based on   ease-of-implementation (as in Examples 1 and~2) or robustness (as in Example 1). 

\section{Numerical illustration} \label{Sec: MC}

A numerical study of the performance of our R-estimators was conducted in dimensions $d=2$ (Sections~\ref{Sect: MClarge}, \ref{Sect: MCsmall}, and~\ref{subsec.AO}) and $d=3$ (Section~\ref{dim3sec}). Further results are available in Appendix~\ref{Supp.sim}.

In dimension $d=2$, we considered the bivariate  VAR(1) model\vspace{-1mm}
\begin{equation}\label{VAR}
\left(\I_d - \A L \right) \X_t =  \bepsilon_t, \quad t \in \Z
\end{equation}
with the same parameter  of interest $\bth := \text{vec}\A = (a_{11}, a_{21}, a_{12}, a_{22})^\prime = (0.2, -0.6, 0.3, 1.1)^\prime$ as in the motivating example of Section~\ref{subsec.glance} and spherical Gaussian, spherical $t_3$, skew-normal, skew-$t_3$, Gaussian mixture, and non-spherical Gaussian innovations, respectively. The skew-normal and skew-$t_3$ distributions are described in Appendix~\ref{skewsec}; the Gaussian mixture  is of the form 
\begin{equation}
\frac{3}{8}   {\cal N}(\bmu_1, \bSigma_1) + \frac{3}{8}   {\cal N}(\bmu_2, \bSigma_2) + \frac{1}{4} {\cal N}(\bmu_3, \bSigma_3),
\label{Eq. Mixture}
\end{equation}
with $\bmu_1 \! =\! (-5, 0)^\prime$,  $\bmu_2\! =\! (5, 0)^\prime$,  $\bmu_3\! =\! (0, 0)^\prime$, 
 $\vspace{1mm}\bSigma_1\! =\! 
 \left(\!
\begin{array}{cc}
7\! & 5 \\
5\! & 5
\end{array}\!\right)$, 
 $\bSigma_2 \! =\! 
 \left(\!
\begin{array}{cc}
7\!\! & \!\!-6 \\
\!-\!6 \!& \! 6
\end{array}\!\right)$, and 
 $\bSigma_3 \! =\! 
 \left(\!
\begin{array}{cc}
4\! & 0 \\
0\! & 3
\end{array}\!\right)\!.
$ 
A scatterplot of  $n=1000$ innovations drawn from this mixture is shown in Appendix~\ref{skewsec}. For the non-spherical Gaussian case, as in  the last panel of Table~\ref{Tab.n1000}\vspace{1mm},  we set the covariance matrix  to~$
\bSigma_4 \! =\! 
 \left(\!
\begin{array}{cc}
5 & 4 \\
4 & 4.5
\end{array}\!\right)\!
$,  
so that the bivariate innovation exhibits a large positive correlation~of~$0.843$.

For each of these  innovation densities, we generated $N=300$ Monte Carlo realizations---larger values of $N$ did not show significant changes--- of the stationary solution of \eqref{VAR}, of length $n=1000$ ($n$ ``large": Section~\ref{Sect: MClarge}) and $n=300$ ($n$ ``small": Section~\ref{Sect: MCsmall}), respectively. For each realization, we computed the QMLE,  our R-estimators (sign test, Spearman, and van der Waerden scores), and, for the purpose of comparison, 
 the QMLE based on $t_5$ likelihood (although  inconsistent, QMLEs based on $t$-distribution are a popular choice in the time series literature) and  the reweighted multivariate least trimmed squares estimator (henceforth, RMLTSE)  of Croux and Joossen (2008). The boxplots and  tables of bias and mean squared errors below allow for a comparison of the finite-sample performance of  our R-estimators and those routinely-applied M-estimators. 

Throughout,   QMLEs were computed from the {\tt MTS} package in R program,  RMLTSEs   were computed from the function {\tt varxfit} in the package {\tt rmgarch} in R program, and $t_5$-QMLEs were obtained by  minimizing the negative log-likelihood function using {\tt optim} function in R program.
The R-estimators were obtained via the one-step procedure as in the algorithm described in 
  Appendix~\ref{secalg}---five iterations for~$n=1000$,  ten iterations for $n=300$.

\begin{table}[htbp]
\caption{The estimated bias ($\times 10^3$), MSE ($\times 10^3$), and overall MSE ratios of the QMLE,   $t_5$-QMLE , RMLTSE, and R-estimators  under various innovation densities.  The sample size is~$n = 1000$; $N = 300$ replications.}\label{Tab.n1000}
\centering
\small
\begin{tabular}{ccccc|cccc|c}
\hline
            & \multicolumn{4}{c}{Bias ($\times 10^3$)}           & \multicolumn{4}{c}{MSE  ($\times 10^3$)}  &  MSE ratio \\  \hline
                    & $a_{11}$          & $a_{21}$ & $a_{12}$ & $a_{22}$ & $a_{11}$ & $a_{21}$ & $a_{12}$ & $a_{22}$ &                   \\ \hline
                    (Normal)           &                &        &        &        &       &       &       &       &       \\
QMLE               & -0.484         & -0.054 & 0.201  & -1.571 & 0.769 & 0.679 & 0.173 & 0.195 &       \\
$t_5$-QMLE          & -0.547         & -0.132 & 0.429  & -1.582 & 0.833 & 0.751 & 0.190 & 0.210 & 0.916 \\
RMLTS              & -0.629         & -0.992 & 0.424  & -1.334 & 0.843 & 0.760 & 0.193 & 0.215 & 0.903 \\
vdW                & -0.662         & -0.434 & 0.504  & -1.833 & 0.780 & 0.688 & 0.178 & 0.205 & 0.982 \\
Spearman           & -1.263         & -0.979 & 1.274  & -2.134 & 0.810 & 0.728 & 0.189 & 0.216 & 0.935 \\
Sign               & -0.372         & -0.600 & 1.545  & -2.642 & 1.314 & 1.141 & 0.305 & 0.310 & 0.592 \\ \hline
(Mixture)          &                &        &        &        &       &       &       &       &       \\
QMLE               & -1.318         & -0.476 & 2.907  & -0.103 & 0.839 & 0.153 & 0.342 & 0.056 &       \\
$t_5$-QMLE          & -0.852         & 0.483  & 4.820  & 0.248  & 4.420 & 0.261 & 1.641 & 0.156 & 0.215 \\
RMLTS              & -0.703         & 0.268  & 3.166  & -0.116 & 0.876 & 0.168 & 0.351 & 0.069 & 0.949 \\
vdW                & -1.111         & -0.465 & 2.347  & -0.883 & 0.316 & 0.085 & 0.149 & 0.042 & 2.346 \\
Spearman           & -0.841         & -0.539 & 2.338  & -0.791 & 0.291 & 0.088 & 0.140 & 0.041 & 2.480 \\
Sign               & -1.691         & 0.048  & 5.256  & -1.425 & 1.332 & 0.149 & 0.564 & 0.074 & 0.656 \\ \hline
\multicolumn{2}{l}{(Skew-normal)} &        &        &        &       &       &       &       &       \\
QMLE               & -0.992         & 1.800  & 0.651  & -2.108 & 0.804 & 1.039 & 0.281 & 0.311 &       \\
$t_5$-QMLE          & -0.378         & 2.588  & -0.083 & -2.827 & 1.000 & 1.294 & 0.365 & 0.397 & 0.797 \\
RMLTS              & -0.519         & 1.515  & 0.172  & -2.383 & 0.835 & 1.111 & 0.295 & 0.333 & 0.946 \\
vdW                & -1.031         & 0.990  & 0.811  & -2.520 & 0.668 & 0.998 & 0.214 & 0.291 & 1.122 \\
Spearman           & -1.295         & 0.625  & 0.848  & -2.171 & 0.694 & 1.032 & 0.222 & 0.294 & 1.086 \\
Sign               & -1.608         & 0.888  & 1.346  & -4.039 & 1.360 & 1.673 & 0.415 & 0.519 & 0.614 \\ \hline
\multicolumn{2}{l}{(Skew-$t_3$)}    &        &        &        &       &       &       &       &       \\
QMLE               & -2.242         & -2.055 & 0.763  & 0.213  & 1.022 & 0.856 & 0.379 & 0.336 &       \\
$t_5$-QMLE          & 3.032          & 1.865  & -2.078 & -2.134 & 1.062 & 0.714 & 0.707 & 0.463 & 0.880 \\
RMLTS              & -0.186         & 0.357  & -0.613 & -1.373 & 0.517 & 0.483 & 0.278 & 0.237 & 1.711 \\
vdW                & -1.250         & 0.170  & 1.100  & -2.014 & 0.432 & 0.526 & 0.151 & 0.204 & 1.973 \\
Spearman           & -1.022         & 0.119  & 1.018  & -1.891 & 0.438 & 0.537 & 0.149 & 0.204 & 1.952 \\
Sign               & -1.515         & -0.532 & 1.065  & -3.410 & 0.966 & 1.095 & 0.333 & 0.501 & 0.895 \\ \hline
($t_3$)            &                &        &        &        &       &       &       &       &       \\
QMLE               & -3.558         & -0.210 & 2.092  & -0.967 & 0.844 & 0.671 & 0.205 & 0.185 &       \\
$t_5$-QMLE          & -2.185         & -0.433 & 1.332  & -0.613 & 0.386 & 0.349 & 0.098 & 0.095 & 2.052 \\
RMLTS              & -2.473         & -0.510 & 1.313  & -0.691 & 0.438 & 0.384 & 0.108 & 0.106 & 1.836 \\
vdW                & -2.680         & -1.937 & 2.393  & -1.053 & 0.602 & 0.557 & 0.143 & 0.135 & 1.325 \\
Spearman           & -2.880         & -2.014 & 2.663  & -1.033 & 0.640 & 0.589 & 0.150 & 0.142 & 1.253 \\
Sign               & -2.204         & -3.916 & 1.996  & 0.104  & 0.784 & 0.681 & 0.201 & 0.179 & 1.032 \\ \hline
\multicolumn{2}{l}{(Non-spherical)} &        &        &        &       &       &       &       &       \\
QMLE               & 0.513          & 2.682  & -0.572 & -2.756 & 1.962 & 1.705 & 1.314 & 1.115 &       \\
$t_5$-QMLE          & 0.992          & 3.953  & -0.154 & -3.008 & 3.105 & 2.618 & 2.013 & 1.696 & 0.646 \\
RMLTS              & 0.077          & 2.834  & 0.043  & -2.473 & 2.118 & 1.886 & 1.400 & 1.181 & 0.926 \\
vdW                & -0.335         & 3.327  & 0.156  & -4.017 & 2.597 & 2.273 & 1.386 & 1.212 & 0.816 \\
Spearman           & -0.373         & 3.361  & 0.487  & -3.853 & 2.562 & 2.268 & 1.411 & 1.222 & 0.817 \\
Sign               & 4.157          & 8.485  & -4.713 & -8.645 & 6.717 & 5.955 & 3.300 & 2.582 & 0.329 \\ \hline
 \end{tabular}
\end{table}

\subsection{Large sample results}\label{Sect: MClarge}

The averaged bias and MSE  of each estimator for $n=1000$ (factorizing into $n_Rn_S=25\times 40$)  are summarized in Table~\ref{Tab.n1000}, where ratios of the sums (over the four parameters) of the MSEs of the QMLE over those of each of the other estimators are also reported. Because of space constraints, the corresponding  boxplots under the skew-normal (Figure~\ref{boxSkewNorm}), skew-$t_3$ (Figure~\ref{boxSkewt3})), spherical~$t_3$ (Figure~\ref{boxt3}) and non-spherical Gaussian (Figure~\ref{boxell}) innovations are provided in Appendix~\ref{AppD31}.

Inspection of Table~\ref{Tab.n1000} reveals that under asymmetric innovation densities (mixture, skew-normal and skew-$t_3$), the vdW and Spearman R-estimators dominate  the other three M-estimators, with significant efficiency gains under the mixture and skew-$t_3$ distributions. One may wonder what happens if  asymmetry is removed and only the heavy-tail feature is kept. The MSE ratios under the spherical $t_3$ distribution answer this question:  the R-estimators  still outperform the QMLE. 
Recalling that asymptotic optimality can be achieved by our R-estimators under spherical densities, it would be interesting to investigate their performance under a non-spherical distribution with large correlation. The MSE ratios under the non-spherical Gaussian distribution show that the vdW and Spearman R-estimators lose only small efficiency with respect to the QMLE: as we have observed in the motivating example of Section~\ref{subsec.glance}, the good performance of the R-estimators under asymmetric distributions is not obtained at the expense of a loss of accuracy under the symmetric ones.\vspace{-1.5mm}

\subsection{Small sample results } \label{Sect: MCsmall}

A major advantage of R-estimation over other semiparametric procedures is the fact that it does not require any kernel density estimation, which allows for 
 applying our method also in relatively small samples. To gain understanding on that aspect, we consider the same setting as in  Section \ref{Sect: MClarge}, but with sample size $n=300$ (an order of magnitude which is quite common  in real-data applications:  
see e.g. Section \ref{empirsec}) factorizing into $n_R n_S=15\times 20$. Due to space constraints, the results are shown  in Appendix~\ref{AppD32}, where Table~\ref{Tab.n300} provides the averaged bias, MSE and overall MSE ratios   of all estimators under various innovation densities; all results 
in line with those in Table~\ref{Tab.n1000}. The corresponding boxplots are displayed (still in Appendix~\ref{AppD32}) in Figures~\ref{boxmixture300}-\ref{boxt3n300} and confirm the superiority over the QMLE,  
also in  small samples, of our R-estimators  under non-elliptical innovations:  even in small samples (with~$n_R$  and~$n_S$ as small as~15 and~20), our R-estimators outperform the 
QMLE under non-Gaussian innovations, while performing equally well under Gaussian conditions.\vspace{-1.5mm}


\subsection{Resistance to outliers}\label{subsec.AO}

We also   investigated the robustness properties  of our estimators and, more particularly, their resistance to additive outliers (AO).  Following Maronna et al.~(2019), 
 we first generated  Gaussian VAR(1) realizations~$\{\X_t\}$ of~(\ref{VAR}) ($n=300$); then, adding the AO,   obtained the  contaminated 
observations~$\{\X^*_t=  \X_t + I(t = h) \bxi\}$, where   $h$ and $\bxi$ denote the location and size of the AO, respectively. We set $h$ in order to have 
$5\%$ equally spaced AOs and put~$\bxi = (4, 4)^\prime$. The parameter $\bth$ remains the same as in the previous settings.  The contaminated 
observations are demeaned prior to estimation procedures.   Figure~\ref{AOboxn300} provides 
the boxplots of our three R-estimators (sign, Spearman, vdW) along with the boxplots of the QMLE, $t_5$-QMLE, and  RMLTSE. Comparing those boxplots and the figures shown at the bottom of Table~\ref{Tab.n300} (Appendix~\ref{AppD32})   with the uncontaminated ones of Figure~\ref{box3Mix} reveals  that   AO have a severe impact on the 
QMLE but a much less significant one on the   R-estimators. For $a_{12}$ and $a_{22}$, the bias and variance of the R-estimates are 
comparable to those of the RMLTSE, with the latter displaying a much larger bias for 
$a_{11}$ and $a_{21}$. Overall,  we remark that for the estimation of all parameters, vdW and Spearman R-estimators feature less variability than the $t_5$-QMLE and  RMLTSE. 

To gauge the trade-off between robustness and efficiency, we compare the MSE ratio of RMLTSE to the MSE ratios of the R-estimators under Gaussian innovation density, as displayed in Table \ref{Tab.n1000} (see top panel)---see also Table \ref{Tab.n300} in Appendix D. 
The vdW and Spearman R-estimators exhibit MSE ratios equal to 0.982 and 0.935, respectively,  which corresponds to a smaller efficiency loss than for the   RMLTSE (MSE ratio equal to 0.903)---suggesting that the  trade-off between robustness and efficiency is more favorable for vdW and Spearman R-estimators than for the RMLTSE.

\begin{center}
\begin{figure}[h!]
\caption{Boxplots of the QMLE, $t_5$-QMLE, RMLTSE, and R-estimators (sign test, Spearman, and van der Waerden scores) under   Gaussian innovations in the presence of additive outliers  (sample size $n = 300$;  $N = 300$ replications). The horizontal red line represents the actual parameter value.}
\begin{center}
\includegraphics[width=1\textwidth, height=0.5\textwidth]{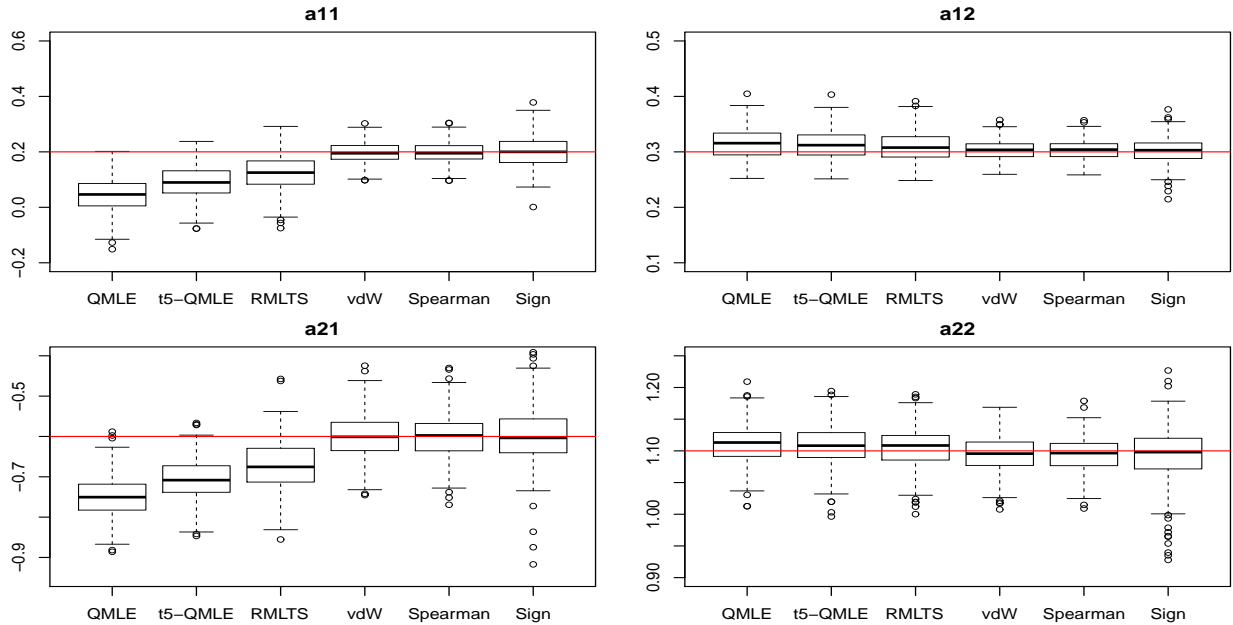}
\end{center}
\label{AOboxn300}
\end{figure}
\end{center}
\vspace{-8mm}

\subsection{Further simulation results}\label{dim3sec}

We also considered a trivariate ($d=3$) VAR(1), with parameter of interest $\bth 
\in \mathbb{R}^9$, Gaussian and Gaussian mixture innovations, and 
 sample size $n=1000$. The results, which confirm the bivariate ones, can be found in Appendix~\ref{dim3App} (Figures 15-16), along with details about the simulation design.

\section{Real-data example}\label{empirsec}

To conclude, we illustrate the applicability and good performance of our R-estimators in a real-data macroeconomic example.  
We consider the seasonally adjusted series of monthly housing starts (\texttt{Hstarts}) and the $30$-year conventional mortgage rate (\texttt{Mortg}---no need for seasonal adjustment) in the US  from January~1989 to January~2016, with a sample size $n=325$  each
(both  series  are freely available on the Federal Reserve Bank of Saint Louis website, to which we refer for details). The same time series were studied by Tsay (2014, Section 3.15.2). 
Following Tsay, we analyze the 
differenced  series;  Figure \ref{TS_US}   displays  plots of their demeaned  differences. 
While the \texttt{Mortg}   series seems to be driven by  skew innovations (with large positive values   more likely than the negative ones), the \texttt{Hstarts} series looks   more symmetric  about zero. Visual inspection suggests the presence of significant auto- and cross-correlations, as    expected from   macroeconomic theory.\vspace{-4mm}

\begin{center}
\begin{figure}[h!]
\caption{Plots of demeaned differences 
  of the monthly housing starts (measured in thousands of units; left panel) and the $30$-year conventional mortgage rate (in percentage; right panel) in the US, from January 1989 through January 2016.\vspace{-0mm}
}
\begin{center}
\includegraphics[scale=.5]{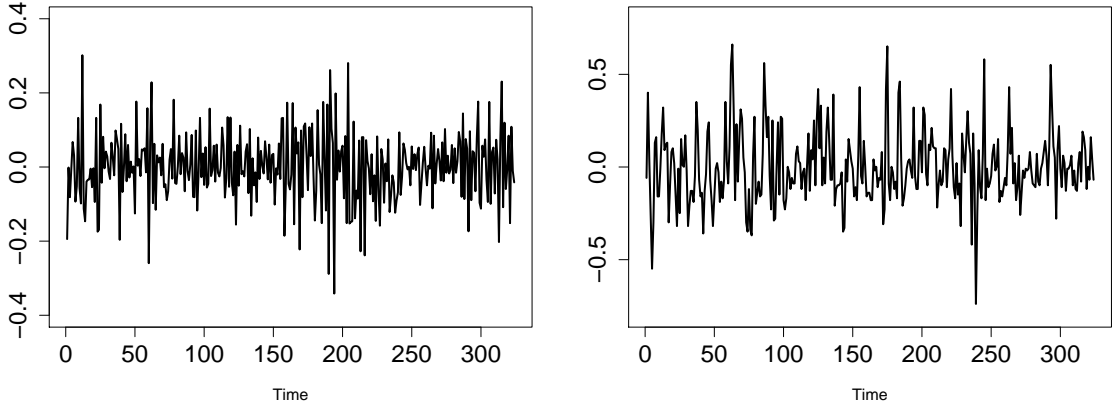}\vspace{-7mm}
\end{center}
\label{TS_US}
\end{figure}
\end{center}

The AIC criterion  selects a VARMA($3, 1$) model, the parameters of which we estimated using the benchmark QMLE (see e.g.~Tsay  (2014), Chapter 3) and our R-estimators (sign, Spearman, and van der Waerden).  The QMLE-based multivariate Ljung-Box test does not reject the model  at  nominal level~1\%.  We report the estimates (along with their standard error, SE, in parentheses) in Table~\ref{emp} in the online Appendix~\ref{tablesec}.  Spotting the differences in Table~\ref{emp} is all but simple, even though some look quite significant (see, for instance, the QMLE and R-estimates of $\A_{21}$ and~$\A_{22}$) and analyzing them is even more difficult. 

Impulse response functions (IRFs) are easier to read and interpret; they are widely applied in macroeconometrics---see e.g. Tsay (2014) for a book-length description. Intuitively the IRFs express the effect of changes in one variable on another variable in multivariate time series analysis. In the VARMA case, the IRF is obtained using a MA representation: see Tsay (2014, Section 3.15.2) and Appendix E.2 in the supplementary material for  mathematical details. 
 In Figures~\ref{IRF_VARMA31_Orig}, we plot the estimated IRFs resulting from the QMLE and R-estimators. The top plots show the response of \texttt{Hstarts} to  its own shocks (left panel)  and to the shocks of \texttt{Mortg}; the bottom panels show  the response of \texttt{Mortg} to  its own shocks (right panel)  and to the shocks of \texttt{Hstarts}. 
 Looking at the plots, we see that all IRFs have similar patterns. For instance, for all estimators, the top left panel illustrates that the IRF of the \texttt{Hstarts} to its own shocks have two consecutive increases after two initial drops. However, the decay of the QMLE-based IRF is uniformly faster than  the R-estimator-based ones. Also, the other plots exhibit  a more pronounced decay in the QMLE-based IRFs. Thus, R-estimators suggest a more persistent impact of the shocks:  decision makers should be aware of this inferential aspect in the implementation of their economic policy. \vspace{-5mm}


\begin{center}
\begin{figure}[H]
\caption{Plots of estimated impulse response functions    of the VARMA($3, 1$) model for the differenced \texttt{Hstarts} (top panels) and \texttt{Mortg}  (bottom panels) data, based on the QMLE and the R-estimators.} \hspace{1.25cm}
\includegraphics[scale=0.45]{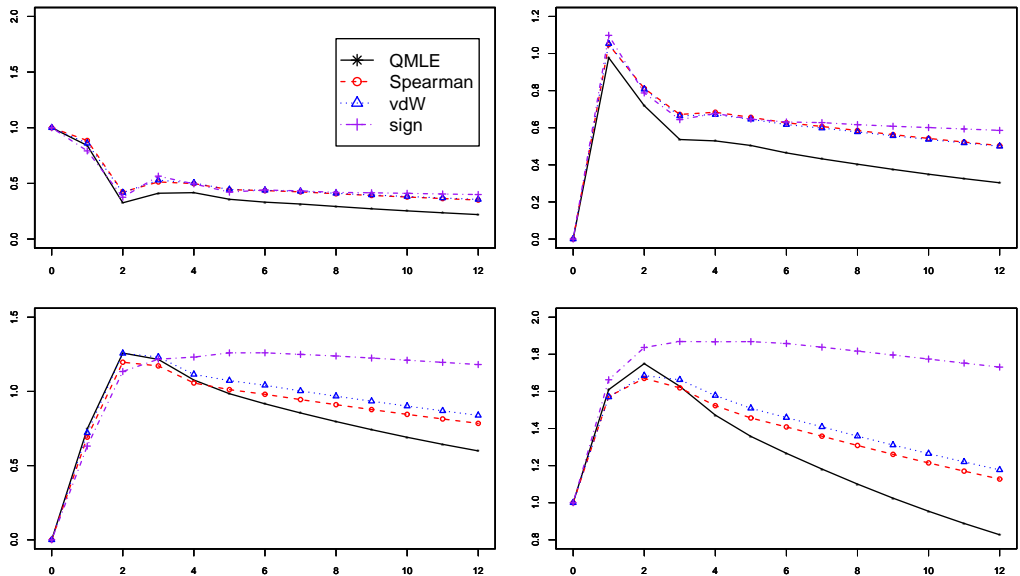}
\label{IRF_VARMA31_Orig}
\end{figure}\vspace{-5mm}
\end{center}

\section{Conclusions and perspectives} \label{Sec:concl}

We define a class of R-estimators based  on the novel concept of center-outward ranks and signs, itself closely related to the theory of optimal measure transportation. Monte Carlo experiments show that these estimators significantly outperform the classical QMLE under skew multivariate innovations, even when the validity conditions for the latter are satisfied.   In a companion paper, 
we study the performance of the corresponding  rank-based tests for VAR models, and, more particularly, propose a center-outward Durbin-Watson test for multiple-output regression and a test of VAR($p_0$) against VAR($p_0 + 1$) dependence. Our  methodology is not limited to the VARMA case, though; its extension to nonlinear multivariate models, like the dynamic conditional correlation model of Engle (2002), is the subject of ongoing research.  

\color{black}


\appendix

\newpage


\section{Technical material: algebraic preparation}\label{AppLAN}
Denote by $\G_u$ and~$\H_u$, $u \in \Z$ the {\it Green's matrices} associated with the linear difference operators~$\A(L)$ and $\B(L)$   in Section~\ref{secdef}: those matrices  
 are defined as the solutions of  the homogeneous   linear recursions 
$$\A(L)\G_u=\G_u - \sum_{i=1}^{p} \A_i   \G_{u-i} = {\bf 0} 
 \quad\text{and}\quad\B(L)\H_u=\sum_{i=0}^{q} \B_i   \H_{u-i}   = {\bf 0},
 \quad u\in\mathbb{Z}$$
 with initial values $\I_d, {\bf 0},\ldots ,{\bf 0}$ at $u=0,-1,\ldots ,-p+1$ and $u=0,-1,\ldots ,-q+1$, respectively. 
 Then,  the residual process $\{\ZZ_t^{(n)}(\bth); 1 \leq t \leq n \}$ has the  representation
\begin{align}
\ZZ_t^{(n)}(\bth) =& \sum_{i=0}^{t-1} \sum_{j=0}^p \H_i \A_j \X^{(n)}_{t-i-j} \nonumber \\
& + \begin{bmatrix}
\H_{t+q-1} & \cdots & \H_t
\end{bmatrix}
\begin{bmatrix}
\I_d & \0 &\cdots & \0 \\
\B_1 & \I_d & \cdots & \0 \\
\vdots & \vdots &\ddots & \vdots \\
\B_{q-1} & \B_{q-2} & \cdots &\I_d
\end{bmatrix}
\begin{bmatrix}
\bepsilon_{-q+1} \\
\vdots \\
\bepsilon_0
\end{bmatrix} \label{Zt.recursive}
\end{align}
(see Hallin (1986), Garel and Hallin (1995), or  Hallin and Paindaveine~(2004)).

Assumption (A1) ensures the exponential decrease of $\{\Vert \H_u\Vert , u\in \N\}$ as $u\to\infty$. Specifically, there exists some $\varepsilon > 0$ such that $\Vert \H_u\Vert  (1+\varepsilon)^u$ converges to $0$ as $u \rightarrow \infty$. This also holds for the Green matrices~$\G_u$ associated with the operator $\A (L)$. It follows that 
  the initial values~$\{\bepsilon_{-q+1}, \ldots , \bepsilon_0\}$ and~$\{\X_{-p + 1}, \ldots , \X_{0}\}$ in (\ref{Zt.recursive}),  which are typically unobservable, have no asymptotic  influence on the residuals nor any asymptotic results. Therefore, they all can safely be set to zero in the sequel. This allows us to invert the AR and MA polynomials, and to define the Green matrices $\G_u$ and~$\H_u$ as the matrix coefficients of the inverted operators~$(\A (L))^{-1}$ and $(\B (L))^{-1}$:
$$\sum_{u=0}^\infty \G_u z^u := \left(  \I_d - \sum_{i=1}^p  \A_i z^i\right)^{-1}\quad\text{and}\quad\sum_{u=0}^\infty \H_u z^u := \left(  \sum_{i=0}^q \B_i z^i\right)^{-1}, \quad z \in \C, |z| < 1.$$

Associated with an arbitrary $d$-dimensional linear difference operator 
 $\CC (L) := \sum_{i=0}^\infty \CC_i L^i$ (this of course includes operators of finite order $s$), define, for any integers $u$ and $v$, the $d^2 u \times d^2 v$ matrices
\[
\CC_{u, v}^{(l)} :=
\begin{bmatrix}
\CC_0 \otimes \I_d & \0 & \ldots  & \0 \\
\CC_1 \otimes \I_d & \CC_0 \otimes \I_d & \ldots  & \0 \\
\vdots & & \ddots & \vdots \\
\CC_{v-1} \otimes \I_d & \CC_{v-2} \otimes \I_d & \ldots  & \CC_0 \otimes \I_d \\
\vdots & & & \vdots \\
\CC_{u-1} \otimes \I_d & \CC_{u-2} \otimes \I_d & \ldots  & \CC_{u-v} \otimes \I_d 
\end{bmatrix}
\]
and
\[
\CC_{u, v}^{(r)} :=
\begin{bmatrix}
\I_d \otimes \CC_0  & \0 & \ldots  & \0 \\
\I_d \otimes \CC_1  &  \I_d \otimes \CC_0  & \ldots  & \0 \\
\vdots & & \ddots & \vdots \\
\I_d  \otimes \CC_{v-1} &  \I_d \otimes \CC_{v-2}  & \ldots  & \I_d \otimes \CC_0  \\
\vdots & & & \vdots \\
\I_d \otimes \CC_{u-1}  &  \I_d \otimes \CC_{u-2}  & \ldots  &  \I_d  \otimes \CC_{u-v} 
\end{bmatrix}
.
\]
Write $\CC_u^{(l)}$ for $\CC_{u, u}^{(l)}$ and $\CC_u^{(r)}$ for $\CC_{u, u}^{(r)}$. With this notation, note that $\G_u^{(l)}, \G_u^{(r)}, \H_u^{(l)}$, and~$\H_u^{(r)}$ are the inverses of $\A_u^{(l)}, \A_u^{(r)}, \B_u^{(l)}$ and $\B_u^{(r)}$, respectively. Denoting by $\CC_{u, v}^{\prime (l)}$ and $\CC_{u, v}^{\prime (r)}$ the matrices associated with the transposed operator $\CC^\prime (L) := \sum_{i=0}^\infty \CC^\prime_i L^i$, we  have~$\G_u^{\prime (l)}~\!=~\!(\A_u^{\prime (l)})^{-1}$, $\H_u^{\prime (l)} = (\B_u^{\prime (l)})^{-1}$, and so on. Define the $d^2 (p+q) \times d^2 (p+q)$ matrix 
\begin{equation}\label{defM}
\M_{\bth} := (\G^{\prime (l)}_{p+q, p} \vdots \H^{\prime (l)}_{p+q, q}):
\end{equation}
 under Assumption (A1), 
 $\M_{\bth}$ is of full rank. 

Also, consider the operator $\D (L) := \I_d + \sum_{i=1}^{p+q} \D_i L^i$ (note that $\D (L)$ and most quantities defined below depends on $\bth$; for simplicity, however, we are dropping this reference to $\bth$), where
\[
\begin{bmatrix}
\D_1^\prime \\
\vdots \\
\D_{p+q}^\prime
\end{bmatrix}
:= -
\begin{bmatrix}
\G_q & \G_{q-1} & \ldots  & \G_{-p+1} \\
\G_{q+1} & \G_{q} & \ldots  & \G_{-p+2} \\
\vdots & & \ddots & \vdots \\
\G_{p+q-1} & \G_{p+q-2} & \ldots  & \G_{0} \\
\H_p & \H_{p-1} & \ldots  & \H_{-q+1} \\
\H_{p+1} & \H_{p} & \ldots  & \H_{-q+2} \\
\vdots & & \ddots & \vdots \\
\H_{p+q-1} & \H_{p+q-2} & \ldots  & \H_{0}
\end{bmatrix}
^{-1}
\begin{bmatrix}
\G_{q+1} \\
\vdots \\
\G_{p+q} \\
\H_{p+1} \\
\vdots \\
\H_{p+q}
\end{bmatrix}
\]
(recall that $\G_{-1} = \G_{-2} = \cdots = \G_{-p+1} = \0 $ and $ \H_{-1} = \H_{-2} = \cdots = \H_{-q+1}=\0$).
 Let~$\{\bpsi_t^{(1)}, \ldots , \bpsi_t^{(p+q)} \}$ be a set of $d \times d$ matrices forming a fundamental system of solutions of the homogeneous linear difference equation associated with $\D (L)$. Such a system can be obtained from the Green matrices of the operator $\D (L)$ (see, e.g., Hallin 1986). Defining
\[
\bar{\bpsi}_m (\bth) :=
\begin{bmatrix}
\bpsi_{1}^{(1)} & \ldots  & \bpsi_{1}^{(p+q)} \\
\bpsi_{2}^{(1)} & \ldots  & \bpsi_{2}^{(p+q)} \\
\vdots & & \vdots \\
\bpsi_{m}^{(1)} & \ldots  & \bpsi_{m}^{(p+q)}
\end{bmatrix}
\otimes \I_d ,
\]
the {\it Casorati matrix} $\mathbf{C}_{\bpsi}$ associated with $\D (L)$ is $\bar{\bpsi}_{p+q}$. Finally, let
\begin{equation}
\P_{\bth} := \mathbf{C}_{\bpsi}^{-1}
\quad\text{and}\quad 
\Q^{(n)}_{\bth} := \H_{n-1}^{(r)} \B_{n-1}^{\prime (l)}  \bar{\bpsi}_{n-1}.\label{defPQ}
\end{equation}


\section{Proofs }\label{Tech.proofs} 

This appendix gathers the proofs of all mathematical results. Throughout, we consider $f\in {\cal F}_d$ (the family of densities introduced in Section 2) and  assume that, for all~$c \in \R^+$, there exist $b_{c; f}$ and $a_{c; f}$ in $\mathbb{R}$ such that 
$0<b_{c; f}  \leq a_{c; f}<\infty$ and $ b_{c; f} \leq f(\x) \leq a_{c; f}$
 for~$\Vert \x\Vert  \leq c$. \\

\noindent  \textbf{Proof of Proposition \ref{Prop.LAN1}.} 

The LAN result is essentially the same as in Garel and Hallin (1995, (LAN~2) in their Proposition 3.1) and, 
moving along the same lines as in  the proof of Proposition 1 in Hallin and Paindaveine (2004), 
we obtain the form \eqref{Delta} of $\bDelta^{(n)}_{f} (\bth)$.
The form of the asymptotic covariance matrix $\bLam_{f} (\bth)$ 
and its  finiteness  
easily follow from applying Lemma 4.12 in Garel and Hallin (1995). Details are left to the reader. \cqfd
%

\vspace{0.3cm}

To prove Propositions \ref{asy0} and \ref{asy}, we first need to  establish the asymptotic normality, under~${\rm P}^{(n)}_{\bth ;f}$ and~${\rm P}^{(n)}_{\bth + n^{-1/2}\btau ;f}$, of the rank-based $\tenq{\bDelta}^{(n)}_{{J}_1, {J}_2}(\bth)$.  As in the univariate case, due to the fact that the ranks are not mutually independent, the asymptotic normality of a rank statistic does not follow from classical central-limit theorems. The approach we are adopting here is inspired from H\' ajek, and consists in establishing an asymptotic representation result for the rank-based statistic under study---namely, its asymptotic equivalence with a sum of independent  variable which are no longer rank-based---then proving the asymptotic normality of   the latter. This is achieved here in a series of lemmas: Lemma~B.1 deals with the asymptotic normality of $(n-i)^{1/2}\text{vec} (\bar{\bGamma}_{i, {J}_1, {J}_2}^{(n)}(\bth))$, a corollary of which is the asymptotic normality of  the truncated versions $\bar{\bDelta}^{(n)}_{m, {J}_1, {J}_2}(\bth)$ of $\bar{\bDelta}^{(n)}_{{J}_1, {J}_2}(\bth)$;  Lemma~B.3 provides the asymptotic representation of~$\text{vec} (\tenq{\bGamma}_{i, {J}_1, {J}_2}^{(n)}(\bth))$ by $\text{vec} (\bar{\bGamma}_{i, {J}_1, {J}_2}^{(n)}(\bth))$; the asymptotic representation of $\tenq{\bDelta}^{(n)}_{{J}_1, {J}_2}(\bth)$ by $\bar{\bDelta}^{(n)}_{{J}_1, {J}_2}(\bth)$ and their asymptotic normality are obtained in Lemma~B.4.  The proofs  of Propositions~\ref{asy0} and \ref{asy} follow.

Let us start with the asymptotic normality of $(n-i)^{1/2}\text{vec} (\bar{\bGamma}_{i, {J}_1, {J}_2}^{(n)}(\bth))$. 

\begin{lemma}\label{asy.Gami.bar}
Let Assumptions (A1), (A2), and (A3) hold. Then, for any positive integer $i$,  the vector $(n-i)^{1/2}\text{\rm vec} (\bar{\bGamma}_{i, {J}_1, {J}_2}^{(n)}(\bth))$  in (\ref{barGam}) is asymptotically normal 
with mean $\0$ under  ${\rm P}^{(n)}_{\bth ;f}$,  mean~$
\K_{{J}_1, {J}_2, {f}} \Q_{i, \bth} \P_{\bth} \M_{\bth} \btau
$ 
under ${\rm P}^{(n)}_{\bth + n^{-1/2}\btau ;f}$, and   covariance   $d^{-2}\sigma^2_{{J}_1}\sigma^2_{{J}_2} \I_{d^2}$ under both.
\end{lemma}

\begin{proof} Since  $L^{(n)}_{\bth + n^{-1/2}\btau/\bth; {f}}=\btau^\prime \bDelta^{(n)}_{f}(\bth)-\frac{1}{2} \btau^\prime \bLam_{f} (\bth) \btau
 + o_{\rm P}(1)$, the joint asymptotic normality  of~$(n-i)^{1/2} \text{\rm vec} (\bar{\bGamma}_{i, {J}_1, {J}_2}^{(n)}(\bth))$ and~$L^{(n)}_{\bth + n^{-1/2}\btau/\bth; {f}}$ under~${\rm P}^{(n)}_{\bth ;f}$  follows, via the classical Wold-Cram\' er argument, from the asymptotic normality   of  
\begin{equation*}
 N^{(n)}_{{\boldsymbol\alpha},\beta}:=
(n-i)^{1/2}{\boldsymbol\alpha}^\prime \text{\rm vec} (\bar{\bGamma}_{i, {J}_1, {J}_2}^{(n)}(\bth)) + \beta \btau^\prime \bDelta^{(n)}_{f}(\bth)
\end{equation*}
 for arbitrary ${\boldsymbol\alpha}\in\mathbb{R}^{d^2}$ and $\beta\in\mathbb{R}$. 
 Since~$\ZZ_1^{(n)},\ldots , \ZZ_n^{(n)}$ are i.i.d.~and $\F_{{\pms},t}:=\F_{\pms}(\ZZ^{(n)}_t)$ is uniform over the unit ball,~${N}^{(n)}_{{\boldsymbol\alpha},\beta}$ is a sum of    martingale differences.  If it is uniformly square-integrable, with  finite variance~$C^{(n)}_{{\boldsymbol\alpha},\beta}$, say, such that~$\lim_{n\to\infty}C^{(n)}_{{\boldsymbol\alpha},\beta}=:C_{{\boldsymbol\alpha},\beta}$ exists and is finite, the   martingale central limit theorem applies, and $ N^{(n)}_{{\boldsymbol\alpha},\beta}$  is asymptotically normal with mean~$0$ and     variance~$C_{{\boldsymbol\alpha},\beta}$. Now, the variance of ${N}^{(n)}_{{\boldsymbol\alpha},\beta}$ takes the form
  \begin{align*}C^{(n)}_{{\boldsymbol\alpha},\beta}
    &=(n-i) {\boldsymbol\alpha}^\prime  {\mathrm{Var}}\big( \text{\rm vec} (\bar{\bGamma}_{i, {J}_1, {J}_2}^{(n)}(\bth)) 
     \big) {\boldsymbol\alpha}  \\
    &\qquad +2\beta {\boldsymbol\alpha}^\prime (n-i)^{1/2} {\mathrm{Cov}} \big( \text{\rm vec} (\bar{\bGamma}_{i, {J}_1, {J}_2}^{(n)}(\bth)),  \btau^\prime \bDelta^{(n)}_{f}(\bth) \big) \\
 &\qquad  + \beta^2   \btau^\prime {\mathrm{Var}} \big(\bDelta^{(n)}_{f}(\bth)  \big)  \btau .
 \end{align*}
 The entries of each $\bar{\bGamma}_{i, {J}_1, {J}_2}^{(n)}(\bth)$ are uniformly square-integrable. As for $\bDelta^{(n)}_{f}(\bth)$, it follows from Lemma~2.2 in Hallin and Werker~(2003) that, for any LAN family, a uniformly $p$th-order integrable version of the central sequence exists: without loss of generality, let us assume that  $\bDelta^{(n)}_{f}(\bth)$, for $p=2$, is one of them. The sequence   ${N}^{(n)}_{{\boldsymbol\alpha},\beta}$ thus has a limiting ${\cal N}(0, C_{{\boldsymbol\alpha},\beta})$ distribution provided that $\lim_{n\to\infty}C^{(n)}_{{\boldsymbol\alpha},\beta}=:C_{{\boldsymbol\alpha},\beta}$   exists and is finite. 
 

Due to the independence between  the signs $\S_{{\pms},t}:=\F_{{\pms},t} /\Vert  \F_{{\pms},t} \Vert $  and the    moduli  $\Vert  \F_{{\pms},t} \Vert $ (which follows from  the fact that $\F_{{\pms},t} \sim {\mathrm U}_d$),  and due to the fact that $\ZZ_1^{(n)},\ldots , \ZZ_n^{(n)}$ are~i.i.d., 
  \begin{align}
\lim_{n\to\infty}{\mathrm{Var}}\big( \text{\rm vec} (\bar{\bGamma}_{i, {J}_1, {J}_2}^{(n)}(\bth)) 
     \big)=&  \lim_{n\to\infty}\mathrm{E}\left\lbrace (n-i)\text{vec} \bar{\bGamma}_{i, {J}_1, {J}_2}^{(n)}(\bth) (\text{vec} \bar{\bGamma}_{i, {J}_1, {J}_2}^{(n)}(\bth))^\prime \right\rbrace \nonumber \\
=&\lim_{n\to\infty} (n-i)^{-1} \mathrm{E}\left\lbrace \left[  \sum_{t=i+1}^n {J}_1(\Vert  \F_{{\pms}, t} \Vert ) {J}_2(\Vert  \F_{{\pms}, t-i} \Vert )  \text{vec}(\S_{{\pms}, t} \S^\prime_{{\pms}, t-i}) \right] \right. \nonumber \\
&\qquad \qquad \qquad \times \left.  \left[  \sum_{t=i+1}^n {J}_1(\Vert  \F_{{\pms}, t} \Vert ) {J}_2(\Vert  \F_{{\pms}, t-i} \Vert )  \text{vec}(\S_{{\pms}, t} \S^\prime_{{\pms}, t-i}) \right]^\prime \right\rbrace \nonumber \\
=&  \frac{1}{d^2} \sigma^2_{{J}_1}\sigma^2_{{J}_2} \I_{d^2}, \label{cov.gambar}
\end{align}
where the last equation follows from the uniform distribution of $\S_{{\pms}, t}$ over  $\mathcal{S}_{d-1}$.  
Next, the uniform square-integrability of $\bDelta^{(n)}_{f}(\bth)$ and its asymptotic normality  in Proposition \ref{Prop.LAN1} yield 
\begin{align}
&  \underset{n\rightarrow \infty}{\lim} (n-i)^{1/2} {\mathrm{Cov}} \big( \text{\rm vec} (\bar{\bGamma}_{i, {J}_1, {J}_2}^{(n)}(\bth)),  \btau^\prime \bDelta^{(n)}_{f}(\bth) \big) \nonumber \\
&=  \underset{n\rightarrow \infty}{\lim}  {\mathrm E}  
\left[(n-i)^{1/2}\text{vec} (\bar{\bGamma}_{i, {J}_1, {J}_2}^{(n)}(\bth))  \btau^\prime \bDelta^{(n)}_{f}(\bth)  \right] \nonumber \\
&= \underset{n\rightarrow \infty}{\lim} {\mathrm E}  \left[(n-i)^{1/2}\text{vec} (\bar{\bGamma}_{i, {J}_1, {J}_2}^{(n)}(\bth))  \bGamma_{f}^{(n)\prime}(\bth)\right]   \Q_{\bth}^{(n)} \P_{\bth} \M_{\bth} \btau, \label{jointNorm1}
\end{align}
where the last equality follows from~(\ref{Delta}). Due to the independence of  $\ZZ_i^{(n)}$ and $ \ZZ_j^{(n)}$ for~$i\neq j$, only~$\bGamma_{i, f}^{(n)}(\bth)$ in~$\bGamma_{f}^{(n)}(\bth)$ is contributing to (\ref{jointNorm1}). Therefore,  using the block matrix form of~$\Q^{(n)}_{\bth} =
\big(
\Q_{1, \bth}^\prime 
\ldots 
\Q_{n-1, \bth}^\prime
\big)^\prime$, the expression in~(\ref{jointNorm1}) reduces to
\begin{equation}\label{jointNorm2}
\underset{n\rightarrow \infty}{\lim}  (n-i) {\mathrm E} \left[ \text{vec} (\bar{\bGamma}_{i, {J}_1, {J}_2}^{(n)}(\bth)) (\text{vec} (\bGamma_{i, {f}}^{(n)}(\bth)))^\prime \right] \Q_{i, \bth} \P_{\bth} \M_{\bth} \btau.
\end{equation}
From (\ref{Gamma1}),  we have
\begin{align}
& (n-i) {\mathrm E}  \left[ \text{vec} (\bar{\bGamma}_{i, {J}_1, {J}_2}^{(n)}(\bth)) (\text{vec} (\bGamma_{i, {f}}^{(n)}(\bth)))^\prime \right] \nonumber \\
&\quad = (n-i)^{-1} {\mathrm E}  \left\lbrace \left[  \sum_{t=i+1}^n {J}_1(\Vert  \F_{{\pms}, t} \Vert ) {J}_2(\Vert  \F_{{\pms}, t-i} \Vert )  \text{vec}(\S_{{\pms}, t} \S^\prime_{{\pms}, t-i}) \right] \left[  \sum_{t=i+1}^n \text{vec} (\bvp_{f}(\ZZ_t^{(n)})) \ZZ^\prime_{t-i}) \right]^\prime \right\rbrace  \nonumber \\
&\quad =  {\mathrm E}  \big[{J}_1(\Vert  \F_{{\pms}, t} \Vert ) {J}_2(\Vert  \F_{{\pms}, t-i} \Vert ) (\I_d \otimes \S_{{\pms}, t})  \S_{{\pms}, t-i} \ZZ^\prime_{t-i}  (\I_d \otimes \bvp^\prime_{f}(\ZZ_t^{(n)}))\big]    \label{jointNorm3}
\quad \end{align}
where the last two equalities follow from the independence of $\ZZ_1^{(n)}, \ldots, \ZZ_n^{(n)} $ and the uniform distribution of~$ \F_{{\pms}, t} \sim {\mathrm U}_d$. In view of (\ref{lik1}), (\ref{jointNorm1}), (\ref{jointNorm2}) and (\ref{jointNorm3}), we thus obtain
\begin{equation}\label{cov.Gamma.lik}
\underset{n\rightarrow \infty}{\lim} (n-i)^{1/2} {\mathrm{Cov}} \big( \text{\rm vec} (\bar{\bGamma}_{i, {J}_1, {J}_2}^{(n)}(\bth)),  \btau^\prime \bDelta^{(n)}_{f}(\bth) \big) =  \K_{{J}_1, {J}_2, {f}} \Q_{i, \bth} \P_{\bth} \M_{\bth} \btau.
\end{equation} 
Combining (\ref{cov.gambar}), (\ref{cov.Gamma.lik}) and the asymptotic normality of $\bDelta^{(n)}_{f}(\bth)$ in Proposition \ref{Prop.LAN1} yields, for arbitrary ${\boldsymbol\alpha}$ and $\beta$, 
\begin{equation}\label{Gamma.lik}
\underset{n\rightarrow \infty}{\lim} C^{(n)}_{{\boldsymbol\alpha},\beta} = {\boldsymbol\alpha}^\prime{\boldsymbol\alpha}d^{-2}\sigma^2_{{J}_1}\sigma^2_{{J}_2} + 2\beta {\boldsymbol\alpha}^\prime  \K_{{J}_1, {J}_2, {f}} \Q_{i, \bth} \P_{\bth} \M_{\bth} \btau + \beta^2 \btau^\prime \bLam_{f} (\bth) \btau.
\end{equation}

It follows that  $\big((n-i)^{1/2}\text{vec}^\prime (\bar{\bGamma}_{i, {J}_1, {J}_2}^{(n)}(\bth)),\ 
L^{(n)}_{\bth + n^{-1/2}\btau/\bth; {f}}\big)^\prime$, under ${\rm P}^{(n)}_{\bth ;f}$,  is asymptotically jointly normal, with mean $\left({\mathbf 0}^\prime ,  -\frac{1}{2} \btau^\prime \bLam_{f} (\bth) \btau \right)^\prime$ and covariance 
 \begin{align}
\begin{bmatrix}
d^{-2}\sigma^2_{{J}_1}\sigma^2_{{J}_2} \I_{d^2} &   \K_{{J}_1, {J}_2, {f}} \Q_{i, \bth} \P_{\bth} \M_{\bth} \btau \\
( \K_{{J}_1, {J}_2, {f}} \Q_{i, \bth} \P_{\bth} \M_{\bth} \btau)^\prime & \btau^\prime \bLam_{f} (\bth) \btau
\end{bmatrix} .
\end{align} 
The desired  result   then readily follows from applying   Le Cam's third Lemma.\end{proof}

Recall that $\T^{(n)}_{\bth} = \M_{\bth}^\prime  \P_{\bth}^\prime \Q_{\bth}^{(n)\prime}$. For any positive integer $m\leq n-1$, let 
\begin{equation}\label{barDelta}
\bar{\bDelta}^{(n)}_{m, {J}_1, {J}_2}(\bth) := \T^{(m+1)}_{\bth} \bar{\bGamma}_{{J}_1, {J}_2}^{(m,n)}(\bth),
\end{equation}
where $$
\bar{\bGamma}_{{J}_1, {J}_2}^{(m,n)}(\bth) := \big((n-1)^{1/2}(\text{vec} \bar{\bGamma}_{1, {J}_1, {J}_2}^{(n)}(\bth))^\prime, \ldots , (n-m)^{1/2}(\text{vec} \bar{\bGamma}_{m, {J}_1, {J}_2}^{(n)}(\bth)\big)^\prime)^\prime
:$$  
clearly,~$\bar{\bDelta}^{(n)}_{m, {J}_1, {J}_2}(\bth)$, it is the truncated version of $\bar{\bDelta}^{(n)}_{{J}_1, {J}_2}(\bth)$ defined in Section \ref{subsec.def}.
The   asymptotic normality of~$\bar{\bDelta}^{(n)}_{m, {J}_1, {J}_2}(\bth)$   follows from Lemma~\ref{asy.Gami.bar}  as a corollary.

\begin{corollary}\label{asy.S.bar}
Let Assumptions (A1), (A2), and (A3) hold.  Then, for any positive integer~$m$, the vector $\bar{\bDelta}^{(n)}_{m, {J}_1, {J}_2}(\bth)$  in (\ref{barDelta}) is asymptotically normal, with mean $\0$ under  ${\rm P}^{(n)}_{\bth ;f}$,  mean
\begin{equation}
\T^{(m+1)}_{\bth} (\I_m \otimes \K_{{J}_1, {J}_2, f}) \T^{(m+1)^\prime}_{\bth} \btau
\end{equation}
under ${\rm P}^{(n)}_{\bth + n^{-1/2}\btau ;f}$, and   covariance   $d^{-2}\sigma^2_{{J}_1}\sigma^2_{{J}_2} \T^{(m+1)}_{\bth} \T^{(m+1)^\prime}_{\bth}$ under both.
\end{corollary}

The following auxiliary  lemma, which follows along the same lines as Lemma 4 in Hallin and Paindaveine~(2002) and Lemma 5 in Hallin and Paindaveine (2004), will be useful in subsequent proofs.

\begin{lemma}\label{fourStats}
Let $i \in \{1, \ldots , n-1\}$ and $t, {t^\prime} \in \{i+1, \ldots , n\}$ be such that $t\neq {t^\prime}$. Assume that~$g: \R^{nd} = \R^d \times \cdots \times \R^d \rightarrow \R$ is even in all its arguments, and such that the expectation in~\eqref{expect}  below exists. Then, under ${\rm P}^{(n)}_{\bth ;f}$,
\begin{equation}\label{expect}
{\mathrm E}\big[g(\ZZ_1^{(n)}, \ldots , \ZZ_n^{(n)})(\P_t^\prime \Q_t) (\RR_{t-i}^\prime \S_{t^\prime-i})\big] = 0,
\end{equation}
where $\P_t, \Q_t, \RR_t$ and $\S_t$ are any four random vectors among $\S^{(n)}_{{\pms}, t}$ and $\S^{(n)}_{{\pms}, t} - \S_{{\pms}, t}$.
\end{lemma}

The next lemma establishes an asymptotic representation result for the rank-based cross-covariance matrices $ \tenq{\bGamma}_{i, {J}_1, {J}_2}^{(n)}(\bth)$ defined in  (\ref{tildeGam}) by showing their asymptotic equivalence with~$\bar{\bGamma}_{i, {J}_1, {J}_2}^{(n)}(\bth)$ defined in (\ref{barGam}). LAN implies 
that ${\rm P}^{(n)}_{\bth + n^{-1/2} \btau ;f}$ and  ${\rm P}^{(n)}_{\bth ;f}$ are mutually contiguous; \eqref{bartilGam.eqn} therefore holds under both. 
This  asymptotic representation in the H\' ajek style of a center-outward serial rank statistic extends to a multivariate setting a  univariate result first established by Hallin et al.~(1985).  

\begin{lemma}\label{bar.til.Gam}
Let Assumptions (A1), (A2), and (A3) hold.  Then, for any positive integer $i$, 
\begin{equation}\label{bartilGam.eqn}
\text{\rm vec}\left( \tenq{\bGamma}_{i, {J}_1, {J}_2}^{(n)}(\bth) - \bar{\bGamma}_{i, {J}_1, {J}_2}^{(n)}(\bth)\right) = o_{\rm P}(n^{-1/2})
\end{equation}
under ${\rm P}^{(n)}_{\bth ;f}$ and ${\rm P}^{(n)}_{\bth + n^{-1/2} \btau ;f}$, as $n \rightarrow \infty$.
\end{lemma}

\begin{proof}
Note that $
(n-i)^{1/2} ( \tenq{\bGamma}_{i, {J}_1, {J}_2}^{(n)}(\bth) - \bar{\bGamma}_{i, {J}_1, {J}_2}^{(n)}(\bth)) = (n-i)^{-1/2} ({\boldsymbol\delta}_1^{(n)} + {\boldsymbol\delta}_2^{(n)})$ 
where 
$${\boldsymbol\delta}_1^{(n)} :=  (n-i)^{-1/2} \sum_{t=i+1}^n \left( {J}_1(\frac{R^{(n)}_{{\pms}, t}}{n_R + 1}) {J}_2(\frac{R^{(n)}_{{\pms}, t-i}}{n_R + 1}) -  {J}_1(\Vert  \F_{{\pms}, t} \Vert ) {J}_2(\Vert  \F_{{\pms}, t-i} \Vert ) \right) \S^{(n)}_{{\pms}, t} \S^{(n)\prime}_{{\pms}, t-i} $$
and
$${\boldsymbol\delta}_2^{(n)} :=  (n-i)^{-1/2} \sum_{t=i+1}^n {J}_1(\Vert  \F_{{\pms}, t} \Vert ) {J}_2(\Vert  \F_{{\pms}, t-i} \Vert )
\left(\S^{(n)}_{{\pms}, t} \S^{(n)\prime}_{{\pms}, t-i} -   \S_{{\pms}, t} \S^\prime_{{\pms}, t-i}\right).  $$
It suffices to show that $\text{vec} ({\boldsymbol\delta}_1^{(n)})$ and $\text{vec} ({\boldsymbol\delta}_2^{(n)})$ both converge in quadratic mean to zero as~$n \rightarrow \infty$ under ${\rm P}^{(n)}_{\bth ;f}$.

Let $\Vert \cdot\Vert _{L^2}$ denote the $l_2$-norm. For ${\boldsymbol\delta}_1^{(n)}$, we make use of Lemma \ref{fourStats}, and we exploit the independence of the ranks $\{R^{(n)}_{{\pms}, t}; t=1,\ldots, n\}$ and the signs $\{\S^{(n)}_{{\pms}, t};  t=1,\ldots, n\}$ (see Hallin~(2017)). Given that $(\text{vec}\A)^\prime (\text{vec} \B) = \text{tr} (\A^\prime \B)$, we have
\begin{equation*}
\big\Vert \text{vec} ({\boldsymbol\delta}_1^{(n)})\big\Vert _{L^2}^2 = (n-i)^{-1} \sum_{t=i+1}^n {\mathrm E}\left[\left( {J}_1(\frac{R^{(n)}_{{\pms}, t}}{n_R + 1}) {J}_2(\frac{R^{(n)}_{{\pms}, t-i}}{n_R + 1}) -  {J}_1(\Vert  \F_{{\pms}, t} \Vert ) {J}_2(\Vert  \F_{{\pms}, t-i} \Vert )\right)^2 \right]. 
\end{equation*}
The Glivenko-Cantelli result  in Hallin~(2017, Proposition 5.1) entails 
\begin{align}
{\max}_{1 \leq t \leq n} \Big\vert R^{(n)}_{{\pms}, t}/(n_R + 1) - \Vert  \F_{{\pms}, t} \Vert  \Big\vert \rightarrow 0 \quad a.s. \quad \text{as} \quad n \rightarrow \infty.
\end{align}
In view of the assumptions made  on ${J}_1$ and ${J}_2$,   Lemma 6.1.6.1 of Hájek et al.~(1999) yields  
\begin{equation}\label{T11}\Vert \text{vec} ({\boldsymbol\delta}_1^{(n)})\Vert _{L^2}^2 \rightarrow 0\quad\text{ as $n\rightarrow \infty$.}
\end{equation}

For ${\boldsymbol\delta}_2^{(n)}$, we have
$${\boldsymbol\delta}_2^{(n)} = (n-i)^{-1/2} \sum_{t=i+1}^n {J}_1(\Vert  \F_{{\pms}, t} \Vert ) {J}_2(\Vert  \F_{{\pms}, t-i} \Vert ) \left[
\left(\S^{(n)}_{{\pms}, t} -   \S_{{\pms}, t}\right) \S^{(n)\prime}_{{\pms}, t-i} + \S_{{\pms}, t}  \left(\S^{(n)\prime}_{{\pms}, t-i} - \S^\prime_{{\pms}, t-i}\right)\right].  $$
Similar to the arguments used for ${\boldsymbol\delta}_1^{(n)}$, Lemma \ref{fourStats} and $(\text{vec} \A)^\prime (\text{vec} \B) = \text{tr} (\A^\prime \B)$ imply
\begin{align}
\Vert \text{vec} ({\boldsymbol\delta}_2^{(n)})\Vert _{L^2}^2 &\leq 2 (n-i)^{-1} \sum_{t=i+1}^n {\rm E}\left[ \left( {J}_1(\Vert  \F_{{\pms}, t} \Vert ) {J}_2(\Vert  \F_{{\pms}, t-i} \Vert ) \right)^2 \Vert  \S^{(n)}_{{\pms}, t} -   \S_{{\pms}, t} \Vert ^2 \right] \label{T21}\\
&\quad + 2 (n-i)^{-1} \sum_{t=i+1}^n {\rm E}\left[ \left( {J}_1(\Vert  \F_{{\pms}, t} \Vert ) {J}_2(\Vert  \F_{{\pms}, t-i} \Vert ) \right)^2 \Vert  \S^{(n)}_{{\pms}, t-i} -   \S_{{\pms}, t-i} \Vert ^2 \right]. \label{T22}
\end{align}
Still in view of  Proposition 5.1 in Hallin~(2017),
 $
{\max}_{1 \leq t \leq n}  \Vert  \S^{(n)}_{{\pms}, t} -   \S_{{\pms}, t} \Vert  \rightarrow 0$ a.s.~as~$n \rightarrow~\!\infty$. 
Since ${J}_1$ and $ {J}_2$ are square-integrable and $\ZZ_1^{(n)}, \ldots , \ZZ_n^{(n)}$ are independent, both (\ref{T21}) and~(\ref{T22}) converge to $0$. The result follows.
\end{proof}

We now can extend the above asymptotic representation and asymptotic normality results  from the rank-based cross-covariance matrices $ \tenq{\bGamma}_{i, {J}_1, {J}_2}^{(n)}(\bth)$ to  the rank-based central sequence~$\tenq{\bDelta}^{(n)}_{{J}_1, {J}_2}(\bth)$. 

\noindent  \begin{lemma}\label{lem.bar.til.Delta.n}
Let Assumptions (A1), (A2), and (A3) hold. Then, 
\begin{equation}\label{bar.til.Delta.n}
\tenq{\bDelta}^{(n)}_{{J}_1, {J}_2}(\bth) - \bar{\bDelta}^{(n)}_{{J}_1, {J}_2}(\bth) = o_{\rm P} (1)\quad\text{as $n \rightarrow \infty$}
\end{equation}
both under  ${\rm P}^{(n)}_{\bth ;f}$ and~${\rm P}^{(n)}_{\bth + n^{-1/2} \btau ;f}$. Moreover,  $\tenq{\bDelta}^{(n)}_{{J}_1, {J}_2}(\bth)$  is asymptotically normal,   with mean~$\0$ under  ${\rm P}^{(n)}_{\bth ;f}$, mean
\begin{equation}\label{bar.tilde.Delta.diff}
\underset{n\rightarrow \infty}{\lim}  \left\lbrace \T^{(n)}_{\bth} (\I_{n-1} \otimes \K_{{J}_1, {J}_2, f}) \T^{(n)^\prime}_{\bth} \right\rbrace \btau 
\end{equation}
under ${\rm P}^{(n)}_{\bth + n^{-1/2}\btau ;f}$, and   covariance   
$d^{-2}\sigma^2_{{J}_1}\sigma^2_{{J}_2}  \, \underset{n\rightarrow \infty}{\lim}  \left\lbrace \T^{(n)}_{\bth} \T^{(n)^\prime}_{\bth} \right\rbrace$
 under both.
\end{lemma}
Note that the limits appearing in the above asymptotic means and covariances exist due to Assumption (A1) on the characteristic roots  of the VARMA operators involved. 

\begin{proof}
For (\ref{bar.til.Delta.n}), due to Lemma \ref{bar.til.Gam} and contiguity, it is sufficient  to prove that,  under~${\rm P}^{(n)}_{\bth ;f}$, for~$m=m(n) \leq n-1$ and provided that $m(n)\to\infty$ as $n\to\infty$,
\begin{equation}\label{bar.Delta.mn}
\underset{n \rightarrow \infty}{\lim \,\sup} \Vert \bar{\bDelta}^{(n)}_{{J}_1, {J}_2}(\bth) - \bar{\bDelta}^{(n)}_{m(n), {J}_1, {J}_2}(\bth)\Vert  = o_{\rm P}(1) 
\end{equation}
and 
\begin{equation}\label{til.Delta.mn}
\underset{n \rightarrow \infty}{\lim \,\sup}  \Vert \tenq{\bDelta}^{(n)}_{{J}_1, {J}_2}(\bth) - \tenq{\bDelta}^{(n)}_{m(n), {J}_1, {J}_2}(\bth)\Vert  = o_{\rm P}(1) .
\end{equation}
For  $m = n-1$, the left-hand sides in~\eqref{bar.Delta.mn} and \eqref{til.Delta.mn} are exactly zero. Therefore, we only need to consider $m \leq n-2$. 
 It follows from Proposition 3.1 (LAN2) in 
 Garel and Hallin (1995) that
\begin{align*}
&\bar{\bDelta}^{(n)}_{{J}_1, {J}_2}(\bth) - \bar{\bDelta}^{(n)}_{m(n), {J}_1, {J}_2}(\bth) \\
&\quad = 
\begin{bmatrix}
\sum_{i=m+1}^{n-1} \sum_{j=0}^{i-1} \sum_{k=0}^{\min (q, i-j-1)} [ (\G_{i-j-k-1} \B_k) \otimes \H_j^\prime ] (n-i)^{1/2} (\text{vec} (\bar{\bGamma}_{i, {J}_1, {J}_2}^{(n)}(\bth))) \\
\vdots \\
\sum_{i=m+1}^{n-1} \sum_{j=0}^{i-p} \sum_{k=0}^{\min (q, i-j-p)} [ (\G_{i-j-k-p} \B_k) \otimes \H_j^\prime ] (n-i)^{1/2} (\text{vec} (\bar{\bGamma}_{i, {J}_1, {J}_2}^{(n)}(\bth))) \\
\sum_{i=m+1}^{n-1} (\I_d \otimes \H_{i-1}^\prime (n-i)^{1/2} (\text{vec} (\bar{\bGamma}_{i, {J}_1, {J}_2}^{(n)}(\bth)) \\
\vdots \\
\sum_{i=m+1}^{n-1} (\I_d \otimes \H_{i-q}^\prime (n-i)^{1/2} (\text{vec} (\bar{\bGamma}_{i, {J}_1, {J}_2}^{(n)}(\bth))
\end{bmatrix} 
\end{align*}
for any~$p \leq m \leq n-2$,. Due to the square-integrability of ${J}_1, {J}_2$ and the fact that~$\ZZ_1^{(n)}, \ldots , \ZZ_n^{(n)}$ are i.i.d., it follows from  $(\text{vec} \A)^\prime (\text{vec} \B) = \text{tr} (\A^\prime \B)$ that 
$$
\Vert  (n-i)^{1/2} (\text{vec} (\bar{\bGamma}_{i, {J}_1, {J}_2}^{(n)}(\bth)) )\Vert _{L^2}^2 
=  (n-i)^{-1} \sum_{t = i+1}^n  {\mathrm E}\left[  {J}_1^2(\Vert  \F_{{\pms}, t} \Vert ) \right] 
{\mathrm E} \left[ {J}_2^2(\Vert  \F_{{\pms}, t-i} \Vert ) \right]  
= \sigma^2_{{J}_1}  \sigma^2_{{J}_2}  < \infty.
$$
Recall that, under Assumption (A1),  the Green matrices $\G_u$ and $\H_u$ decrease exponentially fast (see Appendix~A). Using the fact that   $\Vert  \mathbf{A} \mathbf{x} \Vert _{L^2} \leq \Vert  \mathbf{A} \Vert  \, \Vert \mathbf{x} \Vert _{L^2}$  (where $\Vert  \mathbf{A}\Vert $ denotes the operator norm of $\mathbf{A}$) and the triangular inequality, we thus obtain 
$$\underset{n \rightarrow \infty}{\lim \,\sup} \Vert \bar{\bDelta}^{(n)}_{{J}_1, {J}_2}(\bth) - \bar{\bDelta}^{(n)}_{m(n), {J}_1, {J}_2}(\bth)\Vert _{L^2} = 0.$$
The result \eqref{bar.Delta.mn} follows. 
 Turning to  (\ref{til.Delta.mn}),  we have, in view of \eqref{T11}, \eqref{T21} and \eqref{T22}, 
$$\underset{1 \leq i \leq  n-1}{\max} \Vert  (n-i)^{1/2} [\text{vec} (\bar{\bGamma}_{i, {J}_1, {J}_2}^{(n)}(\bth)) - \text{vec} (\tenq{\bGamma}_{i, {J}_1, {J}_2}^{(n)}(\bth)) ] \Vert _{L^2}^2 = o(1)$$
as $n \rightarrow \infty$. Hence, (\ref{til.Delta.mn}) follows along the same lines as (\ref{bar.Delta.mn}). The asymptotic normality of~$\tenq{\bDelta}^{(n)}_{{J}_1, {J}_2}(\bth)$ then follows from~(\ref{bar.til.Delta.n}) and the asymptotic normality of $\bar{\bDelta}^{(n)}_{{J}_1, {J}_2}(\bth)$, itself implied by (\ref{bar.Delta.mn}) and Lemma \ref{asy.S.bar}. The asymptotic mean and variance are the limits as $m=m(n)$ and~$n$ tend to infinity,  of the asymptotic mean and variance of $\bar{\bDelta}^{(n)}_{m(n), {J}_1, {J}_2}(\bth)$ and do not depend on the way $m$ grows with $n$.
\end{proof}

\vspace{0.2cm}

\noindent \textbf{Proof of Proposition \ref{asy0}.}

Proposition \ref{asy0} readily follows from  (\ref{til.Delta.mn}) and the asymptotic  linearity of the truncated~$\tenq{\bDelta}^{(n)}_{m, {J}_1, {J}_2}(\bth)$  implied by Assumption (A4).\cqfd 

\vspace{0.4cm}

\noindent \textbf{Proof of Proposition \ref{asy}.}

From the definition of $\tenq{\hat{\bth}}\n$ in (\ref{onestep.def}), the asymptotic linearity in Proposition \ref{asy0}, the consistency of  $\hat{\bUpsilon}_{{J}_1, {J}_2}^{(n)}$, the convergence of $\bUpsilon_{{J}_1, {J}_2, f}^{(n)}$ to $\bUpsilon_{{J}_1, {J}_2, f}$, and the asymptotic discreteness of $ \hat{\bth}^{(n)}$ (which allows us to treat $n^{1/2}( \hat{\bth}^{(n)} -  \bth )$ as if it were a bounded constant: see   Lemma~4.4  in Kreiss~(1987)), we have
\begin{align*}
  n^{1/2}  (\tenq{\hat{\bth}}\n - \bth)
  & 
 = n^{1/2} \left\lbrace \hat{\bth}^{(n)} + n^{-1/2} \left[ \left( \hat{\bUpsilon}_{{J}_1, {J}_2}^{(n)} \right)^{-1} \tenq{\bDelta}^{(n)}_{{J}_1, {J}_2} ( \hat{\bth}^{(n)}) \right] - \bth  \right\rbrace \\
 &= n^{1/2} \left\lbrace \hat{\bth}^{(n)} + n^{-1/2} \left[  \bUpsilon_{{J}_1, {J}_2, f}^{-1}  \left( \tenq{\bDelta}^{(n)}_{{J}_1, {J}_2} (\bth) -  \bUpsilon_{{J}_1, {J}_2, f}^{(n)}  n^{1/2} (\hat{\bth}^{(n)} - \bth)  \right) \right] - \bth  \right\rbrace  + o_{\rm P}(1) \\
 &=   \bUpsilon_{{J}_1, {J}_2, f}^{-1}  \tenq{\bDelta}^{(n)}_{{J}_1, {J}_2} (\bth) + o_{\rm P}(1).
\end{align*}
This, in view of the asymptotic normality of $\tenq{\bDelta}^{(n)}_{{J}_1, {J}_2} (\bth)$ in Lemma \ref{lem.bar.til.Delta.n}, completes the proof of Proposition \ref{asy}.\cqfd 


\section{Computational issues}\label{secalg}

\subsection{Implementation details} \label{Sec:ImpDet}

In this section, we briefly discuss some computational aspects related to the implementation of our methodology.

(i) Consistency requires that both $n_R $ and $n_S$  tend to infinity. In practice, we  factorize $n$\linebreak  into~$n_Rn_S + n_0$ in such a way that both $n_R$ and $n_S$ are large. Typically, $n_R$  is of order~$n^{1/d}$ and~$n_S$ is of order~$n^{(d-1)/d}$, whilst $0 \leq n_0 < \min(n_S, n_R)$ has to be small as possible---its value, however, is entirely determined by  the values of $n_R$ and $n_S$.
Generating  ``regular grids" of $ n_S$ points over the unit sphere $\mathcal{S}_{d-1}$ as described in Section \ref{secranks} is easy  for $d=2$, where perfect regularity can be achieved by dividing the unit circle into $n_S$ arcs of equal length~$2\pi/n_S$. {For $d\geq 3$, 
``perfect regularity"  is no longer possible. 
A random array of $n_S$ independent and uniformly distributed unit vectors does satisfy (almost surely) the requirement for weak convergence (to ${\mathrm U}_d$). More regular deterministic arrays (with faster convergence) can be constructed, though, such as the {\it low-discrepancy sequences}   (see, e.g., Niederreiter (1992), Judd~(1998), Dick and Pillichshammer (2014), or Santner et al. (2003)) considered  in numerical integration and the design of computer experiments; we suggest the use of the function {\tt UnitSphere} in R package {\tt mvmesh}. 

(ii) The empirical center-outward distribution function $\F_{{\pms} }^{(n)} $ is obtained as the solution of  an optimal coupling problem. Many efficient algorithms have been proposed in the measure transportation literature (see, e.g., Peyr\' e and Cuturi~(2019)). We followed   Hallin et al.~(2020a),  using a Hungarian algorithm (see the \texttt{clue} R package).

(iii) The computation of the one-step R-estimator  in (\ref{onestep.def}) involves  two basic ingredients: a preliminary root $n$-consistent estimator $\hat{\bth}^{(n)}$ and an estimator of the cross-information matrix~$\bUpsilon_{{J}_1, {J}_2, f}$. For the  preliminary  $\hat{\bth}^{(n)}$,  robust M-estimators such as the  reweighted multivariate least trimmed squares estimator  (RMLTSE) proposed by Croux and Joossens (2008) for VAR models are obvious candidates; provided that  fourth-order moments   finite, the QMLE still constitutes   a reasonable  choice, though. Different preliminary estimators may lead to different one-step R-estimators. Differences, however, gradually wane on iterating (for fixed $n$) the one-step procedure and the asymptotic impact (as~$n\to\infty$) of the choice of~$ \hat{\bth}^{(n)}$ is nil. 
%
 Turning to the estimation of $\bUpsilon_{{J}_1, {J}_2, f}$, the issue is that this matrix depends on the unknown actual density~$f$. 
A simple consistent estimator is obtained by letting $\btau= {\bf e}_i$, $i=1,\ldots,(p+q)d^2$ in~\eqref{asy.linear2} where ${\bf e}_i$  denotes the $i$th vector of the canonical basis in the parameter space~${\mathbb{R}}^{(p+q)d^2}$: the difference~$
\tenq{\bDelta}^{(n)}_{{J}_1, {J}_2}(\hat{\bth}^{(n)} + n^{-1/2}{\bf e}_i) - \tenq{\bDelta}^{(n)}_{{J}_1, {J}_2}(\hat{\bth}^{(n)})
$ 
then provides a consistent estimator of the $i$-th column of $-\bUpsilon_{{J}_1, {J}_2, f}(\bth)$. 
See Hallin et al.~(2006) or Cassart et al.~(2010) for more sophisticated estimation methods. 

\subsection{Algorithm}\label{algsec}

We give here a detailed description of the estimation algorithm. Due to the exponential decay, under Assumption~(A1), of the coefficients of the MA(${\infty}$) representation of VARMA($p,q$) models, there is no need to bother about the truncation   of the central sequence, which safely can be  set to  $m=n-1$ or $m=(1-p)n$ with $p<1$. Then, the implementation of our R-estimation method then goes along the lines of the following algorithm. \vspace{3mm}

\begin{algorithm}[H]
\SetAlgoLined\vspace{3mm}

\KwIn{a $d$-dimensional sample $\{\X_t; 1 \leq t \leq n\}$, orders $p$ and $q$ of the VARMA process, number $k$ of iterations in the one-step procedure; truncation lag $m$.}
\KwOut{R-estimator $\tenq{\hat{\bth}}\n$}
\begin{enumerate}
\item Factorize $n$ into $n_Rn_S + n_0$ and  generate (see {\it (i)} of Appendix C.1),  a ``regular grid" of $n_R n_S$ points over the unit ball $\mathbb{S}_d$.  

\item Compute a preliminary root-$n$ consistent estimator $\hat{\bth}^{(n)}$.

 \item Set the initial values $\bepsilon_{-q+1}, \ldots , \bepsilon_0$ and~$\X_{-p + 1}, \ldots , \X_{0}$ all equal to zero, and compute residuals  $\ZZ_1^{(n)}(\hat{\bth}^{(n)}),\ldots,\ZZ_n^{(n)}(\hat{\bth}^{(n)})$ recursively or from~\eqref{Zt.recursive}.
 
 \item Create a $n\times n$ matrix $\mathbf{D}$ with $(i,j)$ entry 
 the squared Euclidean distance between~$\mathbf{Z}_i^{(n)}$ and the $j$-th gridpoint. Based on that  matrix, compute $\{\F_{\pms}^{(n)} (\ZZ_t^{(n)}); t=1,\ldots, n \}$ solving the optimal pairing problem in (\ref{Fpm0}),  using e.g.  the Hungarian algorithm. 
 
 \item From $\F^{(n)}_{\pms}$, compute the center-outward ranks~(\ref{Ranks}) and signs~(\ref{Signs}).
 
 \item Specify the scores ${J}_1$ and ${J}_2$ (e.g., the standard scores proposed in Section~\ref{Sec: examples})  and compute $\M_{{\hat{\bth}^{(n)}}}$, $\P_{{\hat{\bth}^{(n)}}}$, and $\Q^{(n)}_{{\hat{\bth}^{(n)}}}$, 
 as defined in Appendix A, then   $\tenq{\bGamma}_{i, {J}_1, {J}_2}^{(n)}({\hat{\bth}^{(n)}})$ (use e.g. one of the expressions available in Section~\ref{Sec: examples}). Finally, combine these expressions into $\tenq{\bDelta}^{(n)}_{{J}_1, {J}_2}(\hat{\bth}^{(n)})$. 
 
 \item For some chosen $\btau_1,\ldots,\btau_{(p+q)d^2}$, 
 compute $\tenq{\bDelta}^{(n)}_{{J}_1, {J}_2}(\hat{\bth}^{(n)} + n^{-1/2}\btau)$, then,  via~(\ref{asy.linear2}),~$\hat{\bUpsilon}_{{J}_1, {J}_2}^{(n)}$.
 
\item Set  $\tenq{\hat{\bth}}\n = \hat{\bth}^{(n)}$.
 
\item \For{$i \gets 1$ \textbf{to} $k$}{

$$\tenq{\hat{\bth}}\n \gets \tenq{\hat{\bth}}\n + n^{-1/2} \left( \hat{\bUpsilon}_{{J}_1, {J}_2}^{(n)} \right)^{-1} \tenq{\bDelta}^{(n)}_{{J}_1, {J}_2} ( \tenq{\hat{\bth}}\n).$$
}

\end{enumerate}
\caption{Center-outward R-estimation for semiparametric VARMA models}
\end{algorithm}

\newpage

\section{Supplementary material for Section~\ref{Sec: MC}}\label{Supp.sim}

\subsection{Center-outward quantile contours, with a graphical illustration}

We provide here some additional concepts from Hallin (2017) and Hallin et al.~(2020). Recall that an {\it order statistic} $\ZZ_{(\cdot)}^{(n)}$ of the un-ordered $n$-tuple $\ZZ^{(n)}$ is an arbitrarily ordered version of the same---for instance,  $\ZZ_{(\cdot)}^{(n)} = \left( \ZZ_{(1)}^{(n)}, ..., \ZZ_{(n)}^{(n)}\right)$, where $\ZZ_{(i)}^{(n)}$ is such that its first component is the $i$th order
statistic of the $n$-tuple of first components.

The \textit{center-outward quantile contours} are defined as
\begin{equation}
\mathcal{C}^{(n)}_{{\pms};\ZZ^{(n)}_{(.)}}\left(\frac{j}{n_R +1} \right):= \big\{  \ZZ^{(n)}_t \vert R^{(n)}_{{\pms}, t} = j  \big\},
\label{Cont}
\end{equation}
where ${j}/{(n_R +1)}$, $j = 0, 1, . . . , n_R$ is an empirical probability content, to be interpreted as a quantile order.   Figure \ref{Fig: Contours} provides a graphical illustration of this concept: 
 $n=1000$ (with~$n_R=25$ and~$n_S=40$) bivariate observations were drawn from the Gaussian  mixture~(\ref{Eq. Mixture}), the skew-normal and skew-$t_3$ described in Section~\ref{skewsec}, and, for a comparison, from a spherical multivariate normal. The plots show that the center-outward quantile contours nicely conform to the shape of the  underlying distribution in both symmetric and asymmetric cases.

\subsection{Skew-normal, skew-$t$, and Gaussian mixture innovation densities}\label{skewsec}
The   skew-normal  distribution considered in Section~\ref{Sec: MC} has density (with $\phi({\bf \cdot}; \bSigma)$ standing for the~${\cal N}({\bf 0}, \bSigma)$ density, $\Phi$  for   the univariate  standard normal  distribution function)  
\begin{equation}\label{SN.density}
f_{\bepsilon}(\z; \bxi, \bSigma, \al) := 2 \phi(\z - \bxi; \bSigma) \Phi(\al^\prime {\mbf w}^{-1} (\z - \bxi)),\quad \z \in \R^d,
\end{equation}
where~$\bxi \in \R^d$, $\al \in \R^d$, and  ${\mbf w} = {\rm diag} (w_1, \ldots, w_d) >0$ are location, shape, and scale parameters, respectively. The  skew-$t_\nu$  distribution has density\vspace{-8mm} 
\begin{align}\label{St.density}
&  \\
f_{\bepsilon}(\z; \bxi, \bSigma, \al, \nu) :=2 {\rm det}({\mbf w)}^{-1} 
 t_d(\x; \bSigma, \nu) T\left(\al^\prime \x  \big({{(\nu + d)}/{(\nu + \x^\prime \bSigma^{-1} \x)}}\big)^{1/2}; \nu + d \right),\quad\!\!\! \z \in \R^d,  \nonumber
\end{align}
where $\x = {\mbf w}^{-1}(\z - \bxi)$,   $T(y; \nu)$ denotes the univariate $t_\nu$ distribution function and 
$$t_d(\x; \bSigma, \nu) := \frac{\Gamma((\nu+d)/2)}{(\nu \pi)^{d/2} \Gamma(\nu/2) {\rm det}(\bSigma)^{1/2}} \left( 1+ \frac{\x^\prime \bSigma^{-1} \x}{\nu} \right)^{-(\nu+d)/2},\quad \x \in \R^d.$$ 
We refer to Azzalini and Dalla Valle (1996), Azzalini and Capitanio (2003) for details.  

\begin{figure}[!htbp]
\caption{Empirical center-outward quantile contours (probability contents 26.9\%, 50 \%, and~80\%, respectively) computed from~$n=1000$ points drawn from 
the  Gaussian  mixture~(\ref{Eq. Mixture}) (top left), the skew-normal and skew-$t_3$ described in Section~\ref{skewsec} (top right and bottom left) and, for a comparison, from a standard multivariate normal (bottom right).\vspace{-3mm}}
\begin{center}
\begin{tabular}{cc}
\includegraphics[width=0.5\textwidth, height=0.45\textwidth]{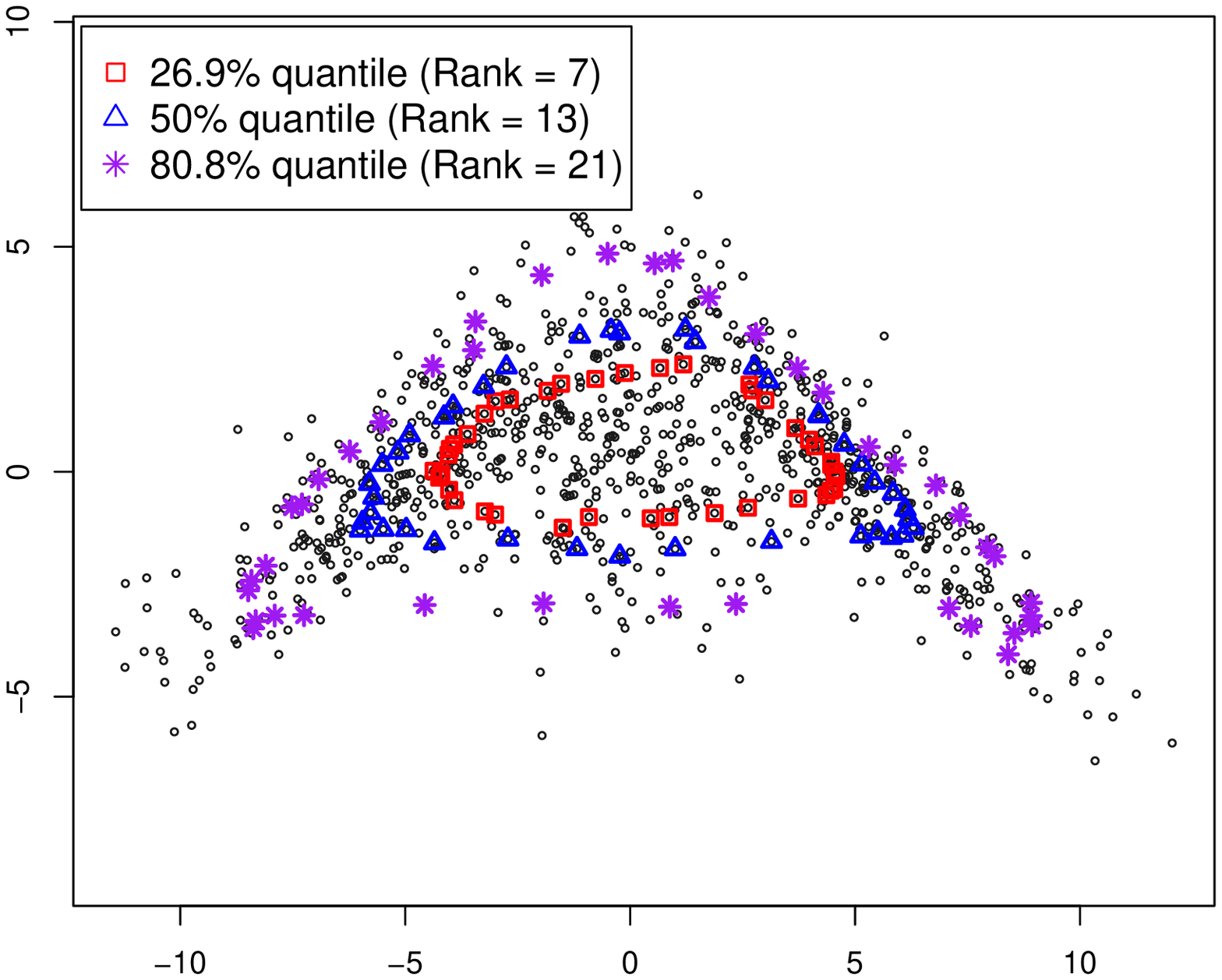} \vspace{-1mm}\hspace{-5mm}
&\includegraphics[width=0.5\textwidth, height=0.45\textwidth]{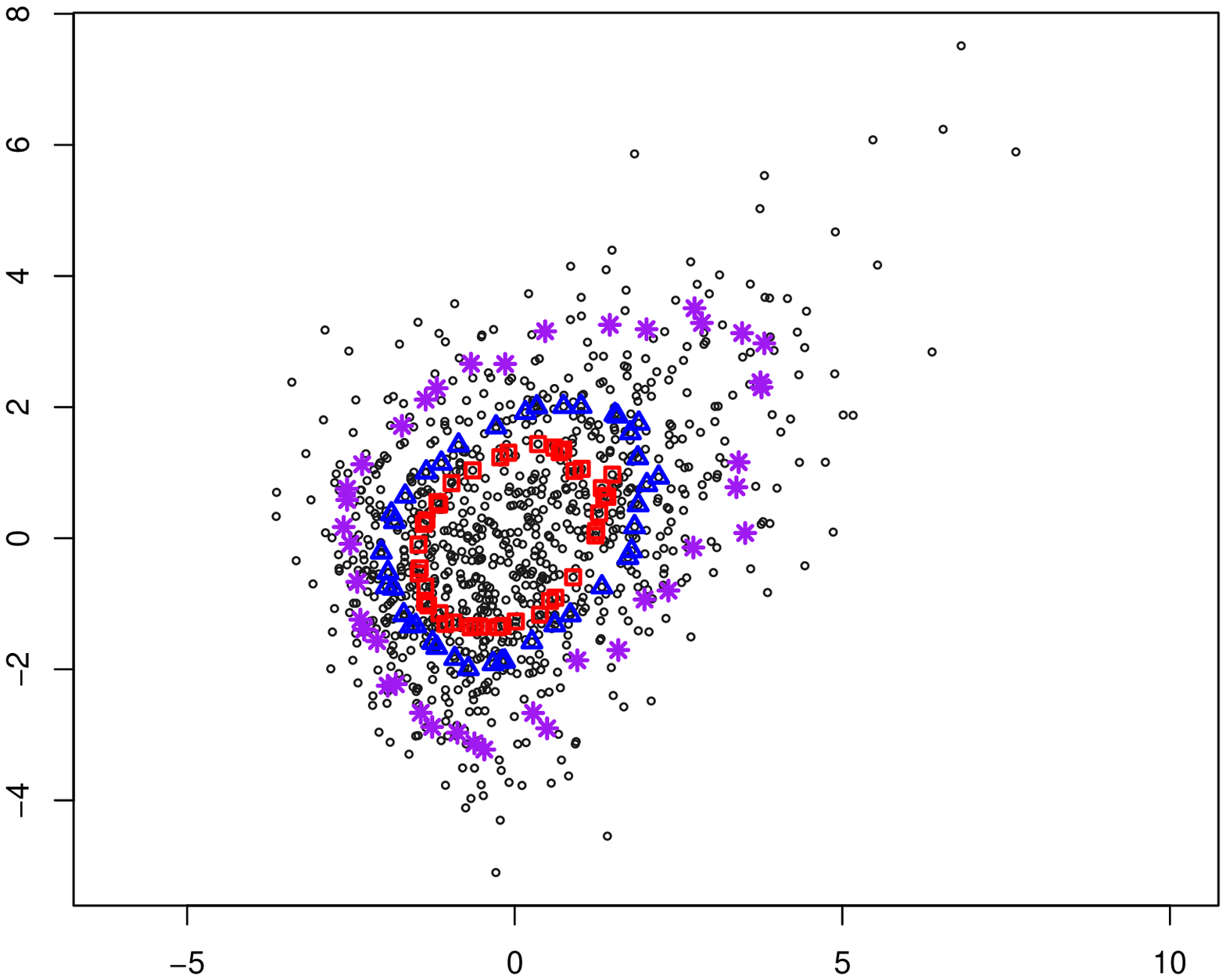}\vspace{-1mm}\hspace{-5mm}
\\
\includegraphics[width=0.5\textwidth, height=0.45\textwidth]{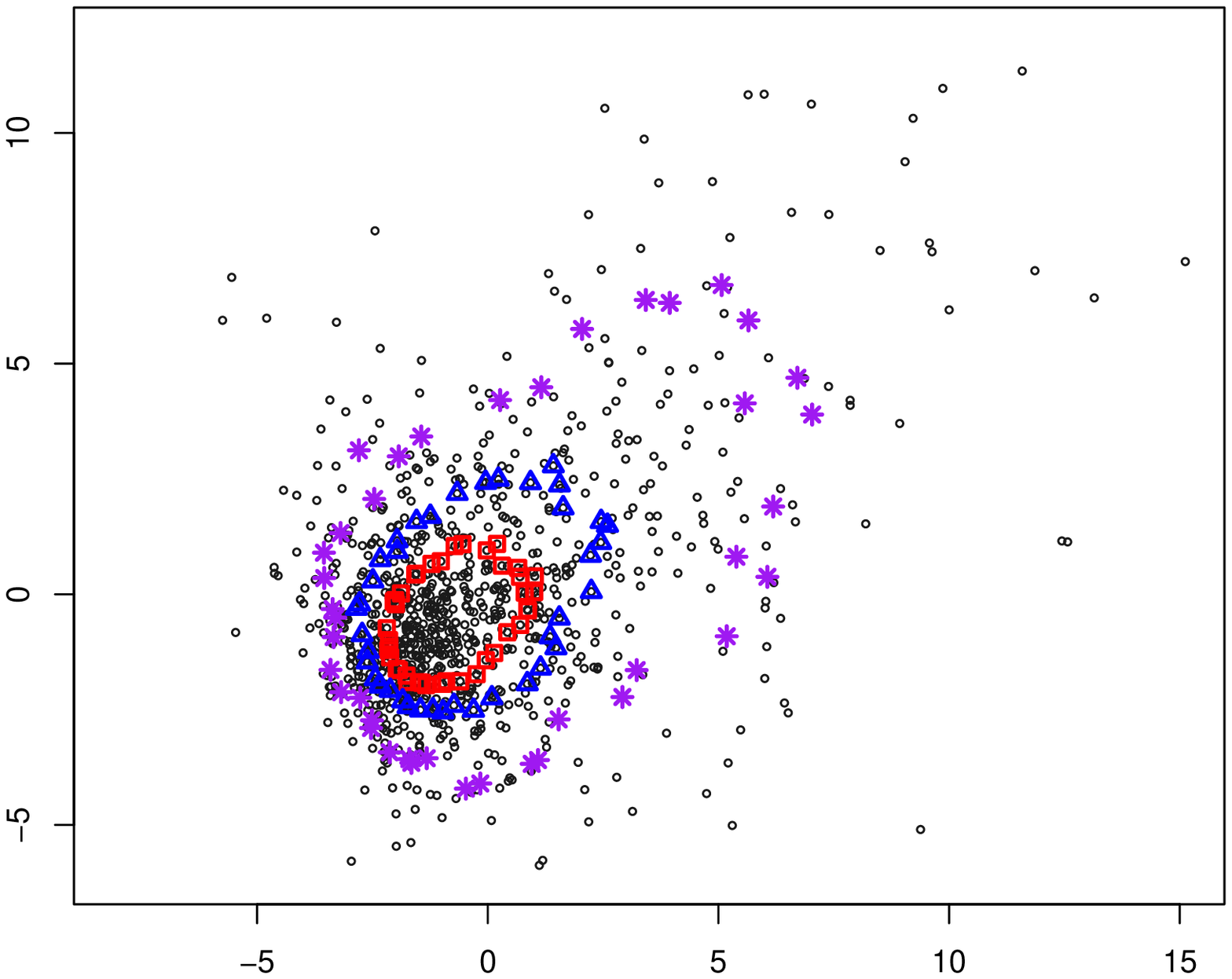}\vspace{-1mm}\hspace{-5mm}
&\includegraphics[width=0.5\textwidth, height=0.45\textwidth]{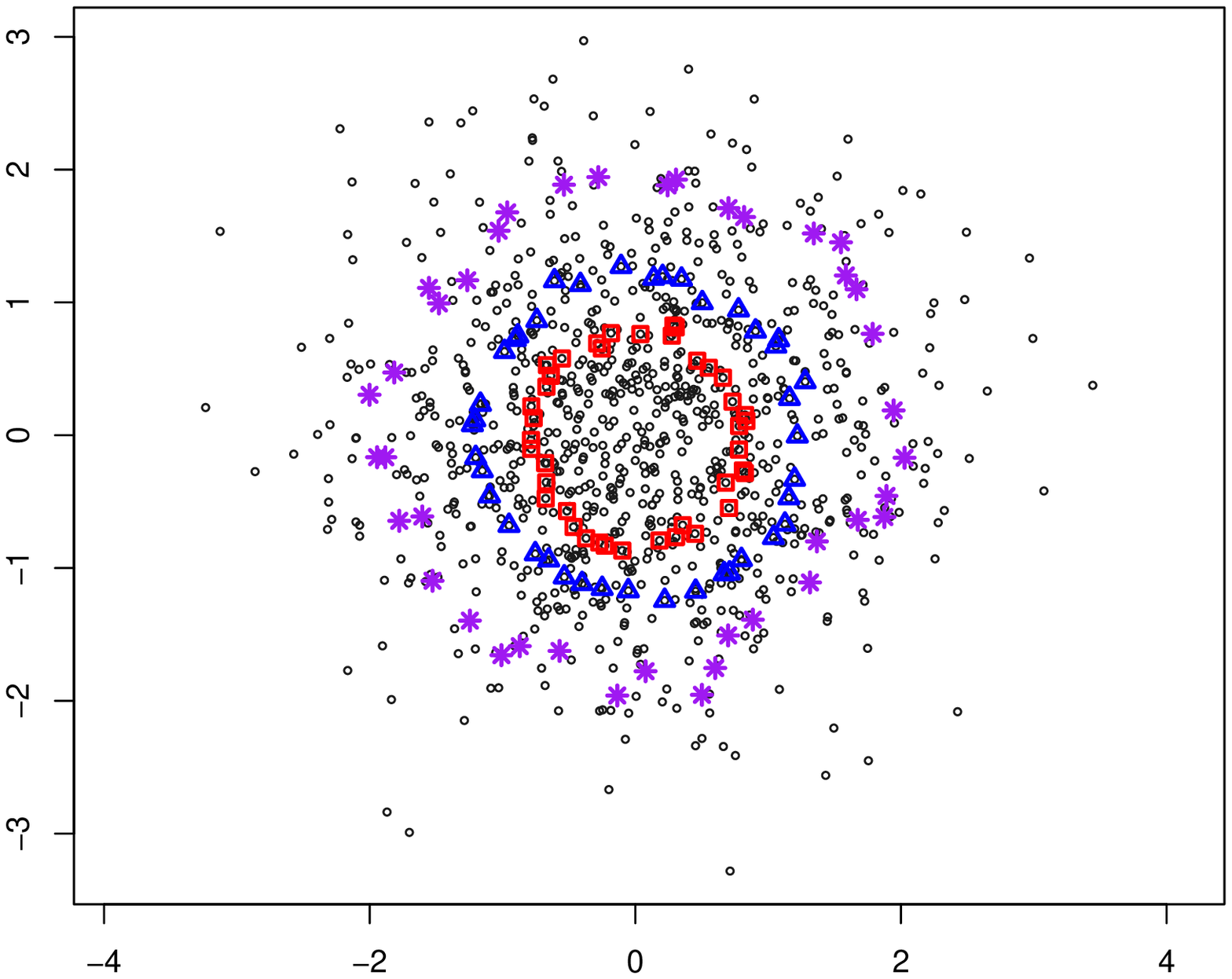}\vspace{-1mm}\hspace{-5mm}
\end{tabular}
\end{center}
\vspace{-3mm}\label{Fig: Contours}
\end{figure}

 Our samples for skew-normal and skew-$t_3$ were simulated from the function {\tt rmst} in the R Package {\tt sn} by setting $\bxi = \0, \al = (5, 2)^\prime, \bSigma = \left( \begin{array}{cc}
7 & 4 \\
4 & 5
\end{array}\right)$. In order to satisfy the classical conditions for M-estimation,  we centered  the simulated innovations about their mean, a centering which does not affect our R-estimators. 

Figure~\ref{Fig: Contours} provides scatterplots of  samples of size  $n=1000$ from the spherical normal, the skew-normal, the skew-$t_3$, and the Gaussian mixture described in Section~\ref{Sec: MC}. 

\subsection{Additional numerical results}\label{NumresApp}

\subsubsection{Large sample }\label{AppD31}

As a complement to Section~\ref{Sect: MClarge}, we provide here, for sample size $n = 1000$, boxplots of the QMLE, $t_5$-QMLE, RMLTSE, and R-estimators (sign test, Spearman, and van der Waerden scores) under skew-normal, skew-$t_3$,  $t_3$ and non-spherical Gaussian innovations with covariance~
$$
\bSigma_4 =\left( \begin{array}{cc}
5 & 4 \\
4 & 4.5
\end{array}\right);$$
See Figure~\ref{boxSkewNorm}, \ref{boxSkewt3}, \ref{boxt3} and \ref{boxell}, respectively. 

Under skew-normal (Figure~\ref{boxSkewNorm}) and skew-$t_3$ (Figure~\ref{boxSkewt3}) innovations, the vdW and Spearman R-estimators are less dispersed than other M-estimators, showing that they are more resistant to skewness. Under spherical $t_3$ innovations (Figure~\ref{boxt3}), outlying observations are relatively frequent and the QMLE is no longer root-$n$ consistent. The RMLTSE does its job as a robustified estimator and slightly outperforms the R-estimators (the weakest of which is the sign-test score one).  The non-spherical Gaussian boxplots (Figure~\ref{boxell}) show that the vdW and Spearman R-estimators are quite similar  to the QMLE.


\begin{center}
\begin{figure}[!htbp]
\caption{Boxplots of the QMLE, $t_5$-QMLE,  RMLTSE, and R-estimators (sign test, Spearman,  and van der Waerden) under   skew-normal innovations     (\ref{SN.density}); sample size $n = 1000$;  $N = 300$ replications.  The horizontal red line represents the actual parameter value.}
\begin{center}
\includegraphics[width=1\textwidth, height=0.5\textwidth]{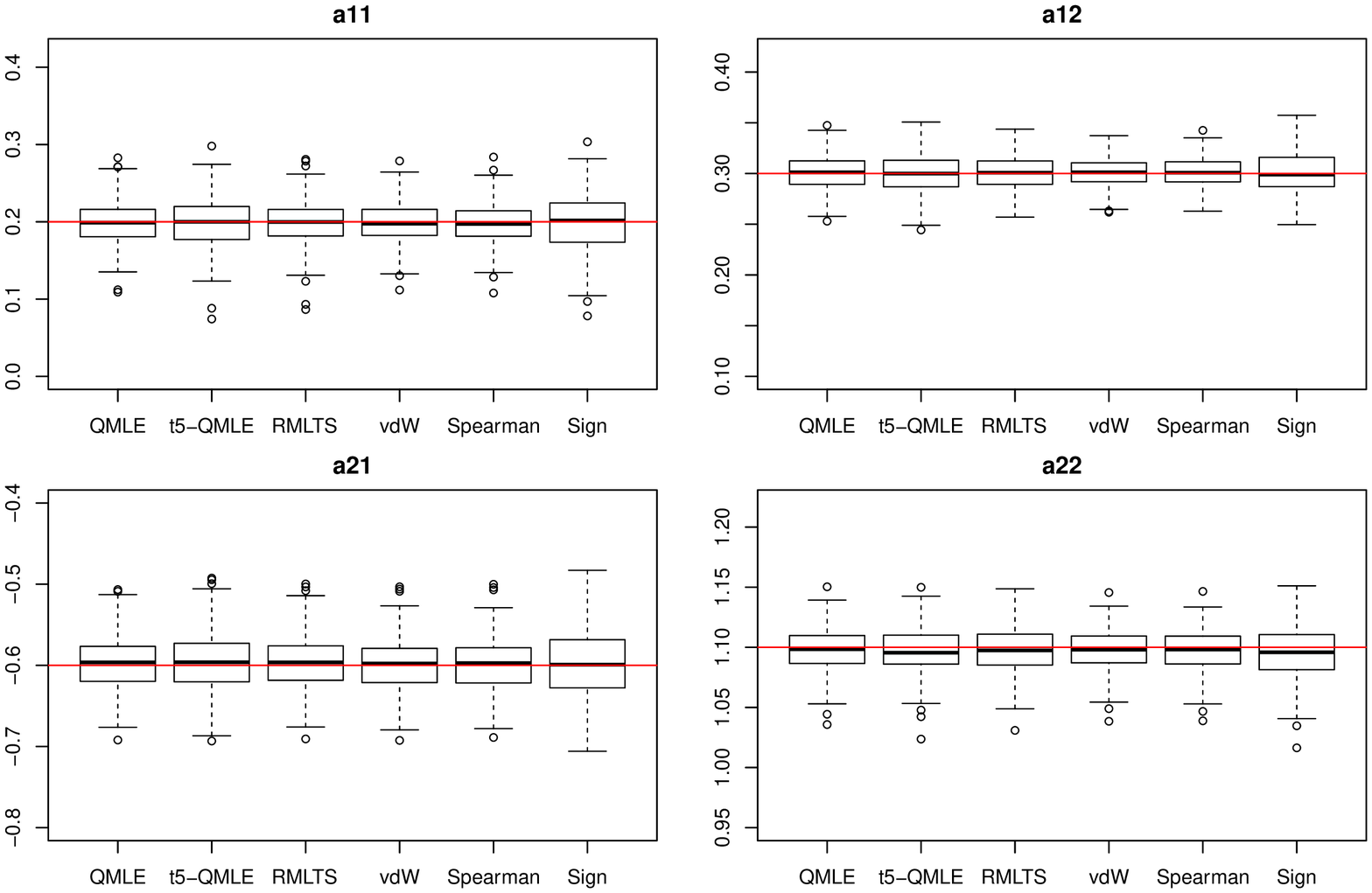} \vspace{-8mm}
\end{center}
\label{boxSkewNorm}
\end{figure}
\end{center}

\begin{center}
\begin{figure}[!htbp]
\caption{Boxplots of the QMLE, $t_5$-QMLE, RMLTSE, and R-estimators (sign test, Spearman,  and van der Waerden) under   skew-$t_3$ innovations     (\ref{St.density}); sample size $n = 1000$;  $N = 300$ replications.  The horizontal red line represents the actual parameter value.}
\begin{center}
\includegraphics[width=1\textwidth, height=0.5\textwidth]{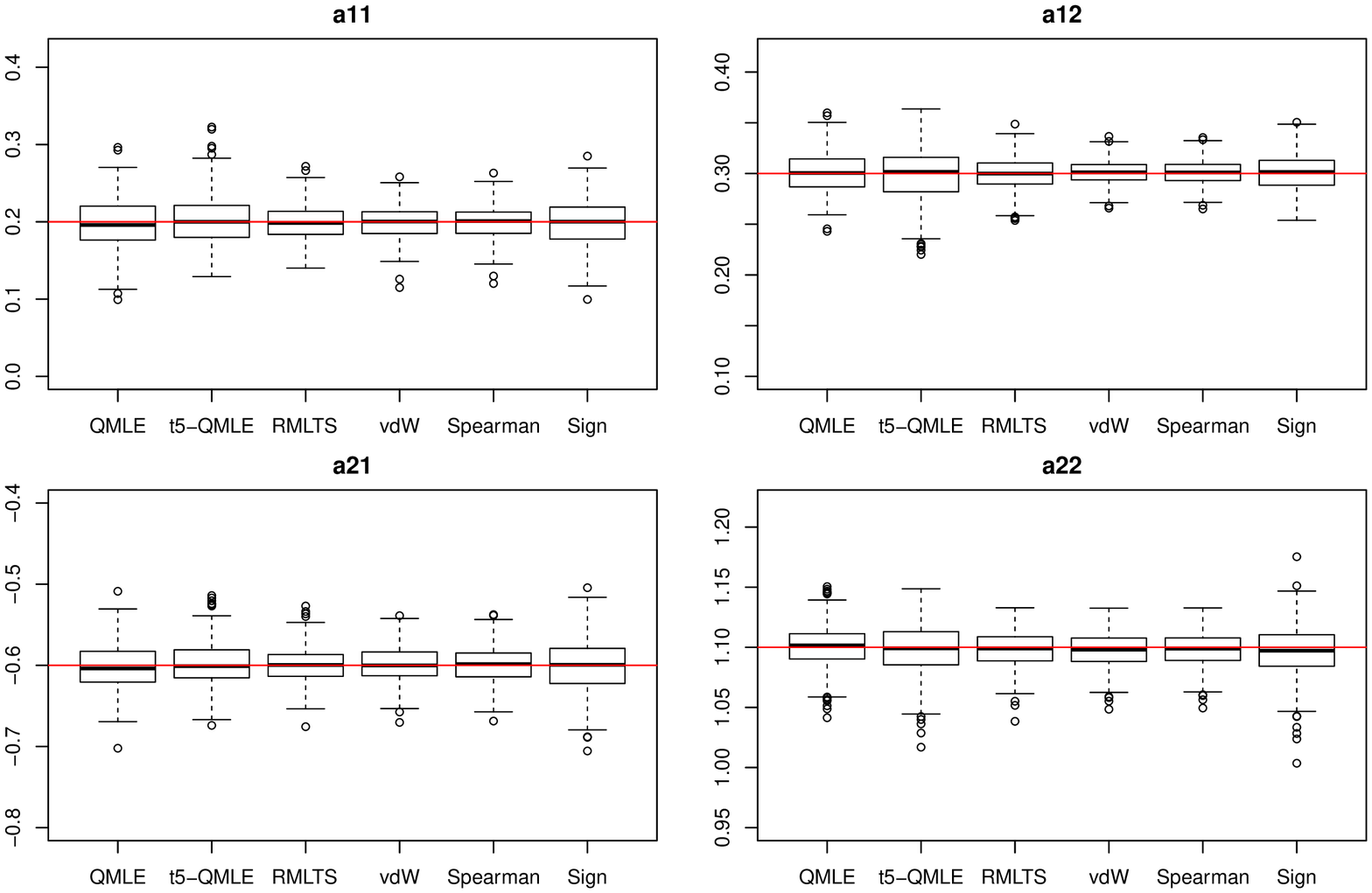}
\end{center}
\label{boxSkewt3}
\end{figure}
\end{center}

\begin{center}
\begin{figure}[!htbp]
\caption{Boxplots of the QMLE, $t_5$-QMLE, RMLTSE, and R-estimators (sign test, Spearman, and van der Waerden scores) under   $t_3$ innovations; sample size $n = 1000$;  $N = 300$ replications.  The horizontal red line represents the actual parameter value.}
\begin{center}
\includegraphics[width=1\textwidth, height=0.55\textwidth]{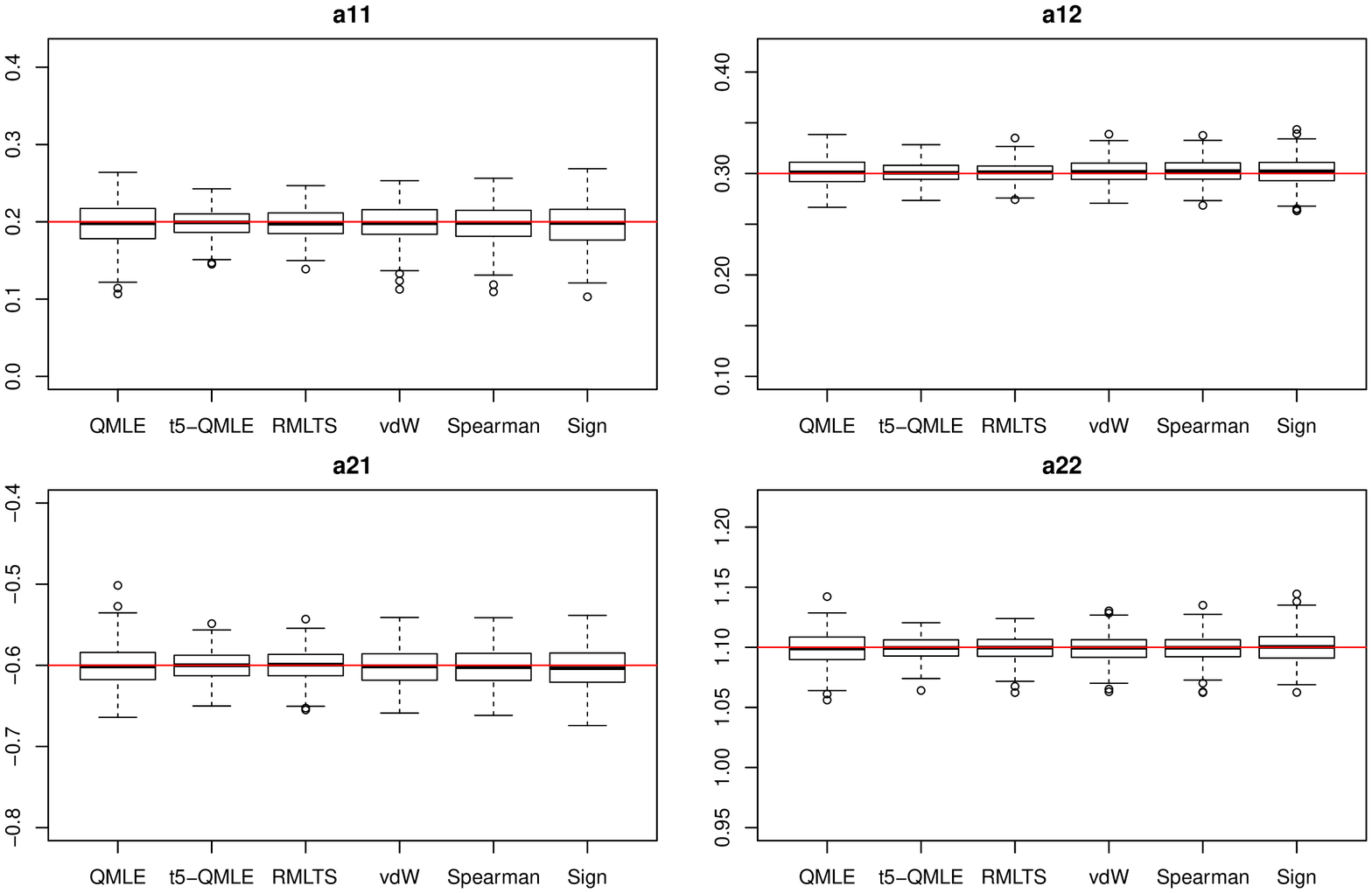} \vspace{-12mm}
\end{center}
\label{boxt3}
\end{figure}
\end{center}

\begin{center}
\begin{figure}[!htbp]
\caption{Boxplots of the QMLE, $t_5$-QMLE, RMLTSE, and R-estimators (sign test, Spearman, and van der Waerden scores) under  non-spherical Gaussian innovations; sample size $n = 1000$;  $N = 300$ replications.  The horizontal red line represents the actual parameter value.}
\begin{center}
\includegraphics[width=1\textwidth, height=0.55\textwidth]{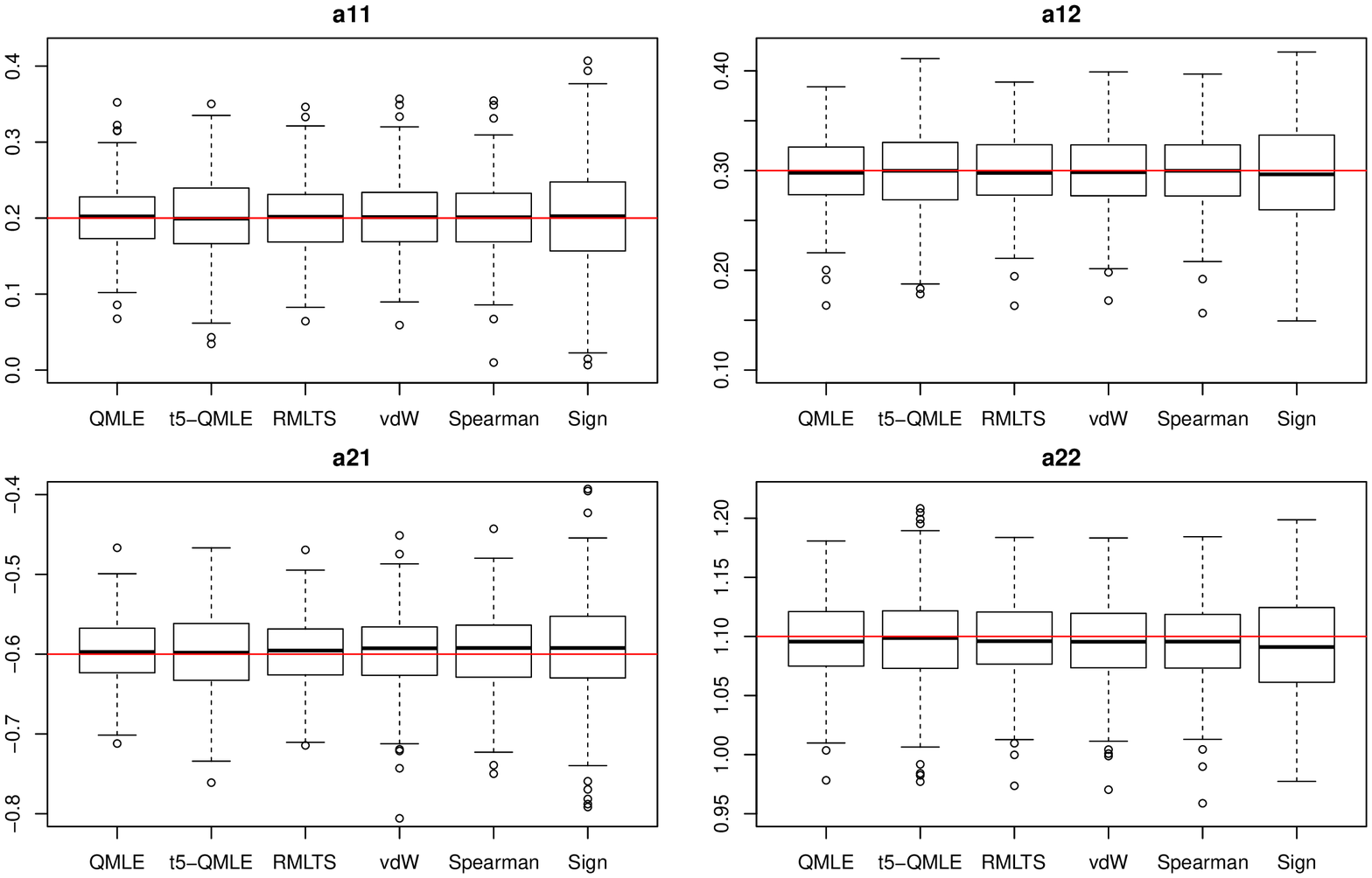} \vspace{-12mm}
\end{center}
\label{boxell}
\end{figure}
\end{center}

\subsubsection{Small sample and outliers}\label{AppD32}

For sample size $n = 300$, we display here, in   Figures~\ref{boxmixture300}, \ref{boxGaussiann300}, \ref{boxSkewNormaln300}, \ref{boxSkewt3n300}, and \ref{boxt3n300}, the boxplots of the QMLE, $t_5$-QMLE, RMLTSE, and R-estimators (sign test, Spearman, and van der Waerden scores) under the Gaussian mixture \eqref{Eq. Mixture}, spherical Gaussian, skew-normal, skew-$t_3$, and $t_3$, respectively. These pictures complement the boxplots available in Section~\ref{subsec.AO},  for  the additive outlier case.

All boxplots, as well as Table~\ref {Tab.n300} confirm the fact that, while doing equally well under spherical and Gaussian-tailed innovations, as the common practice QMLE,  R-estimation is resisting skewness, heavy tails, non-elliptical contours, and the presence of additive outliers, sometimes better even than the robust RMLTSE.

\begin{table}[!htbp]
\caption{The estimated bias ($\times 10^3$), MSE ($\times 10^3$), and overall MSE ratios of the QMLE,   $t_5$-QMLE,  RMLTSE, and R-estimators  under various innovation densities.   The sample size is $n = 300$; $N = 300$ replications.}\label{Tab.n300}
\centering
\small
\begin{tabular}{ccccc|cccc|c}
\hline
            & \multicolumn{4}{c}{Bias ($\times 10^3$)}           & \multicolumn{4}{c}{MSE  ($\times 10^3$)}  &  MSE ratio \\  \hline
                    & $a_{11}$          & $a_{21}$ & $a_{12}$ & $a_{22}$ & $a_{11}$ & $a_{21}$ & $a_{12}$ & $a_{22}$ &                   \\ \hline
(Normal)             &                    &          &        &        &        &        &       &       &        \\
QMLE                 & -7.208             & -2.006   & 2.639  & -0.870 & 2.624  & 2.715  & 0.543 & 0.733 &        \\
$t_5$-QMLE            & -8.352             & -2.065   & 3.701  & -1.071 & 2.783  & 2.796  & 0.580 & 0.751 & 0.957  \\
RMLTS                & -8.374             & -2.423   & 3.481  & -0.706 & 3.014  & 2.818  & 0.607 & 0.714 & 0.925  \\
vdW                  & 4.247              & -3.994   & -2.337 & 1.076  & 1.486  & 1.003  & 0.985 & 1.000 & 1.478  \\
Spearman             & 5.041              & -6.119   & -3.395 & 3.332  & 1.661  & 1.204  & 1.165 & 1.292 & 1.243  \\
Sign                 & 6.124              & -6.672   & -4.254 & 4.294  & 2.586  & 1.839  & 1.487 & 0.992 & 0.958  \\ \hline                    
(Mixture)            &                    &          &        &        &        &        &       &       &        \\
QMLE      & -3.430 & -0.123 & 4.399  & -1.814 & 2.751  & 0.550 & 1.000 & 0.213 &       \\
$t_5$-MLE & -1.593 & 0.240  & 5.467  & -1.277 & 12.295 & 0.918 & 4.129 & 0.461 & 0.254 \\
RMLTS     & -2.459 & -0.397 & 3.997  & -1.392 & 2.707  & 0.578 & 1.025 & 0.220 & 0.997 \\
vdW       & -2.484 & -0.007 & 5.065  & 1.348  & 1.427  & 0.368 & 0.733 & 0.379 & 1.554 \\
Spearman  & -2.632 & 0.742  & 5.160  & 1.104  & 1.329  & 0.379 & 0.694 & 0.332 & 1.652 \\
Sign      & -3.152 & -0.066 & 10.017 & 1.164  & 4.313  & 0.745 & 2.283 & 0.566 & 0.571   \\ \hline                    
\multicolumn{2}{l}{(Skew-normal)}       &          &        &        &        &        &       &       &        \\
QMLE                 & -9.045             & -7.223   & 5.870  & -2.116 & 3.564  & 3.308  & 1.087 & 1.022 &        \\
$t_5$-QMLE            & -7.788             & -7.028   & 6.400  & -1.115 & 4.581  & 3.992  & 1.518 & 1.327 & 0.787  \\
RMLTS                & -9.558             & -6.833   & 5.186  & -1.844 & 3.988  & 3.574  & 1.200 & 1.140 & 0.907  \\
vdW                  & -7.086             & -1.523   & 7.358  & -5.660 & 1.879  & 3.052  & 0.442 & 0.706 & 1.477  \\
Spearman             & -6.960             & -1.101   & 7.198  & -5.676 & 1.911  & 3.109  & 0.448 & 0.721 & 1.451  \\
Sign                 & -12.525            & 0.748    & 10.592 & -6.080 & 3.989  & 5.962  & 1.014 & 1.180 & 0.740  \\ \hline                    
\multicolumn{2}{l}{(Skew-$t_3$)}          &          &        &        &        &        &       &       &        \\
QMLE                 & -11.108            & -4.201   & 3.932  & -1.327 & 3.148  & 2.710  & 1.446 & 1.209 &        \\
$t_5$-QMLE            & 1.801              & 5.000    & 3.371  & -1.652 & 3.796  & 2.771  & 2.269 & 1.417 & 0.830  \\
RMLTS                & -3.378             & 0.428    & 4.358  & -1.058 & 1.918  & 1.780  & 1.129 & 0.833 & 1.504  \\
vdW                  & -7.152             & 0.232    & 6.544  & -3.750 & 1.718  & 2.320  & 0.634 & 1.240 & 1.440  \\
Spearman             & -5.594             & -1.927   & 6.402  & -2.279 & 1.719  & 2.388  & 0.625 & 1.365 & 1.396  \\
Sign                 & -3.380             & -1.968   & 6.469  & -0.033 & 4.816  & 4.863  & 1.900 & 2.054 & 0.624  \\ \hline                    
($t_3$)              &                    &          &        &        &        &        &       &       &        \\
QMLE                 & 0.168              & -0.844   & 2.047  & -1.063 & 2.279  & 2.593  & 0.647 & 0.658 &        \\
$t_5$-QMLE            & -2.189             & 0.647    & 1.176  & -1.347 & 1.160  & 1.215  & 0.339 & 0.343 & 2.021  \\
RMLTS                & -3.538             & 2.340    & 0.680  & -1.734 & 1.343  & 1.377  & 0.379 & 0.358 & 1.787  \\
vdW                  & -3.426             & -0.037   & 3.681  & -6.190 & 1.435  & 2.896  & 0.309 & 0.816 & 1.132  \\
Spearman             & -2.715             & 0.208    & 3.737  & -5.768 & 1.387  & 2.930  & 0.306 & 0.788 & 1.141  \\
Sign                 & -2.552             & 1.297    & 2.626  & -6.454 & 2.842  & 5.634  & 0.564 & 2.045 & 0.557  \\ \hline                    
\multicolumn{2}{l}{(Additive   outliers)} &          &        &        &        &        &       &       &        \\
QMLE                 & -154.990           & -149.720 & 15.327 & 10.173 & 27.667 & 24.982 & 1.021 & 1.080 &        \\
$t_5$-QMLE            & -110.645           & -105.918 & 12.836 & 7.714  & 15.310 & 13.590 & 0.859 & 1.049 & 1.777  \\
RMLTS                & -76.970            & -71.918  & 9.792  & 4.743  & 9.931  & 8.795  & 0.853 & 1.042 & 2.655  \\
vdW                  & -3.426             & -0.037   & 3.681  & -6.190 & 1.435  & 2.896  & 0.309 & 0.816 & 10.034 \\
Spearman             & -2.715             & 0.208    & 3.737  & -5.768 & 1.387  & 2.930  & 0.306 & 0.788 & 10.118 \\
Sign                 & -2.552             & 1.297    & 2.626  & -6.454 & 2.842  & 5.634  & 0.564 & 2.045 & 4.939 \\ \hline                    
 \end{tabular}
\end{table}


\begin{center}
\begin{figure}[!htbp]
\caption{Boxplots of the QMLE, $t_5$-QMLE,  RMLTSE, and R-estimators (sign test, Spearman, and van der Waerden scores) under Gaussian mixture  (sample size $n = 300$;  $N = 300$ replications). The horizontal red line represents the actual parameter value.}
\begin{center}
\includegraphics[width=1\textwidth, height=0.5\textwidth]{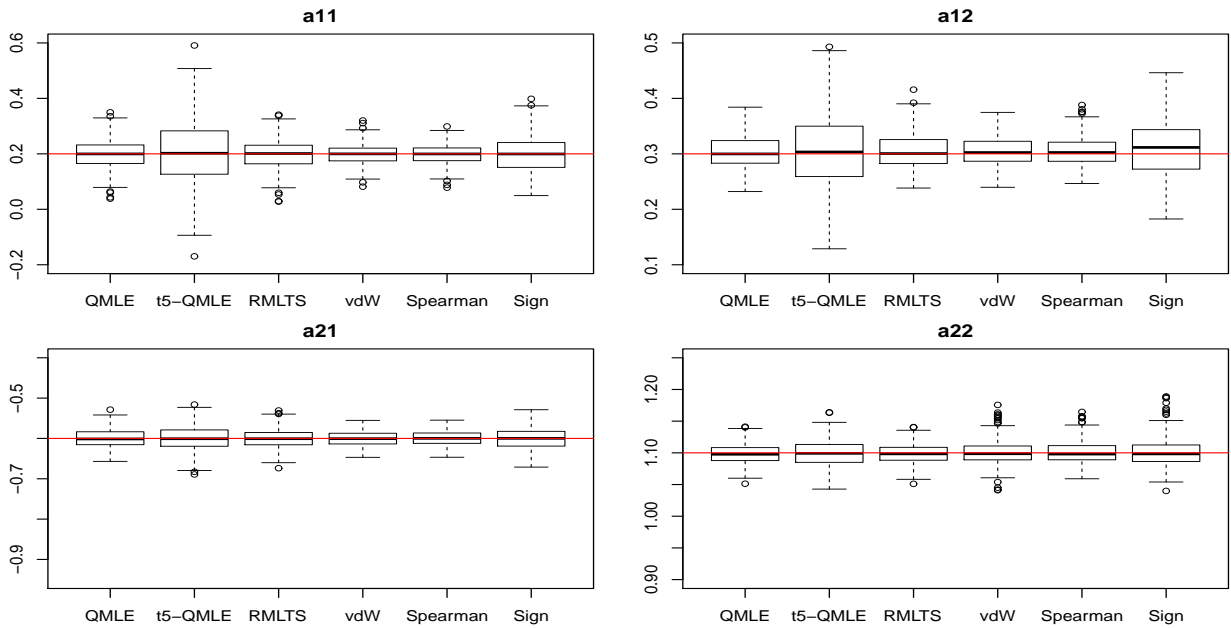}
\end{center}
\label{boxmixture300}
\end{figure}
\end{center}

\begin{center}
\begin{figure}[!htbp]
\caption{Boxplots of the QMLE, $t_5$-QMLE,  RMLTSE, and R-estimators (sign test, Spearman, and van der Waerden scores) under  spherical Gaussian innovations; sample size $n = 300$;  $N = 300$ replications.  The horizontal red line represents the actual parameter value.}
\begin{center}
\includegraphics[width=1\textwidth, height=0.5\textwidth]{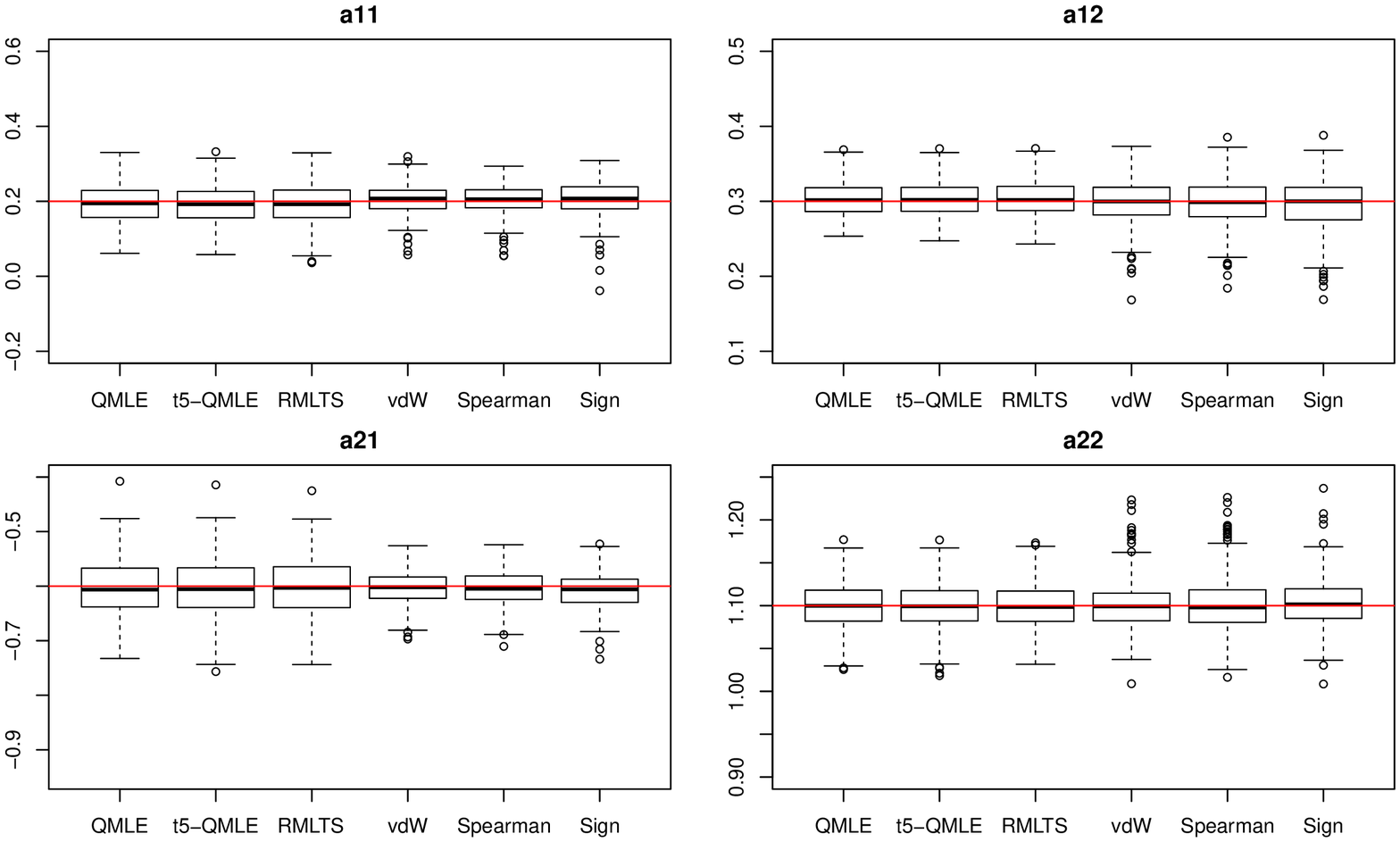} \vspace{-12mm}
\end{center}
\label{boxGaussiann300}
\end{figure}
\end{center}

\begin{center}
\begin{figure}[!htbp]
\caption{Boxplots of the QMLE, $t_5$-QMLE,  RMLTSE, and R-estimators (sign test, Spearman, and van der Waerden scores) under skew-normal innovations   (\ref{SN.density}); sample size $n = 300$;  $N = 300$ replications.  The horizontal red line represents the actual parameter value.}
\begin{center}
\includegraphics[width=1\textwidth, height=0.5\textwidth]{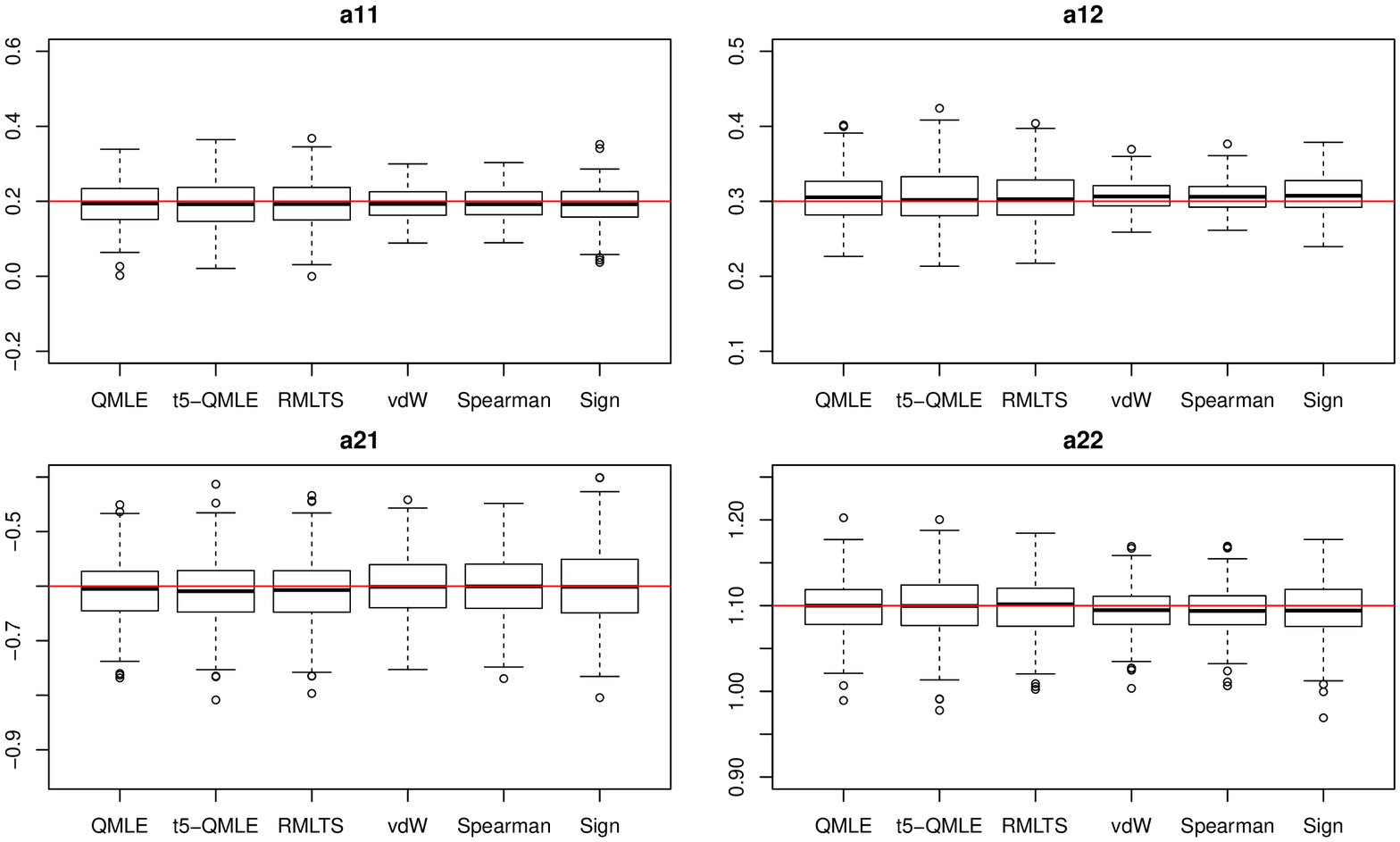} \vspace{-12mm}
\end{center}
\label{boxSkewNormaln300}
\end{figure}
\end{center}

\begin{center}
\begin{figure}[!htbp]
\caption{Boxplots of the QMLE, $t_5$-QMLE,  RMLTSE, and R-estimators (sign test, Spearman, and van der Waerden scores) under skew-$t_3$ innovations    (\ref{St.density}); sample size $n = 300$;  $N = 300$ replications.  The horizontal red line represents the actual parameter value.}
\begin{center}
\includegraphics[width=1\textwidth, height=0.5\textwidth]{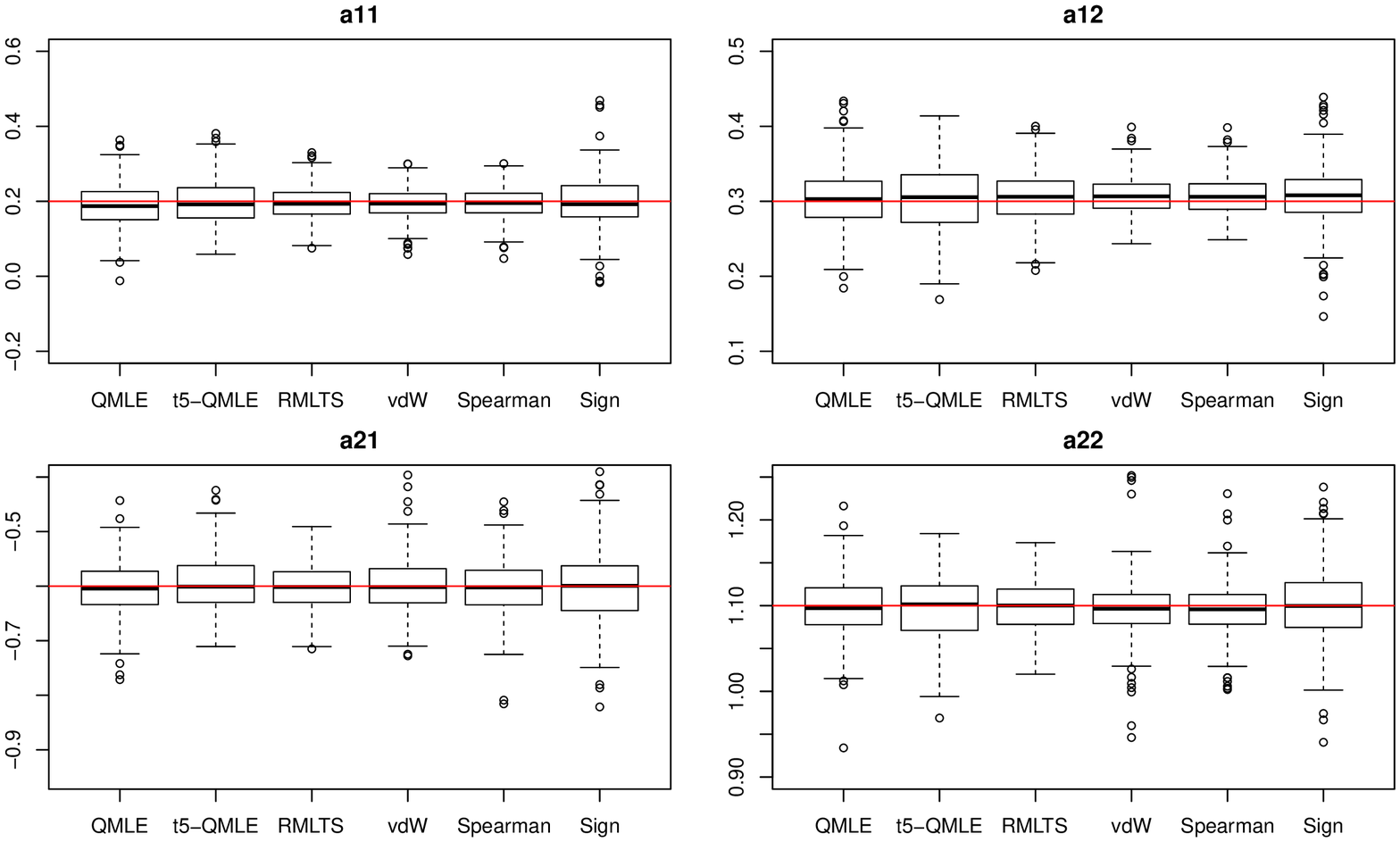} \vspace{-12mm}
\end{center}
\label{boxSkewt3n300}
\end{figure}
\end{center}

\begin{center}
\begin{figure}[!htbp]
\caption{Boxplots of the QMLE, $t_5$-QMLE,  RMLTSE, and R-estimators (sign test, Spearman, and van der Waerden scores) under spherical $t_3$ innovations; sample size $n = 300$;  $N = 300$ replications.  The horizontal red line represents the actual parameter value.}
\begin{center}
\includegraphics[width=1\textwidth, height=0.5\textwidth]{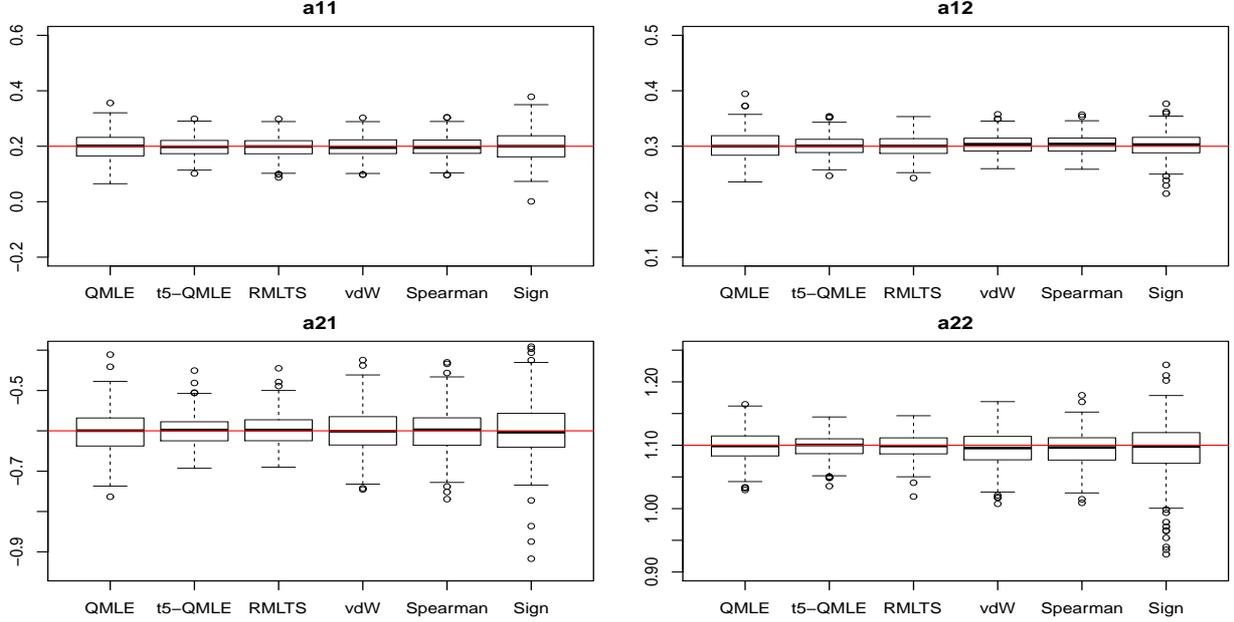} 
\end{center}
\label{boxt3n300}
\end{figure}
\end{center}


\subsection{Higher dimension}\label{dim3App}

Due to the rapid growth of their number of parameters, VARMA models are not meant for the analysis of high-dimensional time series (where different approaches are in order---see, e.g., Hallin et al.~(2020c)).  One may wonder, however, whether the attractive properties of R-estimators extend beyond the bivariate context. We therefore provide here some numerical results in dimension $d=3$. 

Consider the three-dimensional VAR($1$) model\vspace{-2mm}
\begin{equation*}\label{VAR3}
\left(\I_3 - \A L \right) \X_t =  \bepsilon_t, \quad t \in \Z,\vspace{-2mm}
\end{equation*}
with $\bth^\prime := \text{vec}^\prime(\A)= (0.55, 0.2, 0.13, -0.2, 0.5, -0.1, 0.1, 0.11, 0.6)$ satisfying  
Assumption~(A1). 
We are limiting our investigation to two selected innovation densities: the spherical  three-dimensional  Gaussian  and the Gaussian mixture 
\begin{equation}
\frac{3}{8}   {\cal N}(\bmu_1, \bSigma_1) + \frac{3}{8}   {\cal N}(\bmu_2, \bSigma_2) + \frac{1}{4} {\cal N}(\bmu_3, \bSigma_3),
\label{Eq. Mixture2}
\vspace{-2mm}\end{equation} 
with $$\bmu_1 = (-5, -5, 0)^\prime,\ \bmu_2 = (5, 5, 2)^\prime, \ \bmu_3 = (0, 0, -3)^\prime$$ 
and 
 $$\bSigma_1 = 
\begin{bmatrix}
7 & 3 & 5 \\
3 & 6 & 1 \\
5 & 1 & 7
\end{bmatrix}, \  
\bSigma_2 = 
\begin{bmatrix}
7 & -5 & -3 \\
-5 & 7 & 4 \\
-3 & 4 & 5 
\end{bmatrix}\text{, and 
 }\bSigma_3 = 
\begin{bmatrix}
4 & 0 & 0\\
0 & 3  & 0 \\
0 & 0 & 1
\end{bmatrix}.
$$
\begin{figure}[h!]
\caption{Boxplots of the QMLE  and R-estimator (van der Waerden scores) under the Gaussian mixture innovation density \eqref{Eq. Mixture2}  for $d=3$; sample size $n = 1000$;  $N = 300$ replications. In each panel, the MSE ratio of the QMLE with respect to the R-estimator is reported. The horizontal red line represents the actual parameter value.}
\begin{center}
\includegraphics[width=1.05\textwidth, height=1.1\textwidth]{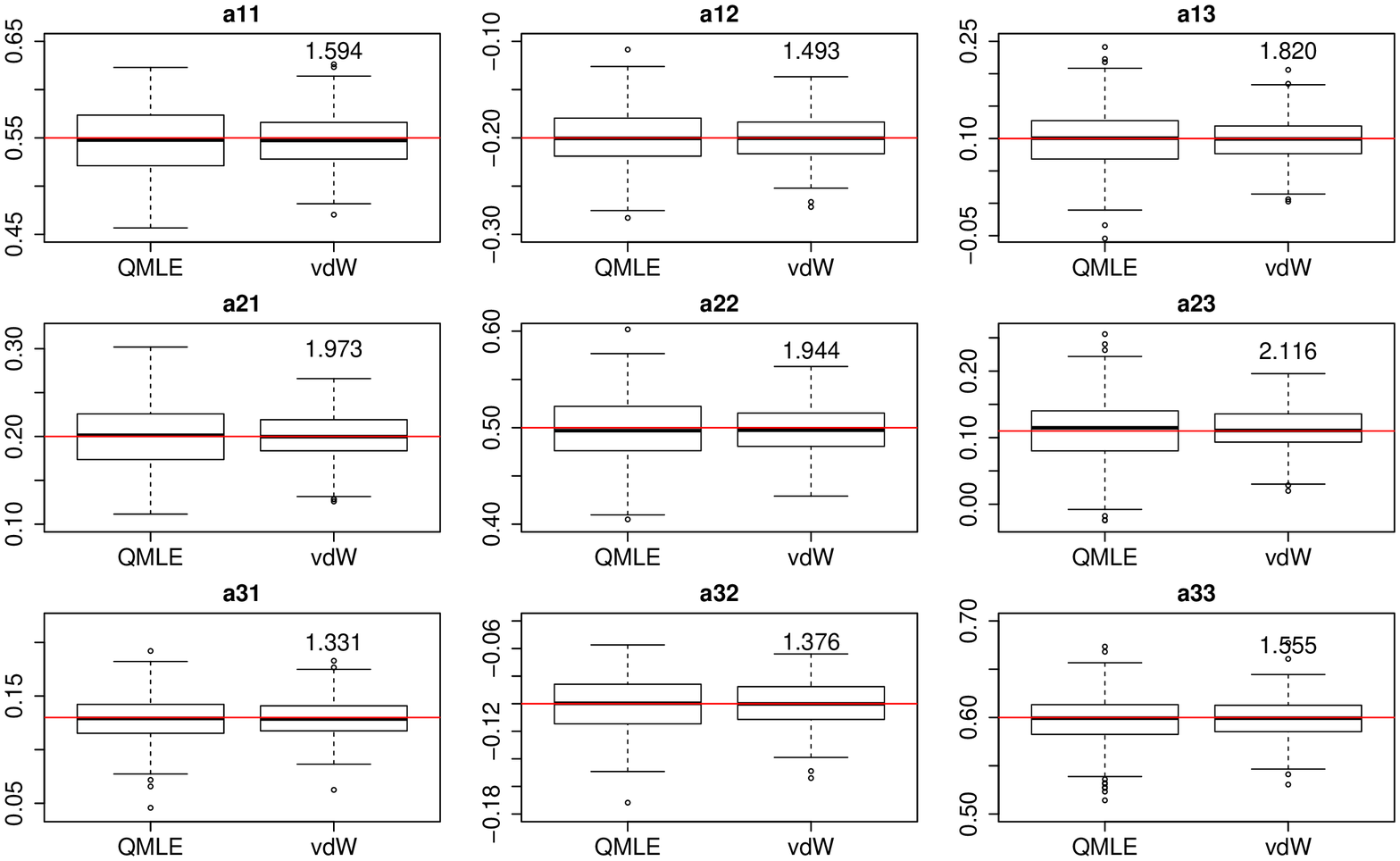}\vspace{-5mm}
\end{center}
\label{d3}
\end{figure}

For the computation of the center-outward ranks and signs, we used the algorithm described in Appendix~\ref{algsec} with $n_R = 15, n_S = 66, n_0 = 10$. For  numerical implementation, we generated regular grids on the sphere via the  routine {\tt UnitSphere} in R package {\tt mvmesh}, where we refer to for details. The boxplots for the Gaussian mixture and spherical Gaussian innovations are displayed in Figures~\ref{d3} and~\ref{3dNormal}, respectively. Inspection of Figures~\ref{d3} and~\ref{3dNormal} yields the same conclusions as in the bivariate motivating example (Figures~\ref{box3Mix}).

\begin{figure}[h!]
\caption{Boxplots of the QMLE and R-estimator (van der Waerden scores) under  spherical Gaussian for $d=3$; sample size $n = 1000$;  $N = 300$ replications. In each panel, the MSE ratio of the QMLE with respect to the R-estimator is reported. The horizontal red line represents the actual parameter value.}
\begin{center}
\includegraphics[width=1.05\textwidth, height=1.1\textwidth]{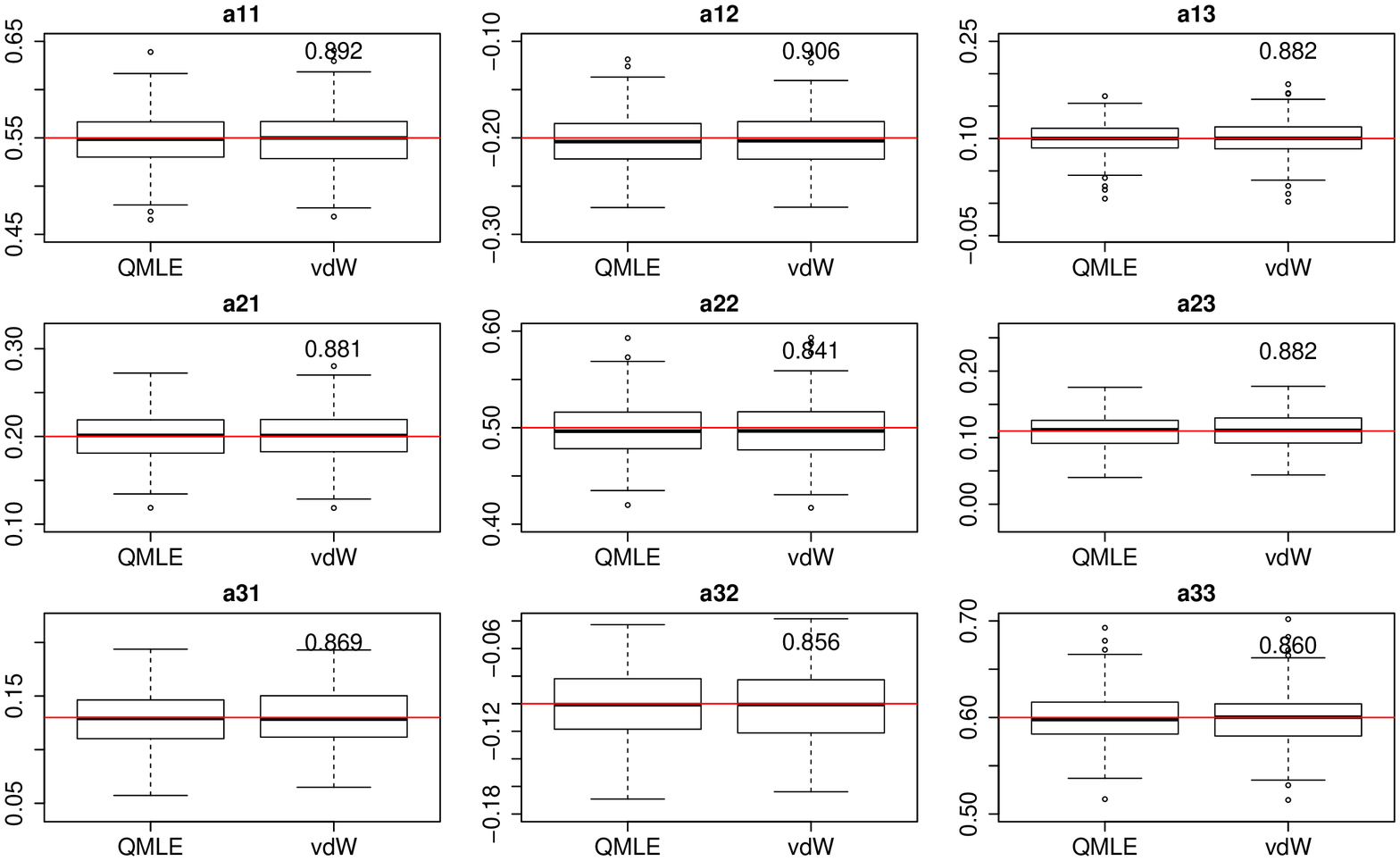}\vspace{-10mm}
\end{center}
\label{3dNormal}
\end{figure}

%

\section{Supplementary results for the real-data example }\label{tablesec} 

\subsection{Estimates for the VARMA(3,1) model}
To complement the real data example of Section~\ref{empirsec}, we provide here the table of estimated coefficients for the macroeconomic time series.

\begin{table}[h!]
\caption{The QMLE and R-estimates of $\bth$ in the VARMA($3, 1$) fitting of the econometric data (demeaned differenced Hstarts and Mortg series);  standard errors  are shown in parentheses. The datasets are demeaned  differenced Hstarts and Mortg series.\vspace{3mm} }\label{emp}
\begin{tabular}{cccccccccccc}
\hline
         & \multicolumn{2}{c}{$\A_1$} &  & \multicolumn{2}{c}{$\A_2$}       &  & \multicolumn{2}{c}{$\A_3$} &  & \multicolumn{2}{c}{$\B_1$} \\ \hline
QMLE     & 0.137 & 0.487 && -0.154 & -0.199 & &0.032 & 0.056  &&  -0.703 & -0.490 \\
         & (0.265) &  (0.353)  && (0.284)  & (0.130)  &&  (0.171) & (0.072)  & & (0.258)  & (0.350)  \\
         & 0.596 & 0.974 && 0.030  & -0.400 && 0.070 & 0.110  & &-0.152 & -0.636 \\
         & (0.327) & (0.537) && (0.436)  & (0.189)  & & (0.285) & (0.077)  & & (0.282)  & (0.533)  \\ \hline
vdW      & 0.155 & 0.526 & & -0.096 & -0.181 &&  0.017 & 0.038  &&  -0.705 & -0.527 \\
         & (0.141) & (0.088) & & (0.122)  & (0.079)  & & (0.133) & (0.062)  &  & (0.088)  & (0.071)  \\
         & 0.561 & 0.943 & & 0.094  & -0.386 & & 0.011 & 0.128  & & -0.161 & -0.627 \\
         & (0.148) & (0.079) & & (0.133)  & (0.100)  && (0.098) & (0.040)  & & (0.081)  & (0.015)  \\ \hline
Sign     & 0.087 & 0.536 & & -0.032 & -0.198 & & 0.075 & -0.044 & & -0.705 & -0.562 \\
         & (0.148) & (0.079) & & (0.133)  & (0.100)  & & (0.098) & (0.040)  & & (0.081)  & (0.015)  \\
         & 0.471 & 1.036 & & 0.107  & -0.403 & & 0.035 & 0.148  & & -0.161 & -0.627 \\
         & (0.178) & (0.084) & & (0.165)  & (0.073)  & & (0.138) & (0.061)  & & ($< 10^{-3}$)  & ($< 10^{-3}$)  \\ \hline
Spearman & 0.180 & 0.511 & & -0.090 & -0.180 & & 0.030 & 0.049  &&  -0.705 & -0.537 \\
         & (0.066) & (0.033) & & (0.092)  & (0.046)  & & (0.113) & (0.049)  &&  ($< 10^{-3}$)  & (0.014)  \\
         & 0.531 & 0.946 & & 0.072  & -0.374 & & 0.011 & 0.121  & & -0.161 & -0.627 \\
         & (0.124) & (0.054) &&  (0.115)  & (0.075)  &  & (0.112) & (0.042)  & & ($< 10^{-3}$)  & ($< 10^{-3}$)  \\ \hline
\end{tabular}
\end{table}

\subsection{Impulse response function: a compendium}

As explained in Section~\ref{empirsec}, impulse response functions provide a  convenient way of exploring the relation between the components of multiple time series. In particular, it is used to study the impact of changes in one variable on its own future values and those of other time series. For the $d$-dimensional VARMA($p, q$) model in (\ref{VARMA_mod}), the impulse response function can be obtained as follows.

Write  (\ref{VARMA_mod}) under  the corresponding VMA($\infty$) form 
$$\X_t = \W(L)  \bepsilon_t, \quad t \in \mathbb{Z},$$
where $$\W(L):= \sum_{l=0}^{\infty} \W_l L^l = \left(\I_d - \sum_{i = 1}^p \A_i L^i \right)^{-1}  \left(\I_d + \sum_{j = 1}^q \B_j L^j\right) \bepsilon_t$$
with $\W_l$ being the coefficient at lag $l$. 

Now, suppose that we are interested in studying the impact on $\X_{t+h}$, $h\geq 0$ of increasing the value at time $t$ of the  $k$th series $X_{kt}$, $1\leq k \leq d$ by one unit. Without loss of generality, we can assume  $t=0$. Setting $\X_t = \0$ for $t \leq 0$, $\bepsilon_0 = \e_k$ and $\bepsilon_t = \0$ for $t > 0$, where  $\e_k$ denotes the $k$th unit vector in the canonical basis of $\mathbb{R}^d$, we then have
$$\X_0 = \bepsilon_0 = \e_k, \quad \X_1 = \W_1 \bepsilon_0 = \W_{1, k}, \quad \X_2 = \W_2 \bepsilon_0 = \W_{2, k},\quad ...,$$
where $\W_{l, k}$ denotes the $k$th column of $\W_l$. Therefore, the impact under study 
 is reflected in the~$k$th column of the coefficient matrix $\W_{h}$. For this reason, the coefficient matrices~$\{\W_{h, k}; h \geq 0\}$ are referred to as the coefficients of impulse response functions; see Tsay~(2014, Chapter 2 and 3) for further  details.

\end{document}